%% file: tilting3.tex
\documentclass{amsart}

\input{decls}

\tikzset{lab/.style={auto,font=\scriptsize}} 
\usetikzlibrary{arrows}

\usepackage{xcolor}
\usepackage{fixme}
\fxsetup{
    status=draft,
    author=,
    layout=margin,
    theme=color
}

\definecolor{fxnote}{rgb}{1.0000,0.0000,0.0000}
\colorlet{fxnotebg}{yellow}

\makeatletter
\renewcommand*\FXLayoutInline[3]{%
  \@fxdocolon {#3}{%
    \@fxuseface {inline}%
    \colorbox{fx#1bg}{\color {fx#1}\ignorespaces #3\@fxcolon #2}}}
\makeatother

\UseAllTwocells

\newcommand{\homr}[3]{#2\mathrel{\rhd_{[#1]}} #3}
\newcommand{\homl}[3]{#3 \mathrel{\lhd_{[#1]}} #2}

\newcommand{\tw}{\ensuremath{\operatorname{tw}}}
\newcommand{\D}{\sD}
\newcommand{\E}{\sE}

\newcommand{\V}{\sV}
\newcommand{\W}{\sW}
\newcommand{\Sp}{\mathcal{S}p}

\autodefs{\cDER\cCat\cGrp\cCAT\cMONCAT\bSet\cProf\bCat
\cPro\cMONDER\cMonCAT\cBICAT\ncomp\nid\ncoker\niso\nIm\sProf\ncyl\fC\fZ\fT\nhocolim\nholim\cPsNat\ncoll\bsSet\cAlg\cIm\cI\cP}

\def\cPDER{\ensuremath{\mathcal{PD}\mathit{ER}}\xspace}

\def\ho{\mathscr{H}\!\mathit{o}\xspace}

\let\oldboxtimes\boxtimes
\def\boxtimes{\mathrel{\oldboxtimes}}

\newcommand{\fib}{\mathsf{fib}}
\newcommand{\cof}{\mathsf{cof}}

\def\cSp{\ensuremath{\mathcal{S}\!\mathit{p}}\xspace}

\newcommand{\sse}{\stackrel{\mathrm{s}}{\sim}}
\newcommand{\exx}{\mathrm{ex}}

\def\ccsub{_{\mathrm{cc}}}
\def\pdh(#1,#2){\llbracket #1,#2\rrbracket}
\def\ldh(#1,#2){\llbracket #1,#2\rrbracket\ccsub}
\def\pend(#1){\pdh(#1,#1)}
\def\lend(#1){\ldh(#1,#1)}

\def\shift#1#2{{#1}^{#2}}

\def\DTl#1#2#3#4#5#6#7{%
  \xymatrix@C=3pc{{#1} \ar[r]^-{#2} &
    {#3} \ar[r]^-{#4} &
    {#5} \ar[r]^-{#6} &
    {#7}
  }}

\newsavebox{\tvabox}
\savebox\tvabox{\hspace{1mm}\begin{tikzpicture}[>=latex',baseline={(0,-.18)}]
  \draw[->] (0,.1) -- +(1,0);
  \node at (.5,0) {$\scriptscriptstyle\bot$};
  \draw[->] (1,-.1) -- +(-1,0);
  \draw[->] (1,-.2) -- +(-1,0);
\end{tikzpicture}\hspace{1mm}}

\renewcommand{\ex}{\mathrm{ex}}
\newcommand{\Ch}{\mathrm{Ch}}
\newcommand{\Mod}{\mathrm{Mod}}

\newcommand{\ARQ}{\mathrm{AR}_Q}
\newcommand{\ARQinv}{\mathrm{AR}\inv_Q}
\DeclareMathOperator{\Pic}{Pic}

\newcommand{\ind}{\mathsf{ind}}

\newcommand{\A}[1]{{\Vec{A}_{#1}}}

\title{Abstract representation theory of Dynkin quivers of type~$A$}

\author{Moritz Groth}
\address{Radboud University, Heyendaalseweg 135, 6525 AJ Nijmegen, Netherlands}
\email{mgroth@math.ru.nl}

\author{Jan \v{S}\v{t}ov\'{\i}\v{c}ek}
\address{Department of Algebra, Charles University in Prague, Sokolovska 83, 186 75 Praha~8, Czech Republic}
\email{stovicek@karlin.mff.cuni.cz}

\date{\today}

\begin{document}

\begin{abstract}
We study the representation theory of Dynkin quivers of type~$A$ in abstract stable homotopy theories, including those associated to fields, rings, schemes, differential-graded algebras, and ring spectra. Reflection functors, (partial) Coxeter functors, and Serre functors are defined in this generality and these equivalences are shown to be induced by universal tilting modules, certain explicitly constructed spectral bimodules. In fact, these universal tilting modules are spectral refinements of classical tilting complexes. As a consequence we obtain split epimorphisms from the spectral Picard groupoid to derived Picard groupoids over arbitrary fields.

These results are consequences of a more general calculus of spectral bimodules and admissible morphisms of stable derivators. As further applications of this calculus we obtain examples of universal tilting modules which are new even in the context of representations over a field. This includes Yoneda bimodules on mesh categories which encode all the other universal tilting modules and which lead to a spectral Serre duality result. 

Finally, using abstract representation theory of linearly oriented $A_n$-quivers, we construct canonical higher triangulations in stable derivators and hence, a posteriori, in stable model categories and stable $\infty$-categories.
\end{abstract}

\maketitle

\setcounter{tocdepth}{1}

\tableofcontents

\section{Introduction}
\label{sec:intro}

The representation theory of Dynkin quivers of type $A$ (or of finite zigzags in a more topological parlance) over a field is well-understood. This includes a fairly detailed understanding of the derived categories of the corresponding path algebras and also of important related functors. For example, it is well known that two Dynkin quivers of type $A$ are derived equivalent over a field if and only if they have the same number of vertices, and that reflection functors in the sense of Bern{\v{s}}te{\u\i}n, Gel$'$fand, and Ponomarev \cite{bernstein-gelfand-ponomarev:Coxeter} induce such derived equivalences. 
Happel showed that these derived equivalences exist and are given by derived tensor product and derived hom functors associated to certain chain complexes of bimodules, so-called \emph{tilting complexes} (see~e.g.~\cite{happel:fd-algebra}). Further interesting equivalences between associated derived categories are given by (partial) Coxeter functors, Serre functors, and derived Nakayama functors. And also these derived equivalences are standard equivalences induced by tilting complexes.

The current understanding of the representation theory of Dynkin quivers of type~$A$ over the integers or over arbitrary commutative rings is less deep. And things get worse if we pass to more complicated contexts, like representations in quasi-coherent $\mathcal{O}_X$-modules over a scheme~$X$, over a differential-graded algebra or over a ring spectrum. The main aim of this paper is to contribute to the understanding of this and, more generally, to the understanding of the representation theory of $A_n$-quivers in arbitrary \emph{abstract stable homotopy theories}. In particular, we extend the definitions of the above-mentioned functors to these more general contexts and show that they are realized by certain spectral bimodules which we refer to as \emph{universal tilting modules}. The terminology alludes to the fact that these spectral bimodules realize these functors in arbitrary abstract stable homotopy theories. Moreover, these bimodules are spectral refinements of the classical tilting complexes which are recovered if we induce up from spectra to chain complexes over a field.

In this paper by an abstract stable homotopy theory we mean a \emph{stable derivator}. Recall for example from \cite{groth:ptstab} that a derivator is some kind of a minimal extension of a classical derived category or homotopy category to a framework with a well-behaved calculus of homotopy limits and homotopy colimits. Slightly more formally, a derivator consists of homotopy categories of diagram categories which are related by various types of constructions. A derivator is \emph{stable} if it admits a zero object and if homotopy pushouts and homotopy pullbacks coincide. Typical examples of stable derivators are the derivator $\D_R$ of a ring~$R$, the derivator $\D_X$ of a scheme~$X$, the derivator $\D_A$ of a differential-graded algebra~$A$, and the derivator $\D_E$ of a symmetric ring spectrum~$E$ (see \cite[\S 5]{gst:basic} for many additional examples). 

Stable derivators provide an enhancement of triangulated categories; for example the underlying category of $\D_R$ is the more classical derived category $D(R)$ and the classical triangulation on it can be constructed from properties of the stable derivator $\D_R$. The following simple observation allows us to apply this enhancement in representation theory: given a stable derivator \D and a small category~$A$ there is the stable derivator $\D^A$ of coherent $A$-shaped diagrams in \D. This shifting operation can be applied to quivers by passing to the corresponding free categories. For example, for a ring $R$ and a quiver $Q$ there is an equivalence of stable derivators
\[
\D^Q_R\simeq\D_{RQ}
\]
where $RQ$ denotes the associated path algebra.

Here we restrict attention to Dynkin quivers of type~$A$. A good deal of information about the corresponding derived categories over a field is classically encoded by Auslander--Reiten quivers. We show that stable derivators of abstract representations of $A_n$-quivers are equivalent to derivators of coherent Auslander--Reiten quivers, i.e., derivators of certain representations of mesh categories (\autoref{thm:AR-independent}). These latter derivators allow us to conveniently encode reflection functors (see also \cite{gst:basic} and \cite{gst:tree}), (partial) Coxeter functors, and Serre functors. Along the way we establish an abstract fractionally Calabi--Yau property and give an explanation of it in down-to-earth terms.

The universal tilting modules realizing these functors are obtained as special cases of a more general theorem. A non-trivial result says that every stable derivator is a closed module over the derivator of spectra \cite{cisinski-tabuada:non-connective}. Consequently, given a spectral bimodule we obtain induced tensor product and hom functors in arbitrary stable derivators. We introduce left admissible, right admissible, and admissible morphisms of stable derivators and the more general theorem guarantees that such morphisms are induced by \emph{explicitly constructed} spectral bimodules (\autoref{thm:kernel}). This theorem is part of a calculus of spectral bimodules, which is formally very similar to the usual calculus of bimodules in algebra and the associated derived tensor product and derived hom functors.

As special cases of such spectral bimodules we obtain in \S\ref{sec:universal-tilting} universal tilting modules for reflection functors, (partial) Coxeter functors, and Serre functors yielding the desired spectral refinements of classical tilting complexes. Generalizing a classical fact, we show that Serre functors are naturally isomorphic to canonical Nakayama functors. One way of summarizing some of these results is by saying that for Dynkin quivers~$Q$ of type~$A$ and for any field~$k$, the group homomorphism $\Pic_{\cSp}(Q)\to\Pic_{\D_k}(Q)$ from the spectral Picard group (\autoref{defn:picard}) to the derived Picard group introduced in \cite{miyachi-yekutieli} is a split epimorphism (\autoref{thm:picard}). And there is a similar such statement for Picard groupoids if one considers all $A_n$-quiver, $n\in\lN,$ at once (\autoref{rmk:picard}). The fact that these epimorphisms are split is a consequence of the abstract fractionally Calabi--Yau property. To illustrate how explicit the spectral bimodules are, we also revisit an example of \cite{gst:basic} and compute the spectral bimodule realizing a strong stable equivalence between the commutative square and $D_4$-quivers. 

While the spectral bimodules described so far are refinements of classical tilting complexes, in \S\ref{sec:yoneda} we also obtain examples of universal tilting modules which are new even in the context of representations over a field. For example, there are universal constructors for coherent Auslander--Reiten quivers. Moreover, for Dynkin quivers of type~$A$  it turns out that for each $n\geq 1$ there is one general universal tilting module~$U_n$ which encodes all the remaining ones for arbitrarily oriented~$A_n$-quivers. This \emph{Yoneda bimodule}~$U_n$ is self-dual up to a twist by the Serre functor --- a result which can be read as a spectral Serre duality theorem for the mesh category (\autoref{thm:U_n-Serre}). 

From a more conceptual perspective, stable derivators are derivators enriched over the monoidal derivator of spectra. In the language of enriched derivator theory, left admissible, right admissible, and admissible morphisms turn out to be instances of weighted homotopy (co)limit constructions and the universal tilting modules are the corresponding weights. The weights themselves are organized in a bicategory which --- together with the basic theory of monoidal derivators --- was established in \cite{gps:additivity}, yielding a convenient calculus which we apply in this paper. Enriched derivators and the calculus of weighted homotopy (co)limits will be studied more systematically in~\cite{gs:enriched}.

We conclude this paper by a section on \emph{higher triangulations} (which can be read immediately after \S\ref{sec:coxeter-reflection}). Although triangulated categories are very useful in a broad variety of situations, from the very beginning on it was obvious that the axioms come with certain defects (see for example already the introduction to \cite{heller:shc}). Besides the enhancements in the sense of stable model categories, stable $\infty$-categories or differential-graded categories, there are also more traditional attempts to fix some of the deficiencies of triangulated categories. Among these are the higher triangulations in the sense of Maltsiniotis \cite{maltsiniotis:higher} (and the closely related notion of Balmer \cite{balmer:separability}). We show as \autoref{thm:higher} that strong stable derivators give rise to \emph{canonical higher triangulations} in the sense of Maltsiniotis (see also \cite{maltsiniotis:higher}). Hence, a posteriori, this also applies to stable model categories and stable $\infty$-categories, illustrating the slogan that these enhancements encode `all the triangulated information'. Considered from this perspective, the universal constructors for coherent Auslander--Reiten quivers specialize to universal constructors for higher triangles.

This paper is part of a larger program aiming for the development of a formal stable calculus or abstract representation theory which applies to the typical contexts arising in algebra, geometry, and topology. First steps of this program are carried out in \cite{groth:ptstab} and \cite{gps:mayer}. While \cite{gps:additivity} and \cite[\S\S 9-10]{gst:basic} are formal studies of the interaction of stability and monoidal structure, the papers \cite{gst:basic} and \cite{gst:tree} mainly provide first examples of strong stable equivalences which is to say abstract tilting results. In this paper the techniques of those papers are combined in order to develop aspects of what we think of as abstract representation theory of $A_n$-quivers. We plan to continue the development of this abstract representation theory elsewhere, including abstract tilting results for posets \cite{gst:poset}. In \cite{gst:acyclic} we consider acyclic quivers and, in particular, show that the reflection morphisms for trees from \cite{gst:tree} are induced by spectral bimodules. We also aim for some abstract representation theory of more complicated algebras.

The content of the respective sections is as follows. In \S\S\ref{sec:review}-\ref{sec:stable} we recall some basics about derivators with particular focus on stable derivators. In \S\S\ref{sec:An}-\ref{sec:coxeter-reflection} we define derivators of coherent Auslander--Reiten quivers as well as (partial) Coxeter functors and Serre functors associated to $A_n$-quivers. Here we also establish the abstract fractionally Calabi--Yau property. In \S\ref{sec:monoidal} we recall some basics concerning monoidal derivators and the calculus of tensor products of functors. We prove in \S\ref{sec:nakayama} that Serre functors are naturally isomorphic to canonical Nakayama functors and are hence instances of tensor products with spectral bimodules. In \S\ref{sec:admissible} we define left admissible, right admissible, and admissible morphisms and construct associated spectral bimodules, yielding large classes of morphisms induced by spectral bimodules. In \S\ref{sec:kernel} we establish a few results related to this calculus of bimodules and the associated tensor and hom functors. These techniques are illustrated in \S\ref{sec:universal-tilting} where we construct universal tilting modules realizing reflection functors, (partial) Coxeter functors, and Serre functors. In \S\ref{sec:yoneda} we obtain further such bimodules, namely universal constructors for Auslander--Reiten quivers and the Yoneda bimodules on the mesh category, leading to a spectral Serre duality result. In \S\ref{sec:field} we relate our results to the more classical context of representations over a field and reformulate some of our results in terms of Picard groupoids. Finally, in \S\ref{sec:higher} we construct canonical higher triangulations on strong stable derivators and conclude with some philosophical comments in \S\ref{sec:cone}.

\section{Review of derivators}
\label{sec:review}

In this section we briefly review some basics of the theory of derivators. This is mainly to fix some notation and to recall a few facts of constant use in later sections. Derivators were introduced independently by Heller \cite{heller:htpythies}, Grothendieck~\cite{grothendieck:derivators}, and  Franke~\cite{franke:adams}. They were recently studied further by Maltsiniotis \cite{maltsiniotis:intro}, Cisinski \cite{cisinski:direct} and others. For more details on the content of this and the following section we refer to \cite{groth:ptstab,gps:mayer} and the many references therein.

We denote by \cCat the 2-category of small categories and by \cCAT the 2-category of not necessarily small categories. The terminal category $\bbone$ consists of one object and its identity morphism only, and for $A\in\cCat$ there is a natural isomorphism between $A$ and the functor category $A^\bbone$. We abuse notation and also write $a\colon\bbone\to A$ for the functor corresponding to an object $a\in A$.

A \textbf{prederivator} is simply a 2-functor $\D\colon\cCat\op\to\cCAT.$\footnote{We follow Heller~\cite{heller:htpythies} and Franke~\cite{franke:adams} and base the notion of a derivator on \emph{diagrams}. Note that there is the alternative, but isomorphic approach using \emph{presheaves} instead; see \cite{grothendieck:derivators,cisinski:direct}.} We define \textbf{morphisms} of prederivators to be pseudo-natural transformations while \textbf{transformations} of prederivators are modifications, yielding the 2-category $ \cPDER$ of prederivators (see \cite{borceux:1} for this 2-categorical language).

Given a prederivator \D, the category $\D(A)$ is the category of \textbf{coherent $A$-shaped diagrams} in \D. If $u\colon A\to B$ is a functor, then we denote the \textbf{restriction functor} by $u^*\colon\D(B) \to \D(A)$. In the case of a functor $a\colon\bbone\to A$ classifying an object, $a^\ast\colon\D(A)\to\D(\bbone)$ is an \textbf{evaluation functor}, taking values in the \textbf{underlying category} $\D(\bbone)$. If $f\colon X\to Y$ is a morphism in $\D(A)$, then its image under $a^\ast$ is denoted by $f_a\colon X_a\to Y_a$. 

These evaluation functors taken together allow us to assign to any coherent diagram $X\in\D(A)$ an \textbf{underlying (incoherent) diagram} $\mathrm{dia}_A(X)\colon A\to\D(\bbone)$. The resulting functor $\mathrm{dia}_A\colon\D(A)\to\D(\bbone)^A$ however, in general, is far from being an equivalence, and coherent diagrams are hence not determined by their underlying diagrams, even not up to isomorphism. Despite the importance of this distinction, frequently we draw coherent diagrams as usual and say that such a diagram has the form of or looks like its underlying diagram. There are similar \textbf{partial underlying diagram functors} $\D(A\times B)\to\D(B)^A$. 

A \emph{derivator} is a prederivator which `allows for a well-behaved calculus of Kan extensions', satisfying key properties of the calculus available in typical approaches to homotopical algebra like model categories and $\infty$-categories as well as in ordinary categories. This seemingly abstract calculus turns out to capture many typical constructions showing up in different areas of algebra, geometry, and topology; see for example the recent \cite{groth:ptstab,gps:mayer,gps:additivity,gst:basic,gst:tree,ps:linearity,ps:linearity-fp}. 

To make precise the definition of a derivator we need the following terminology. If a restriction functor $u^\ast\colon\D(B)\to\D(A)$ admits a left adjoint $u_!\colon \D(A)\to\D(B)$, then $u_!$ is a \textbf{left Kan extension functor}. A right adjoint $u_\ast\colon\D(A)\to\D(B)$ is referred to as a \textbf{right Kan extension functor}. In the examples of interest these are really \emph{homotopy Kan extension functors} but we follow the established terminology from~$\infty$-category theory and simply speak of \emph{Kan extensions}. In the special case that $B=\bbone$ is the terminal category and we hence consider the unique functor $\pi=\pi_A\colon A\to \bbone$, the functor $\pi_!=\mathrm{colim}_A$ is a \textbf{colimit functor} and $\pi_\ast=\mathrm{lim}_A$ a \textbf{limit functor}. 

To actually work with these Kan extensions, one axiomatizes the fact from classical category theory that Kan extensions in (co)complete categories exist and can be calculated pointwise (see \cite[X.3.1]{maclane}). For this purpose, let us consider the \textbf{slice squares}
\begin{equation}
\vcenter{
\xymatrix{
(u/b)\ar[r]^-p\ar[d]_-{\pi_{(u/b)}}\drtwocell\omit{}&A\ar[d]^-u&&(b/u)\ar[r]^-q\ar[d]_-{\pi_{(b/u)}}&A\ar[d]^-u\\
\bbone\ar[r]_-b&B,&&\bbone\ar[r]_-b&B\ultwocell\omit{},
}
}
\label{eq:Der4}
\end{equation}
which come with canonical transformations $u\circ p\to b\circ\pi$ and $b\circ\pi\to u\circ q.$ Here, an object in the \textbf{slice category} $(u/b)$ is a pair $(a,f)$ consisting of an object $a\in A$ and a morphism $f\colon u(a)\to b$ in $B$ while morphisms are morphisms in $A$ making the obvious triangles in $B$ commute. The functor $p\colon (u/b)\to A$ projects onto the first component, and $(b/u)$ and $q$ are defined dually.

Now, if $\bC$ is a complete and cocomplete category, $u\colon A\to B$ a functor between small categories, and $X\colon A\to \bC$ a diagram, then the left and right Kan extensions $\mathrm{LKan}_u(X),\mathrm{RKan}_u(X)\colon B\to\bC$ both exist. Moreover, for every object $b\in B$ certain canonical maps
\[
\colim_{(u/b)} X\circ p \stackrel{\cong}{\to} \mathrm{LKan}_u(X)_b\qquad\text{and}\qquad
\mathrm{RKan}_u(X)_b\stackrel{\cong}{\to}\mathrm{lim}_{(b/u)} X\circ q
\]
are isomorphisms. The existence of Kan extensions and similar pointwise formulas are axiomatized by axioms (Der3) and (Der4) in \autoref{defn:derivator}. The formalism behind (Der4) is the \emph{calculus of mates}; see the discussion of \eqref{eq:hoexactsq'} for more details. In this paper adjunction units and adjunction counits are generically denoted by $\eta$ and $\epsilon$, respectively.

\begin{defn}\label{defn:derivator}
  A prederivator $\D\colon\cCat\op\to\cCAT$ is a \textbf{derivator} if the following properties are satisfied.
  \begin{itemize}[leftmargin=4em]
  \item[(Der1)] $\D\colon \cCat\op\to\cCAT$ takes coproducts to products, i.e., the canonical map $\D(\coprod A_i)\to\prod\D(A_i)$ is an equivalence.  In particular, $\D(\emptyset)$ is equivalent to the terminal category.
  \item[(Der2)] For any $A\in\cCat$, a morphism $f\colon X\to Y$ in $\D(A)$ is an isomorphism if and only if the morphisms $f_a\colon X_a\to Y_a, a\in A,$ are isomorphisms in $\D(\bbone).$
  \item[(Der3)] Each functor $u^*\colon \D(B) \to\D(A)$ has both a left adjoint $u_!$ and a right adjoint $u_*$.
  \item[(Der4)] For any functor $u\colon A\to B$ and any $b\in B$ the canonical transformations
\begin{gather}
  \pi_! p^* \stackrel{\eta}{\to} \pi_! p^* u^* u_! \to \pi_! \pi^* b^* u_! \stackrel{\epsilon}{\to} b^* u_!  \mathrlap{\qquad\text{and}}\label{eq:Der4!}\\
  b^* u_* \stackrel{\eta}{\to} \pi_* \pi^* b^* u_* \to \pi_* q^* u^* u_* \stackrel{\epsilon}{\to} p_* q^*\label{eq:Der4*}
\end{gather}
associated to the slice squares \eqref{eq:Der4} are isomorphisms.
  \end{itemize}
\end{defn}

Axioms (Der1) and (Der3) together imply that $\D(A)$ has small categorical coproducts and products, hence, in particular, initial objects and final objects. In general, an arbitrary derivator does not admit any other 1-categorical (co)limits. Note that all axioms of a derivator ask for \emph{properties} of the underlying prederivator as opposed to asking for more \emph{structure}.

Before we list a few examples, let us introduce the following terminology. A \textbf{morphism} of derivators is simply a morphism of underlying prederivators, i.e., a pseudo-natural transformation $F\colon \D\to\E$. Similarly, given two such morphisms $F,G\colon\D\to\E$, a \textbf{natural transformation} $F\to G$ is a modification. Thus, we define the $2$-category $\cDER$ of derivators as a full sub-2-category of $\cPDER$.

\begin{egs}\label{egs:prederivators}
\begin{enumerate}
\item Associated to $\bC\in\cCAT$ there is the \textbf{represented prederivator}~$y(\bC)$ defined by $y(\bC)(A) \coloneqq \bC^A$. One notes that $y(\bC)$ is a derivator if and only if \bC is complete and cocomplete, and $u_!,u_\ast$ are then ordinary Kan extension functors. The underlying category of $y(\bC)$ is isomorphic to \bC itself. This example can be refined to a 2-functor $y\colon\cCAT\to\cPDER$.
\item Associated to a \emph{Quillen model category} \bC (see e.g.~\cite{quillen:ha,hovey:model}) with \emph{weak equivalences} \bW there is the underlying \textbf{homotopy derivator} $\ho(\bC)$ defined by formally inverting the pointwise weak equivalences $\ho(\bC)(A) \coloneqq (\bC^A)[(\bW^A)^{-1}]$. The underlying category of $\ho(\bC)$ is the homotopy category $\Ho(\bC)=\bC[\bW^{-1}]$ of \bC, and the functors $u_!,u_*$ are derived versions of the functors of $y(\bC)$ (see~\cite{cisinski:direct} for the general case and \cite{groth:ptstab} for an easy proof in the case of combinatorial model categories). 
\item Associated to an $\infty$-category $\cC$ in the sense of Joyal \cite{joyal:barca} and Lurie \cite{HTT} (see \cite{groth:scinfinity} for an introduction) there is the prederivator $\ho(\cC)$ defined by mapping $A$ to the homotopy category of $\cC^{N(A)}$ (here $N(A)$ is the nerve of $A$). A sketch proof that for complete and cocomplete $\infty$-categories this yields a derivator, the \textbf{homotopy derivator} of $\cC$, can be found in~\cite{gps:mayer}.  
\item An important construction at the level of derivators is given by \emph{shifting}. If \D is a derivator and $B\in\cCat$, then the \textbf{shifted derivator} $\shift\D B$ is defined by $\shift\D B(A) \coloneqq \D(B\times A)$ (see~\cite[Theorem~1.25]{groth:ptstab}). This operation amounts to passing to the homotopy theory of coherent diagrams of shape~$B$ in \D. The \textbf{opposite derivator} $\D\op$ is defined by $\D\op(A) \coloneqq \D(A\op)\op$. Finally, given two derivators $\D,\E$, the \textbf{product} $(\D\times\E)(A)\coloneqq\D(A)\times\E(A)$ is again a derivator. These constructions are suitably compatible with each other and with the previous three examples in that there are various natural isomorphisms like $(\shift\D B)\op \cong \shift{(\D\op)}{B\op}$.
\end{enumerate}
\end{egs}

Thus, derivators encode key formal properties of the calculus of (homotopy) Kan extensions in model categories and $\infty$-categories. For specific examples of derivators we refer the reader to \cite[Examples~5.5]{gst:basic}, but see also \autoref{egs:stable}. 

While working with derivators, one frequently shows that certain canonical maps are isomorphisms. This calculus is governed by the notion of \emph{homotopy exact squares} of small categories, a main tool in the study of derivators. Let us consider a derivator~\D and a natural transformation $\alpha\colon up\to vq$ living in a square
\begin{equation}
  \vcenter{\xymatrix{
      D\ar[r]^p\ar[d]_q \drtwocell\omit{\alpha} &
      A\ar[d]^u\\
      B\ar[r]_v &
      C
    }}\label{eq:hoexactsq'}
\end{equation}
of small categories. Using again $\eta$ and $\epsilon$ as generic notation for adjunction units and counits, respectively, we obtain \textbf{canonical mate-transformations}
\begin{gather}
  q_! p^* \stackrel{\eta}{\to} q_! p^* u^* u_! \xto{\alpha^*} q_! q^* v^* v_! \stackrel{\epsilon}{\to} v^* u_!  \mathrlap{\qquad\text{and}}\label{eq:hoexmate1'}\\
  u^* v_* \stackrel{\eta}{\to} p_* p^* u^* v_* \xto{\alpha^*} p_* q^* v^* v_* \stackrel{\epsilon}{\to} p_* q^*,\label{eq:hoexmate2'}
\end{gather}
and one shows that \eqref{eq:hoexmate1'} is an isomorphism if and only if this is the case for \eqref{eq:hoexmate2'}. 

The square \eqref{eq:hoexactsq'} is \textbf{homotopy exact} if the canonical mates \eqref{eq:hoexmate1'} and \eqref{eq:hoexmate2'} are isomorphisms for all derivators \D. (Der4) thus axiomatically asks that slice squares \eqref{eq:Der4} are homotopy exact, and many further examples of homotopy exact squares can be established (see for example \cite{ayoub:I-II,maltsiniotis:htpy-exact} and \cite{groth:ptstab,gps:mayer,gst:basic}). With the exception of \S\ref{sec:admissible}, details about this formalism are not essential to the understanding of this paper, and it instead suffices to apply the first three of the following examples.

\begin{egs}\label{egs:htpy}
\begin{enumerate}
\item If $u\colon A\to B$ is fully faithful, then the square $\id\circ u=\id \circ u$ is homotopy exact, which is to say that the unit $\eta\colon\id\to u^\ast u_!$ and the counit $\epsilon\colon u^\ast u_\ast\to\id$ are isomorphisms  (\cite[Proposition~1.20]{groth:ptstab}). \emph{Kan extensions along fully faithful functors are fully faithful.}
\item Given functors $u\colon A\to B$ and $v\colon C\to D$ then the commutative square
\begin{equation}
\vcenter{
\xymatrix{
A\times C\ar[r]^-{u\times\id}\ar[d]_-{\id\times v}&B\times C\ar[d]^-{\id\times v}\\
A\times D\ar[r]_-{u\times\id}&B\times D
}}
\end{equation}
is homotopy exact (\cite[Proposition~2.5]{groth:ptstab}). \emph{Kan extensions and restrictions in unrelated variables commute.}
\item If $u\colon A\to B$ is a right adjoint, then the square
\[
\xymatrix{
A\ar[r]^-u\ar[d]_-{\pi_A}&B\ar[d]^-{\pi_B}\\
\bbone\ar[r]_-\id&\bbone
}
\]
is homotopy exact, i.e., the canonical mate $\mathrm{colim}_Au^\ast\to\mathrm{colim}_B$ is an isomorphism (\cite[Proposition~1.18]{groth:ptstab}). In particular, if $b\in B$ is a terminal object, then there is a canonical isomorphism $b^\ast\cong\colim_B$. \emph{Right adjoint functors are \textbf{homotopy final}.}
\item The passage to canonical mates \eqref{eq:hoexmate1'} and \eqref{eq:hoexmate2'} is functorial with respect to horizontal and vertical pasting. Consequently, horizontal and vertical pastings of homotopy exact squares are homotopy exact (\cite[Lemma~1.14]{groth:ptstab}). \emph{Homotopy exact squares are compatible with pasting.}
\end{enumerate}
\end{egs}

Using the fact that Kan extensions and restrictions in unrelated variables commute (\autoref{egs:htpy}(ii)), one shows that there are \emph{parametrized versions} of restriction and Kan extension functors. In fact, given a derivator~\D and a functor $u\colon A\to B$, there are adjunctions of derivators
\begin{equation}
(u_!,u^\ast)\colon\D^A\rightleftarrows\D^B\qquad\text{and}\qquad
(u^\ast,u_\ast)\colon\D^B\rightleftarrows\D^A.\label{eq:Kan-adjunction}
\end{equation}
(Adjunctions of derivators are defined internally to the 2-category $\cDER$; see \cite[\S2]{groth:ptstab} for details including convenient reformulations.) If $u$ is fully faithful, then $u_!,u_\ast\colon\D^A\to\D^B$ are fully faithful and hence induce equivalences of derivators onto their respective essential images.

Finally, given a (pre)derivator \D, we write $X\in\D$ to indicate that there is a small category~$A$ such that $X\in\D(A)$.

\section{Stable derivators and exact morphisms}
\label{sec:stable}

In this section we recall some basics about stable derivators, mostly to fix some notation (see \cite{groth:ptstab,gps:mayer} for more details). Homotopy derivators of stable $\infty$-categories and stable model categories are stable, and stable derivators hence describe aspects of the calculus of homotopy Kan extensions available in such examples arising in algebra, geometry, and topology.

\begin{defn} 
A derivator \D is \textbf{pointed} if $\D(\bbone)$ has a zero object. 
\end{defn}

If $\D$ is pointed then so are the shifted derivators $\D^B$ and its opposite $\D\op$. It follows that all $\D(A)$ have zero objects which are preserved by restriction and Kan extension functors. 

For pointed derivators, Kan extensions along inclusions of \emph{cosieves} and \emph{sieves} `extend diagrams by zero objects'. Recall that a functor $u\colon A\to B$ is a \textbf{sieve} if it is fully faithful and if for any morphism $b\to u(a)$ in $B$ there exists an $a'\in A$ with $u(a')=b$. There is the dual notion of a \textbf{cosieve}. Kan extensions along (co)sieves are fully faithful (see \autoref{egs:htpy}).

\begin{lem}[{\cite[Prop.~1.23]{groth:ptstab}}]\label{lem:extbyzero}
Let \D be a pointed derivator and let $u\colon A\to B$ be a sieve. Then $u_\ast\colon\D^A\to\D^B$ induces an equivalence onto the full subderivator of $\D^B$ spanned by all diagrams $X\in\D^B$ such that $X_b$ is zero for all $b\notin u(A)$.
\end{lem}
\noindent
The functor $u_\ast$ is referred to as \textbf{right extension by zero}, and left Kan extension functors along cosieves are referred to as \textbf{left extension by zero}.

Let $[1]$ be the poset $(0<1)$ considered as a category. The commutative square $\square=[1]^2=[1]\times[1],$
\begin{equation}
\vcenter{
\xymatrix@-.5pc{
    (0,0)\ar[r]\ar[d] &
    (1,0)\ar[d]\\
    (0,1)\ar[r] &
    (1,1),
  }
}
\label{eq:square}
\end{equation}
comes with full subcategories $i_\ulcorner\colon\ulcorner\to\square$ and $i_\lrcorner\colon\lrcorner\to\square$ obtained by removing the terminal object and the initial object, respectively. Since both inclusions are fully faithful, so are $(i_\ulcorner)_!\colon\D^\ulcorner\to\D^\square$ and $(i_\lrcorner)_\ast\colon\D^\lrcorner\to\D^\square$. A square $Q\in\D^\square$ is \textbf{cocartesian} if it lies in the essential image of $(i_\ulcorner)_!$ and it is \textbf{cartesian} if it lies in the essential image of $(i_\lrcorner)_\ast$.

For a pointed derivator \D one can define \textbf{suspensions} and \textbf{loops}, \textbf{cofibers} and \textbf{fibers}, and similar constructions. In particular, this gives rise to adjunctions of derivators
\[
(\Sigma,\Omega)\colon\D\rightleftarrows\D\quad\text{and}\quad
(\cof,\fib)\colon\D^{[1]}\rightleftarrows\D^{[1]}.
\]
As a preparation for later sections, we sketch the constructions of $\Sigma$ and $\cof$ while the constructions of the morphisms $\Omega$ and $\fib$ are dual. Associated to a coherent morphism $(f\colon x\to y)\in\D^{[1]}$ we obtain a cocartesian square in \D looking like
\[
\xymatrix{
x\ar[r]^-f\ar[r]\ar[d]&y\ar[d]^-{\cof(f)}\\
0\ar[r]&z
}
\]
by first right extending the morphism by zero to an object of $\D^\ulcorner$ and then applying $(i_\ulcorner)_!\colon\D^\ulcorner\to\D^\square$. A final restriction of this square along the functor $[1]\to\square$ which classifies the right vertical morphism defines the cofiber $\cof(f)\in\D^{[1]}$. If $(f\colon x\to 0)$ itself is already obtained by right extension by zero, then the target of the cofiber is the suspension~$\Sigma x$. Thus there is a defining cocartesian square looking like
\[
\xymatrix{
x\ar[r]\ar[r]\ar[d]&0\ar[d]\\
0\ar[r]&\Sigma x.
}
\]
A \textbf{(\emph{coherent}) cofiber sequence} is a coherent diagram $Q\in\D^{[2]\times[1]}$ of the form
\[
\xymatrix{
x\ar[r]\ar[d]&y\ar[r]\ar[d]&0\ar[d]\\
0\ar[r]&z\ar[r]&w
}
\]
such that both squares are cocartesian and the two corners vanish as indicated. Given a coherent morphism $(x\to y)\in\D^{[1]}(A)$, we obtain an associated cofiber sequence by a right extension by zero followed by a left Kan extension. Note that the cofiber sequence functor $\D^{[1]}\to\D^{[2]\times[1]}$ is an equivalence onto its essential image; for a variant of this see \autoref{thm:AR} and \autoref{rmk:AR}(iii). Since the compound square in such a cofiber sequence is also cocartesian (\cite[Corollary~4.10]{gps:mayer}), the object $w$ is canonically isomorphic to $\Sigma x$. In particular, using this isomorphism we obtain an underlying \textbf{(\emph{incoherent}) cofiber sequence}
\begin{equation}\label{eq:triangles}
x\to y\to z\to \Sigma x,
\end{equation}
which is an ordinary diagram in $\D(A)$.

\begin{defn}
A pointed derivator is \textbf{stable} if the classes of cartesian squares and cocartesian squares coincide. These squares are then called \textbf{bicartesian}.
\end{defn}
\noindent

Homotopy derivators of stable model categories and stable $\infty$-categories are stable (\cite[Examples~5.17-5.18]{gps:mayer}). Different characterizations of stable derivators are given in \cite[Theorem~7.1]{gps:mayer} and \cite[Corollary~8.13]{gst:basic}. For convenience, we mention the following more specific examples of stable derivators, but refer the reader to \cite[\S5]{gst:basic} for many additional examples and references. 

\begin{egs}\label{egs:stable}
\begin{enumerate}
\item Let $k$ be a field. The homotopy derivator $\D_k$ of the projective model structure on unbounded chain complexes over $k$ is stable. Similarly, there is the stable derivator $\D_R$ of a ring~$R$ and the stable derivator $\D_X$ of a (quasi-compact, quasi-separated) scheme~$X$.
\item There are many Quillen equivalent stable model categories of spectra such that the homotopy category is the stable homotopy category $\mathcal{SHC}$. We denote by \cSp the homotopy derivator associated to any of these model categories. This stable derivator enjoys a universal property as we recall in \S\ref{sec:monoidal}.
\item If \D is stable then so is the shifted derivator $\D^B$ (\cite[Proposition~4.3]{groth:ptstab}). If \D and \E are stable derivators, then so are $\D\op$ and $\D\times\E$.
\end{enumerate}
\end{egs}

Recall that a derivator is \textbf{strong} if it satisfies the following axiom.
  \begin{itemize}[leftmargin=4em]
  \item[(Der5)] For any $A$, the partial underlying diagram functor $\D(A\times [1]) \to \D(A)^{[1]}$ is full and essentially surjective.
  \end{itemize}

Homotopy derivators of model categories and $\infty$-categories are strong as are represented derivators. The property of being strong does not play an essential role in the basic \emph{theory} of derivators. However, it is useful if one wants to relate \emph{properties} of stable derivators to the existence of \emph{structure} on its values. This is illustrated by the following theorem, but see also \S\ref{sec:higher}. (We want to warn the reader that, depending on the reference, the precise form of axiom (Der5) varies a bit; see \autoref{rmk:strong}.) 

\begin{thm}[{\cite[Theorem~4.16]{groth:ptstab}}]\label{thm:triang}
Let \D be a strong, stable derivator. The incoherent cofiber sequences \eqref{eq:triangles} define a triangulation on $\D(A),A\in\cCat.$ 
\end{thm}

We refer to the triangulations of the theorem as \textbf{canonical triangulations}. Recall that a morphism of derivators is \textbf{right exact} if it preserves initial objects and pushouts. There is also the dual notion of \textbf{left exact} morphisms and the combined notion of \textbf{exact} morphisms; see for example \cite{groth:ptstab} for some basics on these notions. For stable derivators these three notions clearly coincide.

\begin{prop}[{\cite[Proposition~4.18]{groth:ptstab}}]\label{thm:triangcan}
Let $F\colon\D\to\E$ be an exact morphism of strong, stable derivators. The components $F_A\colon\D(A)\to\E(A)$ can be endowed with the structure of an exact functor with respect to the canonical triangulations.
\end{prop}

Left adjoint morphisms of derivators preserve colimits and dually. Hence, the proposition applied to \eqref{eq:Kan-adjunction} yields the following.

\begin{cor}[{\cite[Corollary~4.19]{groth:ptstab}}]
Let \D be a strong, stable derivator and let $u\colon A\to B$ be a functor between small categories. The functors
\[
u^\ast\colon\D(B)\to\D(A),\quad u_!\colon\D(A)\to\D(B),\quad\text{and}\quad u_\ast\colon\D(A)\to\D(B)
\]
can be canonically turned into exact functors with respect to canonical triangulations.
\end{cor}

Thus, like stable model categories and stable $\infty$-categories, stable derivators provide an enhancement of triangulated categories. In \S\ref{sec:higher} we shall see that stable derivators also provide an enhancement of higher triangulations. Hence, a posteriori, this result also applies to stable model categories and stable $\infty$-categories.

The following theorem reinforces the correctness of the definition of left exact morphisms. Let us recall that a small category $A$ is \textbf{strongly homotopy finite} if the nerve $NA$ has only finitely many non-degenerate simplices, i.e., if $A$ is finite, skeletal, and has no nontrivial endomorphisms. We say that $A$ is \textbf{homotopy finite} if it is equivalent to a strongly homotopy finite category. 

\begin{thm}[{\cite[Theorem~7.1]{ps:linearity}}]\label{thm:exact}
Let $A$ be a homotopy finite category. Any right exact morphism of derivators preserves colimits of shape~$A$. 
\end{thm}

The canonical triangulations of \autoref{thm:triang} reduce to the classical triangulations in \autoref{egs:stable}. With this in mind, the following notion was introduced in \cite{gst:basic}.

\begin{defn}\label{defn:sse}
Two small categories $A$ and $A'$ are \textbf{strongly stably equivalent}, in notation $A\sse A',$ if for every stable derivator \D there is an equivalence of derivators
\[
\D^A\simeq\D^{A'}
\]
which is natural with respect to exact morphisms of stable derivators. Such a natural equivalence is a \textbf{strong stable equivalence}.
\end{defn}

Note that for a quiver $Q$ and a ring $R$ there is an equivalence $\D_R^Q\simeq \D_{RQ}$. Hence, if two quivers $Q$ and $Q'$ are strongly stably equivalent, then we obtain natural equivalences
\[
\D_{RQ}\simeq\D_R^Q\simeq \D_R^{Q'}\simeq\D_{RQ'}.
\]
In particular, $Q$ and $Q'$ are derived equivalent over arbitrary rings. But being strongly stably equivalent is a much stronger property since similar equivalences are also assumed to exist in many other contexts (see \cite[\S5]{gst:basic}). For examples of strong stable equivalences we refer to \cite{gst:basic,gst:tree}.

The following seemingly technical result will be used a lot in later sections. It often allows us to `detect' (co)cartesian squares in larger diagrams.

\begin{lem}[{\cite[Prop.~3.10]{groth:ptstab}}]\label{lem:detection}
  Suppose $u\colon A\to B$ and $k\colon \square \to B$ are functors, with $k$ injective on objects, and let $b = k(1,1)\in B$.
  Suppose furthermore that $b\notin u(A)$, and that the functor $\mathord\ulcorner \to (B-\{b\}/b)$ induced by $k$ has a left adjoint.
  Then for any derivator \D and any $X\in \D(A)$, the square $k^* u_! X$ is cocartesian.
\end{lem}

\section{Coherent Auslander--Reiten quivers}
\label{sec:An}

In this section we construct coherent diagrams of the shape of the Auslander--Reiten quiver for Dynkin quivers of type~$A$. Derivators of such representations will be used in \S\ref{sec:coxeter-reflection} to define reflection functors, Coxeter functors, and Serre functors. The results of this section have variants for pointed derivators, but for simplicity they are only stated for stable derivators.

Let us begin by recalling that associated to every quiver $Q$ there is the repetitive quiver $\widehat{Q}$ described as follows. Vertices in $\widehat{Q}$ are pairs $(k,q)$ with $k\in\lZ$ and $q\in Q$ while for every edge $\alpha\colon q_1\to q_2$ in $Q$ there are edges $\alpha\colon (k,q_1)\to (k,q_2)$ and $\alpha^\ast\colon (k,q_2)\to (k+1,q_1)$ in $\widehat{Q}$. For example if $\A{n}=(1<\ldots < n)$ is the $A_{n}$-quiver with a linear orientation, then the repetitive quiver of $\A{3}$ looks like:
\begin{equation} \label{eq:ar-quiver}
\vcenter{
\xymatrix@R=0.8em@C=0.5em{
\ar[dr] && (-1,3) \ar[dr]^{\beta^*} && (0,3) \ar[dr]^{\beta^*} && (1,3) \ar[dr]^{\beta^*} && (2,3) \ar[dr]\\
\cdots & (-1,2) \ar[ur]^{\beta} \ar[dr]_-{\alpha^\ast} &&
(0,2) \ar[ur]^{\beta}  \ar[dr]_-{\alpha^\ast} && (1,2) \ar[ur]^{\beta}  \ar[dr]_-{\alpha^\ast} &&
(2,2) \ar[ur]^{\beta}  \ar[dr]_-{\alpha^\ast} && \cdots\\
\ar[ur] && (0,1) \ar[ur]_-\alpha && (1,1) \ar[ur]_-\alpha && (2,1) \ar[ur]_-\alpha && (3,1) \ar[ur]
}
}
\end{equation}
We denote by $M_{\A{n}}$ the category obtained from the repetitive quiver of $\A{n}$ by forcing all squares to commute, i.e., the free category generated by the graph modulo these commutativity relations. Committing a minor abuse of terminology, we refer to $M_{\A{n}}$ as the \textbf{mesh category}.

\begin{rmk}
There is the typical source of confusion resulting from the different indexing conventions used in topology and algebra. The topological convention is to denote the ordinal $(0<\ldots<n)$ by $[n]$ while algebraists refer to it as an $A_{n+1}$-quiver, emphasizing the number of vertices. Of course there are obvious variants of the mesh category starting with the quiver $[n]$ instead, the only difference being the decoration of the objects.
\end{rmk}

Let now \D be a stable derivator and $X\in\D^{\A{n}}$. We want to construct a coherent diagram of shape $M_n=M_{[n+1]}$ for $[n+1]=(0<\ldots <n+1)$ satisfying certain exactness and vanishing conditions. Note that there is the fully faithful functor
\begin{equation}
i\colon \A{n}\to M_n\colon l\mapsto (0,l) 
\label{eq:i}
\end{equation}
which we consider as an inclusion. This embedding factors as a composition of inclusions of full subcategories
\[
i\colon \A{n}\stackrel{i_1}{\to} K_1\stackrel{i_2}{\to} K_2\stackrel{i_3}{\to} K_3\stackrel{i_4}{\to} M_n
\]
where
\begin{enumerate}
\item $K_1$ is obtained from $\A{n}$ by adding the objects $(k,n+1)$ for $k\geq 0$ and $(k,0)$ for $k>0$,
\item $K_2$ contains all objects from $K_1$ and the objects $(k,l), k>0$, and
\item $K_3$ is obtained from $K_2$ by adding the objects $(k,n+1)$ for $k<0$ and $(k,0)$ for $k\leq 0$.
\end{enumerate}
The inclusion $i_4$ thus adds the remaining objects in the negative $k$-direction. By \autoref{egs:htpy}, associated to these fully faithful functors there are fully faithful Kan extension functors
\begin{equation}
\vcenter{
\xymatrix{
\D^{\A{n}}\ar[r]^-{(i_1)_\ast}&\D^{K_1}\ar[r]^-{(i_2)_!}&\D^{K_2}\ar[r]^-{(i_3)_!}&
\D^{K_3}\ar[r]^-{(i_4)_\ast}&\D^{M_n}.
}
}
\label{eq:AR}
\end{equation}

Let us denote by $\D^{M_n,\exx}\subseteq\D^{M_n}$ the full subderivator spanned by all coherent diagrams which vanish at $(k,0),(k,n+1)$ for all $k\in\lZ$ and which make all squares bicartesian. The following theorem justifies that we refer to $\D^{M_n,\exx}$ as a derivator.

\begin{thm}\label{thm:AR}
Let \D be a stable derivator. Then \eqref{eq:AR} induces an equivalence of stable derivators $(F,i^\ast)\colon\D^{\A{n}}\rightleftarrows\D^{M_n,\exx}$ which is natural with respect to exact morphisms. Moreover, the inclusion $\D^{M_n,\exx}\to\D^{M_n}$ is exact.
\end{thm}
\begin{proof}
Since the functors in \eqref{eq:AR} are fully faithful, their composition induces an equivalence onto the essential image. Let us analyze the individual steps. The functor $i_1\colon\A{n}\to K_1$ is the inclusion of a sieve. Hence it follows from \autoref{lem:extbyzero} that $(i_1)_\ast$ is right extension by zero. Thus, $(i_1)_\ast\colon\D^{\A{n}}\to\D^{K_1}$ induces an equivalence onto the full subderivator of $\D^{K_1}$ spanned by all diagrams which vanish at $(k,n+1),k\geq 0,$ and $(k,0),k>0$.

The second step $(i_2)_!$ is also fully faithful, and hence induces an equivalence onto its essential image. We claim that this functor amounts to adding cocartesian squares everywhere in the positive $k$-direction. In fact, the functor $i_2\colon K_1\to K_2$ can be factored into countably many intermediate steps adding one object $(k,l)$ with $k>0,0<l<n+1$ at a time using a nested induction over $k>0$ and $l,0<l<n+1,$ in the increasing $l$-direction. In each of these steps, an application of \autoref{lem:detection} together with \autoref{egs:htpy} implies that the corresponding left Kan extension functor induces an equivalence onto the full subderivator consisting of all diagrams making the new square cocartesian. Thus, $(i_2)_!\colon\D^{K_1}\to\D^{K_2}$ induces an equivalence onto the full subderivator of $\D^{K_2}$ spanned by the objects making all squares cocartesian. 

The remaining steps are now symmetric. The functor $i_3$ is the inclusion of a cosieve, hence $(i_3)_!$ is left extension by zero (\autoref{lem:extbyzero}). And using similar arguments as in the case of $(i_2)_!$, we see that $(i_4)_\ast\colon\D^{K_3}\to\D^{M_n}$ induces an equivalence onto the full subderivator spanned by the diagrams making all squares in the negative $k$-direction cartesian (again, by a repeated application of \autoref{lem:detection}). Taking these steps together, we deduce that, for a stable derivator \D, the construction \eqref{eq:AR} induces an equivalence of derivators $\D^{\A{n}}\to\D^{M_n,\exx}$. Since we only extended by zeros and added (co)cartesian squares, one checks that these equivalences are natural with respect to exact morphisms.

It remains to show that the inclusion $\D^{M_n,\exx} \to \D^{M_n}$ preserves zero objects and (co)cartesian squares. We begin by noting that each $\D^{M_n,\exx}(A)$ contains the zero object of $\D^{M_n}(A)$ which is simply the constant $M_n$-shaped diagram taking value $0$. Since it clearly satisfies the defining vanishing conditions of $\D^{M_n,\exx}$ it suffices to observe that constant squares, i.e., coherent diagrams in the image of $\pi^\ast\colon\D\to\D^\square$, are bicartesian (see \cite[Proposition~3.12]{groth:ptstab}). Thus, $\D^{M_n,\exx} \to \D^{M_n}$ preserves zero objects. 

As for the preservation of cartesian squares, let $X \in \D^{M_n,\exx}(\lrcorner)$ and consider the right Kan extension $(i_\lrcorner)_\ast(X) \in \D^{M_n}(\square)$ taken in $\D^{M_n}$. We have to show that $(i_\lrcorner)_\ast(X)$ belongs to $\D^{M_n,\exx}(\square)$, i.e., we have to check certain vanishing conditions and exactness conditions. For $a=(k,0),(k,n+1),k\in\lZ,$ by \autoref{egs:htpy}, we have $(i_\lrcorner)_\ast(X)_a\cong(i_\lrcorner)_\ast(X_a)=(i_\lrcorner)_\ast(0)\cong 0$ so that the vanishing conditions hold. To verify the exactness conditions, let us consider one of the squares $q\colon\square\to M_n$ which has to be made bicartesian. Considering $X\in\D^{M_n,\exx}(\lrcorner)$ as an object of $\D(M_n\times\lrcorner)$, we have to show that $(q\times\id)^\ast(\id\times i_\lrcorner)_\ast(X)$ lies in the essential image of $(i_\lrcorner\times\id)_\ast$. Note that by \autoref{egs:htpy} this object is isomorphic to $(\id\times i_\lrcorner)_\ast(q\times\id)^\ast(X)$. Since $X$ lies in $\D^{M_n,\exx}(\lrcorner)$ it follows that $(q\times\id)^\ast(X)$ lies in the essential image of $(i_\lrcorner\times\id)_\ast$, and we can conclude by observing that $(\id\times i_\lrcorner)_\ast$ and $(i_\lrcorner\times\id)_\ast$ commute up to canonical isomorphism.
\end{proof}

\begin{rmk}\label{rmk:AR}
\begin{enumerate}
\item There is also a variant of this result for pointed derivators. In that case, we would consider the full subderivator of $\D^{M_n}$ spanned by all diagrams satisfying the above vanishing conditions, making all squares in the positive $k$-direction cocartesian and the ones in the negative $k$-direction cartesian. In particular, the following three special cases have variants in the pointed context.
\item For $n=1$ this result tells us that a pointed derivator is equivalent to the derivator of $\lZ$-graded spectrum objects which are suspension spectra in the positive direction and $\Omega$-spectra in the negative direction (see also \cite{heller:stable} and \cite{HA}).
\item For $n=2$ this result specializes to the statement that the derivator of morphisms is equivalent to the derivator of Barratt--Puppe sequences.
\item For larger values of $n$ this result gives us a coherent and doubly-infinite version of the Waldhausen $S_\bullet$-construction \cite{waldhausen:k-theory}.
\end{enumerate}
\end{rmk}

If we consider representations of an $A_n$-quiver over a field, then there are $\frac {n(n+1)}{2}$ isomorphism classes of indecomposable such representations by~\cite[\S2.2]{gabriel:unzerlegbare}. This is reflected by the fact that there are precisely $\frac{n(n+1)}{2}$ different suspension orbits in the classical Auslander--Reiten quiver of the bounded derived category of finite dimensional representations; see~\cite[\S4]{happel:fd-algebra}. Given an arbitrary stable derivator, there is the same number of suspension orbits in the coherent Auslander--Reiten quivers of \autoref{thm:AR}. This is a consequence of the following proposition in which we consider the fully faithful functor

\begin{equation}
j\colon\A{n}\to M_n\colon l\mapsto (l,n+1-l). 
\label{eq:j}
\end{equation}

\begin{prop}\label{prop:susp}
Let \D be a stable derivator. For every $n\geq 1$ there is a natural isomorphism $\Sigma \cong j^\ast F\colon\D^{\A{n}}\to\D^{\A{n}}$.
\end{prop}
\begin{proof}
In order to obtain such a natural isomorphism, let us consider the functor 
\[
q\colon \A{n}\times\square\to M_n
\]
which we define coordinatewise by
\begin{enumerate}
\item $q(l,00)=i(l)=(0,l)$ for each $l\in\A{n}$,
\item $q(l,10)=(0,n+1)$ for each $l\in\A{n}$,
\item $q(l,01)=(l,0)$ for each $l\in\A{n}$, and 
\item $q(l,11)=j(l)=(l,n+1-l)$ for each $l\in\A{n}$.
\end{enumerate}
Since $M_n$ is a poset there is at most one such functor and we leave it to the reader to check that this functor actually exists. For $X\in\D^{\A{n}}$ and $l\in\A{n}$ one observes that $(q^\ast FX)_l\in\D^\square$ is cocartesian. In fact, the functor $q(l,-)\colon\square\to M_n$ can be obtained as a restriction of a unique embedding $p_l\colon \A{n+2-l}\times \A{l+1}\to M_n$ to the boundary such that each square in $p_l^\ast FX$ is cocartesian. Since horizontal and vertical compositions of cocartesian squares are again cocartesian it follows that $(q^\ast FX)_l$ is cocartesian for every $l$. It is then a consequence of \cite[Corollary~3.14]{groth:ptstab} that $q^\ast FX$ considered as an object of $(\D^{\A{n}})^\square$ is also cocartesian. Since the restrictions of $q^\ast FX$ along $(1,0),(0,1)\colon\bbone\to\square$ vanish, the diagram $q^\ast FX$ can be used to calculate the suspension morphism $\Sigma\colon\D^{\A{n}}\to\D^{\A{n}}$. Thus, we obtain the desired natural isomorphism
\[
j^\ast FX= (q^\ast FX)_{(1,1)}\cong\Sigma\big((q^\ast FX)_{(0,0)}\big)=\Sigma i^\ast FX\cong \Sigma X,
\]
concluding the proof.
\end{proof}

Let us consider a pair of composable arrows $X=(x\to y\to z)\in\D^{\A{3}}$ . It is a consequence of \autoref{thm:AR} and \autoref{prop:susp} that $F(X)\in\D^{M_3,\exx}$ looks like

\begin{equation}\label{eq:A3}
\vcenter{
\xymatrix@=1em{
\cdots & 0 \ar[dr] && 0 \ar[dr] && 0 \ar[dr] && 0  \ar[dr] && \cdots\\
\ar[ur]\ar[dr] && \Omega w \ar[ur]\ar[dr] && z\ar[ur] \ar[dr] && \Sigma x\ar[ur] \ar[dr] && \Sigma u\ar[ur] \ar[dr]\\
\cdots & \Omega v \ar[ur] \ar[dr] &&
y \ar[ur] \ar[dr] && v \ar[ur] \ar[dr] &&
\Sigma y \ar[ur] \ar[dr] && \cdots\\
\ar[ur]\ar[dr] && x \ar[ur]\ar[dr] && u \ar[ur] \ar[dr] && w \ar[ur] \ar[dr] && \Sigma z \ar[ur] \ar[dr]\\
\cdots & 0 \ar[ur] && 0 \ar[ur] && 0 \ar[ur] && 0  \ar[ur] && \cdots\\
}
}
\end{equation}
for suitable cone objects $u,v,w\in\D$. So far we know that all objects in $FX$ look as claimed and it will follow from \autoref{cor:susp} that the coherent `triangular fundamental domains' are suspensions of each other. In particular, we obtain the good number of suspension orbits.

The mesh category $M_n$ has a few obvious symmetries, and the induced symmetries on $\D^{M_n}$ actually restrict to $\D^{M_n,\exx}$. We say a bit more about this in \S\ref{sec:coxeter-reflection}. Here, we only observe the following. Let $f=f_n\colon M_n\to M_n$ be the `flip symmetry' which identifies each of the `triangular fundamental domains' with the next one in the positive $k$-direction. More precisely, we set
\begin{equation}
f\colon M_n\to M_n\colon (k,l)\mapsto (k+l,n+1-l).
\label{eq:f}
\end{equation}
It follows immediately from \eqref{eq:i} and \eqref{eq:j} that $f$ satisfies $j=fi\colon \A{n}\to M_n$. 

\begin{cor}\label{cor:susp}
Let \D be a stable derivator. The isomorphism $f^\ast\colon\D^{M_n}\to\D^{M_n}$ restricts to an isomorphism $f^\ast\colon\D^{M_n,\exx}\to\D^{M_n,\exx}$ and there is a natural isomorphism $\Sigma\cong f^\ast\colon\D^{M_n,\exx}\to\D^{M_n,\exx}$.
\end{cor}
\begin{proof}
It is immediate that the defining exactness and vanishing conditions of $\D^{M_n,\exx}$ are invariant under the flip $f\colon M_n\to M_n$, and we hence obtain an induced isomorphism $f^\ast\colon\D^{M_n,\exx}\to\D^{M_n,\exx}$. Since the inclusion $j\colon \A{n}\to M_n$ factors as 
$j=fi\colon \A{n}\to M_n$, it follows from \autoref{prop:susp} that there is a natural isomorphism $\Sigma\cong j^\ast F= i^\ast f^\ast F\colon \D^{\A{n}}\to\D^{\A{n}}$. As equivalences are exact, \autoref{thm:AR} allows us to conclude by $f^\ast\cong F\Sigma i^\ast\cong \Sigma F i^\ast\cong \Sigma\colon\D^{M_n,\exx}\to\D^{M_n,\exx}$.
\end{proof}

It turns out that the construction of \autoref{thm:AR} does not depend in an essential way on the original orientation of the $A_n$-quiver. Given an $A_n$-quiver~$Q$ with arbitrary orientation, observe first that the repetitive quiver $\widehat{Q}$ of $Q$ is isomorphic to the repetitive quiver of $\A{n}$, \cite{riedtmann:back}. By definition, $Q$ has exactly two points $e_1,e_2$ incident with one arrow only (ignoring the degenerated case $n=1$ here). Let $Q_\mathrm{ext}$ be the quiver of type $A_{n+2}$ obtained from $Q$ by attaching new arrows $\alpha_i$ starting at $e_i$ for $i=1,2$. Using the terminology of~\cite[\S8]{gst:tree}, $Q_\mathrm{ext}$ is obtained from $Q$ by applying one-point extension twice, and there is an obvious embedding $j\colon Q \to Q_\mathrm{ext}$.
By the above, there exists an isomorphism $\varphi\colon \widehat{Q_\mathrm{ext}} \to \widehat{Q_n}$ of the repetitive quivers, where $Q_n = [n+1] = (0<\ldots <n+1)$, and viewing $\varphi$ also as an isomorphism between the corresponding free categories, we obtain a functor
\begin{equation}\label{eq:iQ}
i_Q\colon
\xymatrix@1{ Q \ar[r]^-j & Q_\mathrm{ext} \ar[r] & \widehat{Q_\mathrm{ext}} \ar[r]^-\varphi_-\cong & \widehat{Q_n} \ar[r] & M_n }.
\end{equation}
It is easily checked that $i_Q$ is fully faithful and injective on objects. Strictly speaking, $i_Q$ also depends on $\varphi$ rather than just on $Q$, but we suppress this from notation. Such an embedding $i_Q$ always has the property that for each $1 \le l \le n$, there is precisely one integer $k \in \lZ$ such that $(k,l)$ is in the image of $i_Q$.

\begin{eg}\label{eg:embedding}
If we choose the $A_3$-quiver $Q=(1\ot 2\to 3)$, then the embedding $i_Q\colon Q\to M_3$ (under the appropriate choice of $\varphi$) is determined by the fact that the restriction of~\eqref{eq:A3} along $i_Q$ looks like $(u\ot y\to z)$.
\end{eg}

The mesh category $M_n$ hence contains the free categories of all $A_n$-quivers as full subcategories. The strategy for proving \autoref{thm:AR} can be easily adapted to establishing the corresponding result for arbitrary $A_n$-quivers.

\begin{thm} \label{thm:AR-independent}
Let \D be a stable derivator, let $Q$ be an $A_n$-quiver for $n\geq 1$, and let $i_Q\colon Q\to M_{n}$ be as in \eqref{eq:iQ}. Then there is an equivalence of stable derivators $(F_Q,i_Q^\ast)\colon \D^Q\rightleftarrows \D^{M_n,\exx}$ which is natural with respect to exact morphisms.
\end{thm}

Thus, $\D^{M_n,\exx}$ is the stable derivator of coherent Auslander--Reiten quivers for arbitrary Dynkin quivers of type~$A_n$ in \D.

\section{Coxeter, Serre, and the fractionally Calabi--Yau property}
\label{sec:coxeter-reflection}

Let $Q$ be an $A_n$-quiver and let \D be a stable derivator. In this section we introduce some symmetries on the stable derivators $\D^Q\simeq\D^{M_n,\exx}$, namely \emph{reflection functors}, \emph{(partial) Coxeter functors}, \emph{Auslander--Reiten translations}, and \emph{Serre functors}. We also establish the fractionally Calabi--Yau property (which will prove useful in the comparison of spectral and derived Picard groupoids) and give an elementary explanation of it. The justification for the terminology employed in this section will be given in \S\ref{sec:field} where we see that for the derivator of a field these functors induce the usual classical counterparts, defined at the level of derived categories.

\subsection{Coxeter functors}
\label{subsec:Coxeter}

Let $Q$ be a quiver and let $a\in Q$ be a sink, i.e., a vertex such that all edges adjacent to it end at it. The reflected quiver $Q'=\sigma_a Q$ is obtained from $Q$ by changing the orientations of all edges adjacent to $a$. In particular, $a\in Q'$ is now a source. In \cite[Theorem~9.11]{gst:tree} it was shown that if the quiver under consideration is a tree, then $Q$ and $Q'$ are strongly stably equivalent in the sense of \autoref{defn:sse}.

Elementary aspects in the case of Dynkin quivers of type~$A$ were already discussed in \cite[\S6]{gst:basic}, but this situation will be studied more systematically here. Given a sink $a\in Q$ in an $A_n$-quiver, let $Q'=\sigma_aQ$ be the reflected $A_n$-quiver. Implicitly, \textbf{reflection functors}
\begin{equation}\label{eq:reflection}
(s^-_a,s^+_a)\colon\D^{Q'}\rightleftarrows \D^Q
\end{equation}
are constructed in \emph{loc.~cit.} and shown to define strong stable equivalences. Depending on the valence of the sink the morphism $s_a^+$ is obtained simply by passing to the fibre of the unique morphism adjacent to the sink or by extending the two morphisms adjacent to the sink to a cartesian square and then restricting appropriately. (A naturally isomorphic description is obtained by specializing the constructions in~\cite{gst:tree}.) By induction these reflection functors show that two Dynkin quivers of type~$A$ are strongly stably equivalent if and only if they have the same length (\cite[Theorem~6.5]{gst:basic}). Strictly speaking the notation of these reflection functors should come with additional decorations given by the quivers but we prefer to drop them from notation.

Recall that the mesh category $M_{n}$ comes with embeddings $i_Q\colon Q\to M_{n}$ for all $A_n$-quivers~$Q$ and that there are natural equivalences $(F_Q,i_Q^\ast)\colon\D^Q\rightleftarrows\D^{M_n,\exx}$ (\autoref{thm:AR} and \autoref{thm:AR-independent}). There is some choice involved regarding the allowable functors $i_Q$. In fact, if $\varphi\colon M_n\to M_n$ is an automorphism, then the automorphism $\varphi^\ast\colon\D^{M_n}\to\D^{M_n}$ clearly restricts to an automorphism of $\D^{M_n,\exx}$. Thus, if we compose the embedding $i_Q$ with such an automorphism $\varphi\colon M_{n} \to M_{n}$ to obtain a new embedding $i'_Q = \varphi \circ i_Q$, then there is a further equivalence
\[
(F'_Q,(i'_Q)^\ast)=((\varphi^{-1})^\ast,\varphi^\ast)\circ(F_Q,i_Q^\ast)\colon\D^Q\rightleftarrows\D^{M_n,\exx},
\]
which typically is different from $(F_Q,i_Q^\ast)$. Given two $A_n$-quivers which are related by a reflection, we can always assume that the corresponding embeddings are compatible in the following sense.

\begin{hyp} \label{hyp:rel-embed}
Let $Q$ be a Dynkin quiver of type $A_n$, let $a\in Q$ be a sink, and let $Q'=\sigma_a Q$ be the reflected quiver. Then we assume that the embeddings $i_Q\colon Q \to M_{n}$ and $i_{Q'}\colon Q' \to M_{n}$ are chosen so that
\begin{itemize}
\item $i_Q(b) = i_{Q'}(b)$ for all vertices different from $a$, and
\item if $i_Q(a) = (k,l) \in M_{n}$, then $i_{Q'}(a) = (k-1,l)$.
\end{itemize}
\end{hyp}

\begin{eg}
Taking up again \autoref{eg:embedding}, let us reflect $Q$ at the sink $a=1$. Then $Q'=\sigma_1Q=\A{3}$ is the linearly oriented $A_3$-quiver and the resulting $i_{Q'}$ is the embedding $i$ as in \eqref{eq:i}. 
\end{eg}

This hypothesis allows us for example to give the following description of the reflection functors.

\begin{prop}\label{prop:sa}
Let \D be a stable derivator, let $Q$ be an $A_n$-quiver, let $a\in Q$ be a sink, let $Q'=\sigma_a Q$ be the reflected quiver, and let $i_Q,i_{Q'}$ be as in \autoref{hyp:rel-embed}. There are natural isomorphisms
\[
s^+_a\cong i_{Q'}^\ast F_Q\colon\D^Q\to \D^{Q'} \qquad\text{and}\qquad
s^-_a\cong i_Q^\ast F_{Q'}\colon\D^{Q'}\to \D^Q.
\]
\end{prop}
\begin{proof}
This follows immediately from the construction of the reflection functors \eqref{eq:reflection}, \autoref{thm:AR-independent}, and the defining exactness properties of $\D^{M_n,\exx}$ given prior to \autoref{thm:AR}.
\end{proof}

By iteration one obtains the following proposition which reproduces \cite[Theorem~6.5]{gst:basic}.

\begin{prop}\label{prop:tiltAn}
Let \D be a stable derivator and let $Q,Q'$ be $A_n$-quivers. There are strong stable equivalences 
\[
(i_Q^\ast F_{Q'},i_{Q'}^\ast F_Q)\colon\D^{Q'}\rightleftarrows \D^Q.
\]
\end{prop}

Let us assume that $a_1,a_2\in Q$ are distinct sinks in an $A_n$-quiver. Note that $a_2$ is also a sink in $\sigma_{a_1}Q$ and we can hence pass to $\sigma_{a_2}\sigma_{a_1} Q$. It is easy to observe that at the level of quivers the order of these reflections does not matter and we can hence define
\[
\sigma_{\{a_1,a_2\}}Q=\sigma_{a_1}\sigma_{a_2} Q=\sigma_{a_2}\sigma_{a_1} Q.
\]
Also the associated reflection functors commute up to natural isomorphisms.

\begin{cor}\label{cor:sasb-commute}
Let \D be a stable derivator and let $Q$ be an $A_n$-quiver.
\begin{enumerate}
\item If $a_1,a_2\in Q$ are sinks and if $Q'=\sigma_{\{a_1,a_2\}}Q$ is the reflected quiver, then there is a natural isomorphism $s^+_{a_1}s^+_{a_2}\cong s^+_{a_2}s^+_{a_1}\colon\D^Q\to\D^{Q'}$. 
\item If $b_1,b_2\in Q$ are sources and if $Q'=\sigma_{\{b_1,b_2\}}Q$ is the reflected quiver, then there is a natural isomorphism $s^-_{b_1}s^-_{b_2}\cong s^-_{b_2}s^-_{b_1}\colon\D^Q\to\D^{Q'}$. 
\end{enumerate}
\end{cor}
\begin{proof}
By duality it is enough to prove the first statement. Assuming \autoref{hyp:rel-embed} for the three pairs $(Q, \sigma_{a_1} Q)$, $(Q, \sigma_{a_2} Q)$, and $(\sigma_{a_1} Q, \sigma_{\{a_1,a_2\}}Q)$, the hypothesis is automatically satisfied also for the pair $(\sigma_{a_2} Q, \sigma_{\{a_1,a_2\}}Q)$. Hence by combining the natural isomorphisms from \autoref{prop:sa} for both sinks we obtain a natural isomorphism
\[
s^+_{a_2}s^+_{a_1}\cong i^\ast_{\sigma_{\{a_1,a_2\}}Q}F_{\sigma_{a_1} Q}i^\ast_{\sigma_{a_1} Q} F_Q\cong i^\ast_{\sigma_{\{a_1,a_2\}}Q} F_Q.
\]
By symmetry we also obtain $s^+_{a_1}s^+_{a_2}\cong i^\ast_{\sigma_{\{a_1,a_2\}}Q} F_Q$, concluding the proof.
\end{proof}

Let $Q$ be a finite quiver. Recall that an \textbf{admissible sequence of sinks} in $Q$ is a total ordering $(a_1,\ldots,a_n)$ of all its vertices such that 
\begin{enumerate}
\item $a_1$ is a sink in $Q$ and
\item $a_i$ is a sink in $\sigma_{a_{i-1}}\ldots\sigma_{a_1}Q$ for all $2\leq i\leq n$.
\end{enumerate}
There is the dual notion of an \textbf{admissible sequence of sources}. By a combinatorial argument~\cite[Lemma 1.2(1)]{bernstein-gelfand-ponomarev:Coxeter}, every finite, acyclic quiver admits an admissible sequence of sources and sinks. Moreover, one observes that $\sigma_{a_n}\ldots\sigma_{a_1}Q=Q$ and similarly in the case of sources.

In general, a quiver admits many such admissible sequences. Let $Q$ be an $A_n$-quiver and let $(a_1,\ldots,a_n)$ be an admissible sequence of sinks. Associated to this sequence, for every stable derivator \D there is the composition of reflection functors
\[
\Phi^+_{(a_1,\ldots,a_n)}=s^+_{a_n}\circ\ldots\circ s^+_{a_1}\colon\D^Q\to\D^Q.
\]
Dually, we define $\Phi^-_{(b_1,\ldots,b_n)}$ for every admissible sequence of sources.

\begin{prop} \label{prop:indep-admis}
Let \D be a stable derivator and let $Q$ be an $A_n$-quiver.
\begin{enumerate}
\item If $(a_1,\ldots,a_n)$,  $(a'_1,\ldots,a'_n)$ are admissible sequences of sinks in $Q$, then there is a natural isomorphism $\Phi^+_{(a_1,\ldots,a_n)}\cong \Phi^+_{(a'_1,\ldots,a'_n)}\colon\D^Q\to\D^Q.$
\item If $(b_1,\ldots,b_n)$,  $(b'_1,\ldots,b'_n)$ are admissible sequences of sources in $Q$, then there is a natural isomorphism $\Phi^-_{(b_1,\ldots,b_n)}\cong \Phi^-_{(b'_1,\ldots,b'_n)}\colon\D^Q\to\D^Q.$
\end{enumerate}
\end{prop}
\begin{proof}
By means of \autoref{cor:sasb-commute} the task is reduced to a purely combinatorial problem. In fact, it suffices to show that any two admissible sequences of sinks can be related by finitely many `transpositions' as in \autoref{cor:sasb-commute}. But this is well-known, see the proof of~\cite[Lemma 1.2(3)]{bernstein-gelfand-ponomarev:Coxeter}. 
\end{proof}

By this proposition the following functors are well-defined up to natural isomorphisms. 

\begin{defn} \label{defn:Coxeter}
Let \D be a stable derivator and let $Q$ be an $A_n$-quiver.
\begin{enumerate}
\item The \textbf{Coxeter functor} $\Phi^+=\Phi^+_Q\colon\D^Q\to\D^Q$ is $\Phi^+=\Phi^+_{(a_1,\ldots,a_n)}$ for some admissible sequence of sinks $(a_1,\ldots,a_n)$ in $Q$.
\item The \textbf{Coxeter functor} $\Phi^-=\Phi^-_Q\colon\D^Q\to\D^Q$ is $\Phi^-=\Phi^-_{(b_1,\ldots,b_n)}$ for some admissible sequence of sources $(b_1,\ldots,b_n)$ in $Q$.
\end{enumerate}
\end{defn}

These Coxeter functors correspond to the following symmetry of $\D^{M_n,\exx}$. The functor 
\begin{equation}
t\colon M_{n}\to M_{n}\colon (k,l)\mapsto (k-1,l)
\label{eq:t}
\end{equation}
induces a restriction functor $\tau=t^\ast\colon \D^{M_{n}}\to\D^{M_{n}}$. For a stable derivator \D, it is obvious that the defining exactness and vanishing conditions of $\D^{M_n,\exx}$ are invariant under $t$ and we hence obtain an induced functor $\tau\colon\D^{M_n,\exx}\to\D^{M_n,\exx}$ which shifts a diagram one step in the positive $k$-direction. We write $\tau^-$ for its inverse $\tau^-=\tau^{-1}=(t^\ast)^{-1}=(t^{-1})^\ast$.

\begin{cor}\label{cor:phi-tau}
Let \D be a stable derivator and let $Q$ be an $A_n$-quiver. The diagrams 
\[
\xymatrix{
\D^Q\ar[d]_-{F_Q}\ar[r]^-{\Phi^+_Q}&\D^Q\ar[d]^-{F_Q}&&
\D^Q\ar[d]_-{F_Q}\ar[r]^-{\Phi^-_Q}&\D^Q\ar[d]^-{F_Q}\\
\D^{M_n,\exx}\ar[r]_-{\tau}&\D^{M_n,\exx},&&
\D^{M_n,\exx}\ar[r]_-{\tau^-}&\D^{M_n,\exx}
}
\]
commute up to natural isomorphisms.
\end{cor}
\begin{proof}
Let $i_Q\colon Q\to M_{n}$ be an embedding of the $A_n$-quiver as discussed above. Given a sink $q_1\in Q$ we denote by $Q^{(1)}=\sigma_{q_1}Q$ the reflected quiver and we choose the embedding $i_{Q^{(1)}}\colon Q^{(1)} \to M_{n}$ according to \autoref{hyp:rel-embed}. That is, the image of $i_{Q^{(1)}}$ agrees with the image of $i_Q$ with the exception of the point~$q_1$, and if $i_Q(q_1)=(k,l)$, then $i_{Q^{(1)}}(q_1)=(k-1,l)$. Let now $(q_1,\ldots,q_n)$ be an admissible sequence of sinks and define $Q^{(i)}$ inductively as $Q^{(i)}=\sigma_{q_i}Q^{(i-1)}$. Then the quiver $Q^{(n)}$ agrees with $Q$ but the embedding obtained by induction is $t\circ i_Q$. Thus, $i_Q^\ast t^\ast F_Q$ is naturally isomorphic to $\Phi^+_Q$ and we obtain $\tau F_Q\cong F_Q\Phi^+_Q$.
\end{proof}

In a similar way one defines \textbf{partial Coxeter functors} as compositions of some reflection functors. In contrast to Coxeter functors, these are strong stable equivalences between different $A_n$-quivers. 

The following definition will be justified in \S\ref{sec:field}.

\begin{defn}\label{defn:ar-transl}
Let \D be a stable derivator. The functor $\tau\colon\D^{M_n,\exx}\to\D^{M_n,\exx}$ induced by \eqref{eq:t} is the \textbf{Auslander--Reiten translation}.
\end{defn}

Thus, the Auslander--Reiten translation $\tau$ and the Coxeter functors $\Phi^+$ correspond to each other under the above equivalences.

\subsection{The fractionally Calabi--Yau property}
\label{subsec:fCY}

We show that Dynkin quivers of type $A$ enjoy an abstract fractionally Calabi--Yau property. Some of our results in this paper imply that there is an epimorphism from the \emph{spectral Picard groupoid} spanned by Dynkin quivers of type~$A$ onto the corresponding \emph{derived Picard groupoid} over arbitrary fields. The abstract fractionally Calabi--Yau property allows us to conclude that this epimorphism is actually split (see \autoref{thm:picard} and \autoref{rmk:picard}).

For convenience, we recall some related notions. Let $k$ be a field and let $\cT$ be a $k$-linear triangulated category with finite dimensional mapping vector spaces. Denoting vector space duals by $(-)^\ast$, there is the following definition (see~\cite[\S I.1]{reiten-bergh:serre} and~\cite{bondal-kapranov:serre}). An exact autoequivalence $S\colon\cT\to\cT$ is a \textbf{Serre functor}  if it comes with an isomorphism
\begin{equation} \label{eq:Serre}
\hom_\cT(x,y) \stackrel{\cong}\longrightarrow \big(\hom_\cT(y, S x)\big)^*
\end{equation}
which is natural in $x$ and $y$. The Yoneda lemma implies that Serre functors are essentially unique if they exist.

Examples of triangulated categories with a Serre functor (see \cite[Examples 3.2]{bondal-kapranov:serre}) which are of interest to us here are bounded derived categories of finitely generated modules over finite dimensional algebras of finite global dimension (\cite[\S3.6]{happel:fd-algebra}, \cite{krause-le:AR}). In that case, it can be shown that Serre functors are given by \emph{derived Nakayama functors}, and we shall see in \S\ref{sec:nakayama} and \S\ref{sec:universal-tilting} that for Dynkin quivers of type~$A$ this extends to arbitrary stable derivators if we use \emph{canonical Nakayama functors} instead. In the classical context, Serre functors and Auslander--Reiten translations agree up to a suspension; see~\cite[\S I.2]{reiten-bergh:serre}. This motivates the following definition.

\begin{defn}\label{defn:Serre}
Let \D be a stable derivator and let $Q$ be an $A_n$-quiver. The functors
 \[
S=\Sigma\Phi^+\cong\Phi^+\Sigma\colon\D^Q\to\D^Q\quad\text{and}\quad
S=\Sigma\tau\cong\tau\Sigma\colon\D^{M_n,\exx}\to\D^{M_n,\exx}
\]
are called \textbf{Serre functors}. 
\end{defn}

This terminology will be justified in \S\ref{sec:field} but also by \autoref{thm:U_n-Serre} which can be considered as a spectrum level version of Serre duality~\eqref{eq:Serre}. Here we only collect a few immediate facts about Serre functors.

\begin{lem}\label{lem:Serre-match}
Let \D be a stable derivator and let $Q$ be an~$A_n$-quiver. There are natural isomorphisms
$S F_Q\cong F_Q S$ and $S i_Q^\ast \cong i_Q^\ast S$.
\end{lem}
\begin{proof}
Using the equivalences $(F_Q,i_Q^\ast)\colon\D^Q\rightleftarrows\D^{M_n,\exx}$ of \autoref{thm:AR-independent}, it suffices to construct a natural isomorphism $i_Q^\ast S F_Q\cong S$. By \autoref{defn:Serre} and \autoref{cor:phi-tau} we deduce $i_Q^\ast S F_Q= i_Q^\ast \Sigma\tau F_Q\cong \Sigma i_Q^\ast \tau F_Q\cong \Sigma \Phi^+= S$.
\end{proof}

The following facts about the Serre functors will be useful in~\S\ref{sec:nakayama} and~\S\ref{sec:universal-tilting}. At the level of the mesh category the Serre functors correspond to the functors
\begin{equation}\label{eq:s}
s=ft=tf\colon M_n\to M_n\colon (k,l)\mapsto (k+l-1,n+1-l).
\end{equation}
Given an $A_n$-quiver~$Q$, we consider the composition
\begin{equation}\label{eq:s_Q}
s_Q=si_Q=fti_Q\colon Q\to M_n.
\end{equation}

\begin{lem}\label{lem:Serre-shift}
Let \D be a stable derivator and let $Q$ be an $A_n$-quiver. There are natural isomorphisms 
\[
S\cong s^\ast\colon\D^{M_n,\exx}\to\D^{M_n,\exx}\qquad \text{and}\qquad S\cong s_Q^\ast F_Q\colon\D^Q\to\D^Q.
\]
\end{lem}
\begin{proof}
The first statement is immediate from the definitions and \autoref{cor:susp} while the second statement follows from the first one together with \autoref{lem:Serre-match}.
\end{proof}

In the special case of the linear $A_n$-quiver $\A{n}$ the formulas \eqref{eq:i}, \eqref{eq:f}, and~\eqref{eq:t} immediately imply that we have
\begin{equation}
s_{\A{n}}\colon\A{n}\to M_n\colon l\mapsto (l-1,n+1-l).
\label{eq:sAn}
\end{equation}
We illustrate these definitions by an example which we take up again in \S\ref{sec:nakayama}. This example is a direct consequence of \autoref{lem:Serre-shift}.

\begin{eg}\label{eg:Serre}
Let \D be a stable derivator and let us consider the linearly oriented $A_3$-quiver $\A{3}$. The Serre functor $S\colon\D^{\A{3}}\to\D^{\A{3}}$ sends $X=(x\to y\to z)$ to the restriction of \eqref{eq:A3} with underlying diagram $(z\to v\to w)$. 
\end{eg}

Let us recall that a $k$-linear triangulated category~$\cT$ with a Serre functor $S$ is \textbf{fractionally Calabi--Yau} if $\Sigma^{m_1}\cong S^{m_2}$ for certain integers $m_1 \in \lZ, m_2 > 0$ (\cite[\S8.2]{keller:orbit} or~\cite{roosmalen:CY}). This is also referred to by saying that $\cT$ has \textbf{Calabi--Yau dimension}~$\frac {m_1}{m_2}$. It is possible that $m_1$ and $m_2$ are not coprime, but the common factor cannot be canceled~\cite[Example 5.3]{roosmalen:CY}. In fact, we see this phenomenon also in our computations for quivers of type $A_3$. \autoref{cor:susp} and \autoref{lem:Serre-shift} imply that $\Sigma^2 \cong S^4$ since one directly checks the equality $f^2 = (tf)^4$ for the underlying symmetries of $M_3$ defined in \eqref{eq:f} and \eqref{eq:t}. However, $f \ne (tf)^2$, and in general $\Sigma \not\cong S^2$ as autoequivalences of $\D^{M_3,\exx}$ for a stable derivator $\D$ (see \autoref{rmk:picard-faithful} and \S\ref{sec:field}). This underlines that the fraction $\frac {m_1}{m_2}$ is to be considered as a formal expression.

Generalizing the above considerations for Dynkin type $A_3$, we obtain the following abstract versions of the fractionally Calabi--Yau property.

\begin{thm}\label{thm:frac-CY}
If \D is a stable derivator, then $\Sigma^{n-1}\colon\D^{M_n,\exx}\to\D^{M_n,\exx}$ and $S^{n+1}\colon\D^{M_n,\exx}\to\D^{M_n,\exx}$ are naturally isomorphic.
\end{thm}
\begin{proof}
The morphism $S^{n+1}=(\Sigma\tau)^{n+1}\cong\Sigma^{n+1}\tau^{n+1}$ is naturally isomorphic to $\Sigma ^{n-1}$ if and only if $\Sigma^2\cong \tau^{-(n+1)}$. By \autoref{cor:susp} there is a natural isomorphism $\Sigma\cong f^\ast$ with $f$ defined in \eqref{eq:f}. Moreover, $\tau$ is defined as $t^\ast$ with $t$ given by \eqref{eq:t}. It follows immediately from these two defining equations that $t^{-(n+1)}$ and $f^2$ both send $(k,l)$ to $(k+n+1,l)$ which concludes the proof.
\end{proof}

\begin{cor}\label{cor:frac-CY}
Let \D be a stable derivator and let $Q$ be an $A_n$-quiver. There is a natural isomorphism
\[
\Sigma^{n-1}\cong S^{n+1}\colon\D^Q\to\D^Q.
\]
\end{cor}
\begin{proof}
By \autoref{thm:AR-independent} there is an equivalence $(F_Q,i_Q^\ast)\colon \D^Q\rightleftarrows\D^{M_n,\exx}$ which together with \autoref{thm:frac-CY} gives us a natural isomorphism
\[
i_Q^\ast \Sigma^{n-1} F_Q\cong i_Q^\ast S^{n+1} F_Q\colon \D^Q\to\D^Q.
\]
Since equivalences are exact there is an isomorphism $i_Q^\ast\Sigma^{n-1} F_Q\cong \Sigma^{n-1}$ so that \autoref{lem:Serre-match} concludes the proof. 
\end{proof}

Thus, this result tells us that $A_n$-quivers enjoy a fairly abstract fractionally Calabi--Yau property and that their Calabi--Yau dimension is $\frac{n-1}{n+1}$.

\begin{rmk}\label{rmk:rel-CY}
Keller--Scherotzke introduced a relative version of the Calabi--Yau property~\cite[\S2]{keller-scherotzke:integral} in the context of $k$-linear triangulated categories where $k$ is a commutative ring (as opposed to a field). It follows from our \autoref{thm:Serre-Nakayama} and \autoref{thm:Serre-Nakayama-indep} that $\Sigma^{n-1}\cong S^{n+1}$ holds also in that context, since the relative Serre functor is defined in \cite{keller-scherotzke:integral} precisely as the underlying functor of the Nakayama functor in the monoidal derivator $\D_k$ (see \autoref{defn:Nakayama} and \autoref{rmk:Nakayama}). In the context of derived categories of diagrams in abelian categories, related functors and relations between them have been also studied by Ladkani in \cite{ladkani:posets}.
\end{rmk}

\subsection{The group of symmetries of the mesh category}
\label{subsec:symm-mesh}

The above results offer a more down-to-earth interpretation of the fractional Calabi--Yau property: it is the defining relation of the group of symmetries of the mesh category $M_n$. Note that the group of automorphisms of $M_n$ (as depicted in~\eqref{eq:A3}), which is the same as the group of automorphisms of the repetitive quiver of $[n+1]=(0<\dots<n+1)$, has two generators, namely
\begin{enumerate}
\item the horizontal translation $t\colon M_n \to M_n$ given by~\eqref{eq:t} and
\item the reflection with respect to the horizontal axis, possibly combined with a horizontal translation. Here we can take $f\colon M_n\to M_n$ defined in~\eqref{eq:f}.
\end{enumerate}

As we have seen in the proof of \autoref{thm:frac-CY}, the automorphisms $f$ and $t$ commute and satisfy the relation
\begin{equation} \label{eq:Aut(M_n)-def}
f^2 = t^{-(n+1)}.
\end{equation}
It is straightforward to check that these are the only relations between $f$ and $t$. Hence the (necessarily abelian) group $G_n$ of automorphisms of $M_n$ is given by $G_n = \langle f,t \mid ft=tf, f^2 = t^{-(n+1)} \rangle$ and as an abstract group by
\begin{itemize}
\item $G_n \cong \lZ \oplus \lZ/2\lZ$ if $n$ is odd and by
\item $G_n \cong \lZ$ if $n$ is even. In this case, $G_n$ is generated by $ft^{\frac{n}{2}}$.
\end{itemize}

To formalize the induced symmetries on $\D^{M_n,\exx}$ and $\D^Q$ for an $A_n$-quiver $Q$ we make the following definitions. Let $G$ be a discrete group and let $BG$ be the $2$-category obtained by considering the groupoid with one object associated to $G$ as a $2$-category with identity $2$-cells only. A \textbf{strict (right) action} of $G$ on a derivator~\D is a $2$-functor $\alpha\colon BG\op\to\cDER$ sending the unique object to \D. If we have a pseudo-functor instead then we speak of an \textbf{action} of $G$.

Now, given a stable derivator $\D$, there is a strict action of $G_n$ on $\D^{M_n,\exx}$ which sends $f$ to $f^\ast$ and $t$ to $t^\ast$. We know from \autoref{cor:susp} and \autoref{lem:Serre-shift} that $f^*\cong\Sigma$ and $(ft)^*\cong S$. The defining relation~\eqref{eq:Aut(M_n)-def} precisely translates to the fractional Calabi--Yau property $\Sigma^{n-1}\cong S^{n+1}$. 

A conjugation of this strict action with the equivalences of \autoref{thm:AR-independent} yields the following proposition, where we use a presentation of $G_n$ in terms of $f$ and $s=ft$ rather than $f$ and $t$ as above.

\begin{prop} \label{prop:CY-action}
Let $\D$ be a stable derivator, let $Q$ be an $A_n$-quiver, and let $G_n$ be the group given by the presentation
\[ G_n = \langle f,s \mid fs=sf, f^{n-1} = s^{n+1} \rangle. \]
There is an action $\alpha$ of $G_n$ on $\D^Q$ such that $\alpha(f)\cong\Sigma$ and $\alpha(s)\cong S$.
\end{prop}

\begin{rmk} \label{rmk:picard-faithful}
For $n\geq 2$ this group action is generally faithful in the sense that there exists a stable derivator \D such that $\alpha(g)\colon \D \to \D$ is not naturally isomorphic to the identity morphism of $\D$ unless $g$ is the neutral element. In fact, we can choose $\D$ to be the derivator $\D_k$ of any field $k$ as we shall see in \S\ref{sec:field}.
\end{rmk}

\section{Monoidal derivators and tensor products of functors}
\label{sec:monoidal}

In this section we review some facts about monoidal derivators and the calculus of tensor products of functors. These notions were studied in quite some detail in \cite{gps:additivity} and we follow that reference in our notational conventions. We also recall that the derivator of spectra \cSp enjoys a certain universal property, with the important consequence that \emph{every stable derivator} is a closed module over \cSp. The content of this section is central to \S\S\ref{sec:nakayama}-\ref{sec:field}.

To begin with let us recall that the 2-category $\cDER$ of derivators is cartesian with the product $\D_1\times \D_2$ defined by $A\mapsto \D_1(A)\times\D_2(A)$. In particular, this allows us to consider pseudo-monoid objects in $\cDER$, i.e., derivators \V endowed with coherently associative and unital multiplications $\otimes\colon\V\times\V\to\V$. The unit is specified by a morphism $\lS\colon y(\bbone)\to\V$ and there are hence unit objects $\lS_A\in\V(A)$ which are compatible with restrictions up to coherent isomorphisms. Such a pseudo-monoid structure defines (and is equivalent to) a lift of $\V\colon\cCat\op\to\cCAT$ against the forgetful $2$-functor $\cMonCAT\to\cCAT$ from monoidal categories to categories.

In order to obtain a meaningful definition of monoidal structures which commute with colimits in both variables separately, one has to use \emph{external versions} of morphisms of two variables. More precisely, a morphism of two variables in the internal version is a pseudo-natural transformation with components 
\begin{equation}\label{eq:int}
\otimes_A\colon \D_1(A)\times\D_2(A)\to \D_3(A),\qquad A\in\cCat.
\end{equation}
 Using restrictions along projections and diagonal functors, one observes that such a morphism is equivalently specified by a pseudo-natural transformation with components 
\begin{equation}\label{eq:ext}
\otimes\colon\D_1(A)\times\D_2(B)\to\D_3(A\times B),\qquad A,B\in\cCat
\end{equation}
(see \cite[Theorem~3.11]{gps:additivity}). This external variant is distinguished notationally from the internal one by dropping the decoration. As pointed out in \cite[Warning~3.6]{gps:additivity}, it is this external version one has to use to express a compatibility with colimits.

Given a morphism $\otimes$ of two variables and a functor $u\colon A_1\to A_2$, then, using a variant of \eqref{eq:hoexmate1'}, one constructs canonical natural transformations
{\small
\begin{equation}\label{eq:first}
(u\times 1)_! (X \otimes Y) \to (u\times 1)_! (u^* u_! (X) \otimes Y) \stackrel{\cong}{\to}
(u\times 1)_! (u\times 1)^* (u_! (X) \otimes Y) \to u_! (X) \otimes Y
\end{equation}}
\!for $X\in\D_1(A_1), B\in\cCat,Y\in\D_2(B).$ If \eqref{eq:first} is an isomorphism for all $X,B,$ and $Y$, then we say that $\otimes$ preserves left Kan extensions along $u$ in the first variable. There is an obvious variant for the second variable and a combined notion of morphisms of two variables which preserve left Kan extensions in both variables separately. As in the case of morphisms of one variable (\cite[Prop.~2.3]{groth:ptstab}), it is enough to check that colimits are preserved.

\begin{defn}
A \textbf{monoidal derivator} is a pseudo-monoid object $(\V,\otimes,\lS)$ in \cDER such that the monoidal structure $\otimes\colon\V\times\V\to\V$ preserves left Kan extensions in both variables separately.
\end{defn}

Thus, the definition of a monoidal derivator is obtained by categorifying the notion of a monoid \emph{and} imposing a cocontinuity condition. There are obvious variants of braided or symmetric monoidal derivators and, using the extension of the notion of adjunctions of two variables to the framework of derivators given in \cite[\S8]{gps:additivity}, also of suitably closed monoidal derivators. 

A similar combination of a categorification and cocontinuity conditions leads to the notion of a \textbf{module over a monoidal derivator}. If \V is a monoidal derivator, then a $\V$-module is a derivator \D endowed with coherently associative and unital action maps $\otimes\colon\V(A)\times\D(B)\to\D(A\times B)$ which preserve left Kan extensions in both variables separately. There is also the notion of a \textbf{closed \V-module} which is to say \emph{enriched derivator} and this notion will be studied in more detail in \cite{gs:enriched}.

\begin{egs}\label{egs:monoidal}
\begin{enumerate}
\item Let $\bV$ be a complete and cocomplete category. If \bV is endowed with a monoidal structure $\otimes\colon\bV\times\bV\to\bV$ which preserves colimits in both variables separately, then the represented derivator $y(\bV)$ inherits a monoidal structure. There is a variant for complete and cocomplete categories~$\bD$ coming with colimit-preserving actions $\otimes\colon\bV\times\bD\to\bD$ (see \cite[Example~3.22]{gps:additivity}).
\item If \bV is a monoidal model category, then $\ho(\bV)$ is a monoidal derivator. There is a more general statement for Quillen adjunctions of two variables and hence also for \bV-model categories (see \cite[Example~3.23]{gps:additivity}). 
\item Let $R$ be a commutative ring and let $\Ch(R)$ be the category of unbounded chain complexes over~$R$. This category can be endowed with the projective model structure \cite[\S2.3]{hovey:model}. Together with the tensor product $\otimes_R$ it is a stable, closed symmetric monoidal model category. Thus, the derivator $\D_R$ together with the derived tensor product is a stable, closed symmetric monoidal derivator. In particular, this applies to the derivator $\D_k$ of a field~$k$.
\item The derivator \cSp of spectra together with the derived smash product is a stable, closed symmetric monoidal derivator. It can be obtained by passing to the homotopy derivator associated to any of the Quillen equivalent stable, symmetric monoidal closed model categories of spectra discussed for example in \cite{hss:symmetric,ekmm:rings,mmss:diagram}.
\end{enumerate}
\end{egs}

Besides the internal and the external variant of morphisms of two variables, there also is the \emph{canceling variant} which we discuss next. These can be thought of as a categorification of the usual tensor products of bimodules over rings and play a key role in \S\S\ref{sec:nakayama}-\ref{sec:field}. To make this variant precise, we have to recall some basics about \emph{coends} in derivators. There are various ways of defining coends in classical category theory and all of them can be extended to the framework of derivators and shown to be equivalent (see \cite[\S5]{gps:additivity} and \cite[Appendix~A]{gps:additivity}). 

As a definition we choose the approach based on twisted arrow categories which seems to be particularly convenient. The \textbf{twisted arrow category} $\tw(A)$ of $A\in\cCat$ is the category of elements of $\hom_A\colon A\op\times A\to\mathrm{Set}$. Thus, objects are morphisms in $A$ and a morphism $f_1\to f_2$ is a commutative diagram
\[
\xymatrix{
a_1\ar[r]^-{f_1}&b_1\ar[d]\\
a_2\ar[r]_-{f_2}\ar[u]&b_2,
}
\]
which is to say a $2$-sided factorization of $f_2$ through $f_1$. The category $\tw(A)$ comes with a functor 
\begin{equation}\label{eq:(s,t)}
(s,t)\colon\tw(A)\to A\op\times A
\end{equation}
given by the source and target functors. As in the case of ordinary category theory, in order to define the coend construction we need essentially the opposite of this functor, namely 
\begin{equation}\label{eq:tw-op}
(t\op,s\op)\colon\tw(A)\op\stackrel{(s,t)\op}{\to}(A\op\times A)\op\cong A\op\times A.
\end{equation}

\begin{defn}\label{defn:coend}
Let $\D$ be a derivator and $A\in\cCat$. The \textbf{coend} $\int^A\colon\D(A\op\times A)\to \D(\bbone)$ is given by
\begin{equation}\label{eq:coend}
\int^A\colon\D(A\op\times A)\stackrel{(t\op,s\op)^\ast}{\to}\D(\tw(A)\op)\stackrel{(\pi_{\tw(A)\op})_!}{\to}\D(\bbone).
\end{equation}
\end{defn}

In the case of a represented derivator $y(\bD)$, this reduces to a formula in \cite{maclane} and hence reproduces the usual notion of a coend in a (complete and) cocomplete category. There is also a different classical description of coends as certain coequalizers. In homotopical frameworks, coequalizers have to be replaced by geometric realizations of simplicial bar constructions and for an extension of this reformulation to derivators we refer to \cite[Appendix~A]{gps:additivity}.

Now, if we are given a morphism of two variables $\otimes\colon\D_1\times\D_2\to\D_3$ of derivators, then the \emph{canceling version} is $X\otimes_{[A]} Y=\int^AX\otimes Y$ for $X\in\D_1(A\op)$ and $Y\in\D_2(A)$. Thus, as a functor $\otimes_{[A]}$ is defined as
\begin{equation}\label{eq:cancel}
\otimes_{[A]}\colon\D_1(A\op)\times \D_2(A)\stackrel{\otimes}{\to}\D_3(A\op\times A)\stackrel{\int^A}{\to}\D_3(\bbone),
\end{equation}
and we refer to it as the \textbf{(canceling) tensor product of functors}. Note that the subscript $[A]$ is added in order to indicate notationally which `category is canceled'. 

\begin{rmk}\label{rmk:variants}
Given a morphism of two variables $\otimes\colon\D_1\times\D_2\to\D_3$, its internal \eqref{eq:int}, external \eqref{eq:ext}, and canceling versions \eqref{eq:cancel} can be suitably combined. In particular, one obtains parametrized versions of these and, similarly, also of coends \eqref{eq:coend}, thus leading to morphisms of derivators, namely
\begin{enumerate}
\item parametrized internal tensor products $\otimes_A\colon\D_1^A\times\D_2^A\to\D_3^A$,
\item parametrized external tensor products $\otimes\colon\D_1^A\times\D_2^B\to\D_3^{A\times B}$,
\item parametrized coends $\int^A\colon\D^{A\op}\times\D^A\to\D$, and 
\item parametrized canceling tensor products $\otimes_{[A]}\colon\D_1^{A\op}\times\D_2^A\to\D_3$.
\end{enumerate}
It turns out that if $\otimes\colon\D_1\times\D_2\to\D_3$ is a left adjoint of two variables, then the same is true for its internal, external, and canceling variants (see \cite[\S8]{gps:additivity}). 
\end{rmk}

\begin{lem} \label{lem:tensor-opp}
Let $\V$ be a symmetric monoidal derivator, let $X\in\V(A\op)$, and let $Y\in\V(A)=\V((A\op)\op)$. The symmetry constraint of \V induces a canonical isomorphism
\[
X\otimes_{[A]} Y\cong Y \otimes_{[A\op]} X.
\]
\end{lem}
\begin{proof}
We recall that the symmetry constraint yields, in particular, coherent isomorphisms $s\colon c^\ast(X\otimes Y)\cong Y\otimes X$ where $c$ denotes the canonical symmetry isomorphism $c\colon A\times A\op\cong A\op\times A$ in $\cCat$. Note next that the passage to formal opposite morphisms in $A$ induces an isomorphism of categories $\mathrm{op}\colon\tw(A)\to\tw(A\op)$. This isomorphism interchanges sources and targets and hence makes the diagram
\[
\xymatrix{
\tw(A)\op \ar[r]^-{\mathrm{op}}_-\cong \ar[d]_{(t_A\op,s_A\op)} & \tw(A\op)\op \ar[d]^{(t_{A\op}\op,s_{A\op}\op)}\\
A\op \times A \ar@{<-}[r]_-c & A \times A\op 
}
\]
commutative. Since isomorphisms are homotopy final (\autoref{egs:htpy}), we obtain natural isomorphisms
\begin{align}
X\otimes_{[A]} Y &= (\pi_{\tw(A)\op})_!(t_A\op,s_A\op)^\ast(X \otimes Y) \\
&= (\pi_{\tw(A)\op})_!\mathrm{op}^\ast(t_{A\op}\op,s_{A\op}\op)^\ast c^\ast (X \otimes Y) \\
&\cong (\pi_{\tw(A)\op})_!\mathrm{op}^\ast(t_{A\op}\op,s_{A\op}\op)^\ast(Y \otimes X) \\
&\cong (\pi_{\tw(A\op)\op})_!(t_{A\op}\op,s_{A\op}\op)^\ast(Y \otimes X) \\
&= Y \otimes_{[A\op]} X,
\end{align}
concluding the proof.
\end{proof}

The canceling tensor product is also associative and unital in the following precise sense (see \cite[Theorem~5.9]{gps:additivity}).

\begin{thm}\label{thm:bicategory}
  If \V is a monoidal derivator, then there is a bicategory $\cProf(\V)$ described as follows:
  \begin{itemize}
  \item Its objects are small categories.
  \item Its hom-category from $A$ to $B$ is $\V(A\times B\op)$.
  \item Its composition functors are the external-canceling tensor products
    \[ \otimes_{[B]} \colon \V(A\times B\op) \times \V(B\times C\op) \too \V(A\times C\op). \]
  \item The identity 1-cell of a small category $B$ is
    \begin{equation}
      \lI_B\;=\;(t,s)_! \lS_{\tw(B)} \;\cong\; (t,s)_! \pi_{\tw(B)}^* \lS_{\bbone} \; \in \V(B\times B\op).\label{eq:unit}
    \end{equation}
  \end{itemize}
\end{thm}

We refer to $\cProf(\V)$ as the \textbf{bicategory of profunctors} in \V. Let us emphasize that the identity 1-cells $\lI_B\in\V(B\times B\op)$, also called \textbf{identity profunctors}, are obtained from the monoidal unit $\lS$ as follows. Associated to the twisted arrow category $\tw(B)$, there is the functor \eqref{eq:(s,t)}
and $\lI_B$ is defined as $(t,s)_!\lS_{\tw(B)}\cong(t,s)_!\pi_{\tw(B)}^\ast\lS_\bbone$. Note that while in the definition of the coend we use the opposite of the twisted arrow category \eqref{eq:tw-op}, here we are using the twisted arrow category itself.

In the represented case, $\cProf(y(\bV))$ reduces to the ordinary bicategory of profunctors and $\lI_B$ specializes to the usual identity profunctors (see \cite[Remark~5.10]{gps:additivity}). We will say a bit more about the profunctors $\lI_B$ in \S\ref{sec:nakayama} where they show up prominently in the study of canonical Nakayama functors (in particular, for $A_n$-quivers). In \S\S\ref{sec:admissible}-\ref{sec:field} we will see how starting from these identity profunctors we obtain our universal tilting modules and universal constructors.

We conclude this review section by recalling the universal property of the derivator of spectra and the resulting canonical actions on stable derivators; related to this see \cite{heller:htpythies,heller:stable,cisinski:derived-kan,tabuada:universal-invariants,cisinski-tabuada:non-connective,cisinski-tabuada:non-commutative}. We begin by recalling that the homotopy derivator of spaces~\cS is freely generated under colimits by the singleton~$\Delta^0$ \cite[Theorem~3.24]{cisinski:derived-kan}. In more detail, let $F\colon\cS\to\D$ be a colimit preserving morphism of derivators. The evaluation of the underlying functor $\cS(\bbone)\to\D(\bbone)$ at $\Delta^0$ induces an equivalence of categories
\[
\Hom_!(\cS,\D)\stackrel{\sim}{\to}\D(\bbone)\colon F\mapsto F_\bbone(\Delta^0).
\]
Here, $\Hom_!(-,-)$ denotes the categories of colimit preserving morphisms and all natural transformations. The derivator $\cSp$ of spectra enjoys a similar universal property in the framework of stable derivators; it is freely generated under colimits by the sphere spectrum~$\lS$ \cite[Theorem~A.11]{cisinski-tabuada:non-connective}. This can be used to show that every stable derivator is canonically a closed module over the derivator of spectra \cite[Appendix~A.3]{cisinski-tabuada:non-connective}. In fact, this action is characterized by the fact that it preserves colimits in both variables separately and $\lS \otimes - \cong \id_\D$.

Thus, if \D is a stable derivator, then there is a canonical action $\otimes\colon\cSp\times\D\to\D$. Moreover, this action is a left adjoint of two variables, i.e., every component $\otimes\colon\Sp(A)\times\D(B)\to\D(A\times B)$ is a left adjoint of two variables and the external version $\otimes\colon\cSp\times\D\to\D$ preserves left Kan extensions in both variables separately. Thus, following the notational conventions of \cite{gps:additivity}, for $X\in\Sp(A),Y\in\D(B),Z\in\D(A\times B)$ there are natural isomorphisms
\begin{equation}
  \D(A\times B)\big(X\otimes Y,Z\big) \;\cong\;
  \cSp(A)\big(X, \homr{B}{Y}{Z}\big) \;\cong\;
  \D(B)\big(Y,\homl{A}{X}{Z}\big)\label{eq:canonical}
\end{equation}
for certain functors 
\[
\homr{B}{}{}\colon\D(B)\op\times\D(A\times B)\to\cSp(A)\quad\text{and}\quad
\homl{A}{}{}\colon\D(A\times B)\times\cSp(A)\op\to\D(B).
\]
As indicated by the notation, it turns out that there are morphisms of derivators of two variables
\begin{equation} \label{eq:Sp-cotensor}
\rhd\colon\D\op\times\D\to\cSp\quad\text{and}\quad\lhd\colon\D\times\cSp\op\to\D
\end{equation}
such that the above functors $\homr{B}{}{}$ and $\homl{A}{}{}$ are obtained by passing to suitable canceling versions. In these cases however, as in the case of ordinary category theory, one has to use \emph{ends} instead of coends. More generally, if $\otimes\colon\D_1\times\D_2\to\D_3$ is a left adjoint of two variables, then in the background there are three morphisms of derivators
\begin{equation}\label{eq:two-variables}
\otimes\colon\D_1\times\D_2\to\D_3,\quad \rhd\colon\D_2\op\times\D_3\to\D_1,\quad\text{and}\quad
\lhd\colon\D_3\times\D_1\op\to\D_2.
\end{equation}
For more details see \cite[\S\S 8-9]{gps:additivity}.

There is a subtle point here which will be useful in later computations. Considering the opposite morphism of the cotensors in \eqref{eq:Sp-cotensor}, we obtain a morphism
\begin{equation} \label{eq:opp-action}
\cSp\times\D\op\cong\D\op\times\cSp\stackrel{\lhd\op}{\to}\D\op ,
\end{equation}
which is again a left adjoint in two variables. This is in fact precisely the cycled adjunction of $\otimes$ in the sense of \cite[\S11]{gps:additivity}. Now $\D\op$ has two potentially different closed $\cSp$-module structures: the canonical one since $\D\op$ is again stable and the one from~\eqref{eq:opp-action} induced by the canonical action on $\D$. As a matter of fact, these two actions turn out to be isomorphic.

\begin{prop}\label{prop:stable-frames-opp}
Let $\D$ be a stable derivator. The canonical action of $\cSp$ on~$\D\op$ is naturally isomorphic to the action \eqref{eq:opp-action} induced by the canonical action on \D.
\end{prop}
\begin{proof}
This follows from the above characterization of the canonical action, given that
\begin{enumerate}
\item the external versions of both $\otimes$ and $\lhd\op$ preserve left Kan extensions in both variables separately and
\item both $\lS \otimes -\colon \D\op \to \D\op$ and $(- \lhd \lS)\op\colon \D\op \to \D\op$ are (canonically) isomorphic to the identity morphism on $\D\op$. \qedhere
\end{enumerate}
\end{proof}

\section{Serre functors as canonical Nakayama functors}
\label{sec:nakayama}

As mentioned in \S\ref{subsec:fCY}, one class of examples of triangulated categories with a Serre functor arises from representation theory. In fact, let $k$ be a field and let $A$ be a finite dimensional $k$-algebra of finite global dimension. The $k$-linear dual $A^*$ is canonically an $A$-$A$-bimodule yielding the Nakayama functor $\nu=A^\ast\otimes_A-$. The Serre functor on the bounded derived category $D^b(A)$ can be constructed as the derived tensor product $S = A^* \Lotimes_A -$; see~\cite[\S3.6]{happel:fd-algebra}. In this section we show that for every stable derivator \D the Serre functor $S\colon\D^{\A{n}}\to\D^{\A{n}}$ can be obtained as a \emph{canonical Nakayama functor} (\autoref{thm:Serre-Nakayama}) which we define as the canceling tensor product with a suitable spectral bimodule. It is rather formal to extend this to arbitrary $A_n$-quivers as we will see in \S\ref{subsec:univ-Coxeter}.

We begin by mimicking the definition of the $k$-linear dual of an algebra. Let \V be a monoidal derivator and let $A\in\cCat$. The identity profunctor 
\[
\lI_A\in\V(A\times A\op)\op\cong\V\op(A\op\times A)
\]
is the analogue of the regular bimodule associated to an algebra. The analogue of the $k$-linear dual is obtained as follows where we use the notation as in \eqref{eq:two-variables}.

\begin{defn} \label{defn:dualizing}
Let $\V$ be a closed monoidal derivator and let $A\in\cCat$. The \textbf{canonical duality bimodule} $D_A\in\V(A\times A\op)$ is the image of the identity profunctor~$\lI_A$ under
\[
\V\op(A\op\times A)\stackrel{-\rhd\lS}{\to}\V(A\op\times A)\stackrel{\cong}{\to}\V(A\times A\op).
\]
\end{defn}

Still using the notation of \eqref{eq:(s,t)} there is the following lemma.

\begin{lem}\label{lem:dual}
Let \V be a closed monoidal derivator and let $A\in\cCat$. There is a canonical isomorphism $D_A\cong(s\op,t\op)_\ast\lS_{\tw(A)\op}$.
\end{lem}
\begin{proof}
The morphism $-\rhd\lS\colon\V\op\to\V$ sends left Kan extensions in \V to right Kan extensions in \V, up to canonical isomorphisms. The definition of $\lI_A=(t,s)_!\lS_{\tw(A)}$, the calculation of right Kan extensions in opposite derivators, and the natural isomorphism $\lS\rhd-\cong\id$ imply that there are canonical isomorphisms
\[
\lI_A\rhd\lS\cong (t\op,s\op)_\ast(\lS_{\tw(A)}\rhd\lS)\cong (t\op,s\op)_\ast\lS_{\tw(A)\op}.
\]
Swapping $A$ and $A\op$ as in \autoref{defn:dualizing} then yields $D_A\cong (s\op,t\op)_\ast\lS_{\tw(A)\op}$.
\end{proof}

\begin{defn}\label{defn:Nakayama}
Let $\V$ be a closed monoidal derivator and let $\D$ be a \V-module. The \textbf{Nakayama functor} associated to $A\in\cCat$ is  
\[
D_A\otimes_{[A]}-\colon \D^A\to \D^A.
\]
\end{defn}

As recalled in \S\ref{sec:monoidal}, every stable derivator comes with a canonical action of the derivator of spectra. We refer to the associated Nakayama functors as \textbf{canonical Nakayama functors}. The goal of this section is to show that for linear $A_n$-quivers canonical Nakayama functors are naturally isomorphic to Serre functors. In order to obtain a better understanding of 
\[
D_n=D_{\A{n}}\in\V(\A{n}\times \A{n}\op)
\]
we begin by spelling out in more detail the bimodules $\lI_P,D_P$ for posets~$P$. 

\begin{lem}\label{lem:poset}
Let $P$ be a poset and let $\Delta_P=\{(p,p)\mid p\in P\}$ be the diagonal considered as a full subcategory of $P\times P\op$.
\begin{enumerate}
\item The functor $(t,s)\colon\tw(P)\to P\times P\op$ induces an isomorphism onto the cosieve generated by $\Delta_P\subseteq P\times P\op$. In particular, $\lI_P\in\V(P\times P\op)$ is obtained from $\lS_{\tw(P)}$ by left extension by zero along $(t,s)$.
\item The functor $(s\op,t\op)\colon\tw(P)\op\to P\times P\op$ induces an isomorphism onto the sieve generated by $\Delta_P\subseteq P\times P\op$. In particular, $D_P\in\V(P\times P\op)$ is obtained from $\lS_{\tw(P)\op}$ by right extension by zero along $(s\op,t\op)$.
\end{enumerate}
\end{lem}
\begin{proof}
The functors $(t,s),(s\op,t\op)$ are clearly injective on objects, fully faithful, and the images are as stated. By \autoref{lem:extbyzero} we know that $(t,s)_!$ is left extension by zero while $(s\op,t\op)_\ast$ is right extension by zero. The remaining statements follow from the definition of $\lI_P$ and from \autoref{lem:dual}, respectively.
\end{proof}

Thus $\lI_P\in\V(P\times P\op)$ is constant with value the monoidal unit $\lS\in\V(\bbone)$ on the cosieve generated by the diagonal $\Delta_P$ and it vanishes on the complement. Similarly, the canonical dual $D_P\in\V(P\times P\op)$ is constantly $\lS$ on the sieve generated by the diagonal $\Delta_P$ and vanishes elsewhere. 

\begin{eg}\label{eg:D2}
Let us consider the poset $P=\A{3}$. By the above discussion, the identity profunctor $\lI_3=\lI_{\A{3}}\in\V(\A{3}\times \A{3}\op)$ looks as shown in the diagram on the left, while the canonical dual $D_3\in\V(\A{3}\times \A{3}\op)$ is indicated in the diagram on the right. The drawing convention is that the horizontal direction is the first coordinate, while the vertical one is given by the second coordinate.
\begin{equation}\label{eq:D2}
\vcenter{
\xymatrix@=1.0em{
\lS\ar[r]&\lS\ar[r]&\lS&&& 
\lS\ar[r]&0\ar[r]&0\\
0\ar[u]\ar[r]&\lS\ar[u]\ar[r]&\lS\ar[u] && &
\lS\ar[u]\ar[r]&\lS\ar[u]\ar[r]&0\ar[u]\\
0\ar[u]\ar[r]&0\ar[r]\ar[u]&\lS\ar[u] && &
\lS\ar[u]\ar[r]&\lS\ar[r]\ar[u]&\lS\ar[u]
}
}
\end{equation}
\end{eg}

Before we attack the proof of \autoref{thm:Serre-Nakayama} we sketch what happens in the case of $n=3$. This particular case motivates certain constructions in the proof of \autoref{thm:Serre-Nakayama}.

\begin{eg}\label{eq:Serre-Nakayama}
Let \D be a stable derivator considered with the canonical action of $\cSp$, and let $X=(x\to y\to z)\in\D(\A{3})$. We want to show that there are isomorphisms $(D_3\otimes_{[\A{3}]} X)_k\cong S(X)_k$ in the underlying category $\D(\bbone)$ for all $k\in\A{3}$. Here, $S$ of course denotes the Serre functor (\autoref{defn:Serre}).

By definition of the canceling tensor product we have to calculate $D_3\otimes X$ which is an object of $\D(\A{3}\times\A{3}\op\times\A{3})$. Let us write $(k,l,m)$ for a typical object in $\A{3}\times\A{3}\op\times\A{3}$. Using \autoref{eg:D2} we know that the underlying diagrams of $(D_3\otimes X)_{(k,-,-)}\in\D(\A{3}\op\times\A{3}), k\in\A{3},$ look as shown in \autoref{fig:external} where we draw the $\A{3}\op$-coordinate horizontally and the remaining $\A{3}$-coordinate vertically.

\begin{figure}[h]
 \centering
\[
\xymatrix@=1.0em{
x\ar[d]&x\ar[l]\ar[d]&x\ar[l]\ar[d] && 
0\ar[d]&x\ar[l]\ar[d]&x\ar[l]\ar[d] && 
0\ar[d]&0\ar[l]\ar[d]&x\ar[l]\ar[d]\\
y\ar[d]&y\ar[l]\ar[d]&y\ar[l]\ar[d] && 
0\ar[d]&y\ar[l]\ar[d]&y\ar[l]\ar[d] && 
0\ar[d]&0\ar[l]\ar[d]&y\ar[l]\ar[d]\\
z&z\ar[l]&z\ar[l] && 
0&z\ar[l]&z\ar[l] && 
0&0\ar[l]&z\ar[l]
}
\]
\caption{$(D_3\otimes X)_{(k,-,-)}$ for $k=1,2,$ and $3$.}
\label{fig:external}
\end{figure}

In order to pass to the canceling tensor products $(D_3\otimes_{[\A{3}]}X)_k$, it is a consequence of \autoref{lem:poset} that for each $k$ we have to restrict $(D_3\otimes X)_{(k,-,-)}$ to the sieves generated by the diagonal and then pass to the colimit. Obviously, these colimits can be calculated by first left Kan extending the respective restrictions to $\A{3}\op\times\A{3}$ and then evaluating at the terminal object $(1,3)$ (see \autoref{egs:htpy}). Sticking again to the decoration of the objects as in \eqref{eq:A3}, a repeated application of \autoref{lem:detection} implies that the respective left Kan extensions look as in \autoref{fig:left}. 

\begin{figure}[h]
\centering
\[
\xymatrix@=1.0em{
x\ar[d]&x\ar[l]\ar[d]&x\ar[l]\ar[d] && 
0\ar[d]&x\ar[l]\ar[d]&x\ar[l]\ar[d] && 
0\ar[d]&0\ar[l]\ar[d]&x\ar[l]\ar[d]\\
y\ar[d]&y\ar[l]\ar[d]&y\ar[l]\ar[d] && 
u\ar[d]&y\ar[l]\ar[d]&y\ar[l]\ar[d] && 
0\ar[d]&0\ar[l]\ar[d]&y\ar[l]\ar[d]\\
z&z\ar[l]&z\ar[l] && 
v&z\ar[l]&z\ar[l] && 
w&w\ar[l]&z\ar[l]
}
\]
\caption{Towards $(D_3\otimes_{[\A{3}]} X)_k$ for $k=1,2,$ and $3$.}
\label{fig:left}
\end{figure}

If we now evaluate at the terminal object $(1,3)\in\A{3}\op\times\A{3}$ for $k=1,2,3$, then we deduce that the underlying incoherent diagram of $D_3\otimes_{[\A{3}]} X$ is of the form
\[
z\to v\to w.
\]
Comparing this to \autoref{eg:Serre}, we see that it matches at least objectwise with our earlier description of the Serre functor.
\end{eg}

\begin{thm}\label{thm:Serre-Nakayama}
Let \D be a stable derivator and let $n\geq 1$. The Serre functor and the canonical Nakayama functor on $\D^{\A{n}}$ are naturally isomorphic,
\[
S\cong D_n\otimes_{[\A{n}]}-\colon\D^{\A{n}}\to\D^{\A{n}}.
\]
\end{thm}
\begin{proof}
Unraveling the definition of the canonical Nakayama functor, we see that it is given by
\[
\D^{\A{n}}\stackrel{D_n\otimes-}{\to}\D^{\A{n}\times\A{n}\op\times\A{n}}\stackrel{(\id\times (t\op,s\op))^\ast}{\longrightarrow} \D^{\A{n}\times\tw(\A{n})\op}\stackrel{(\pi_{\tw(\A{n})\op})_!}{\to}\D^{\A{n}}. 
\]
Here $(t\op,s\op)\colon\tw(\A{n})\op\to \A{n}\op\times\A{n}$ is essentially the opposite of \eqref{eq:(s,t)} and the functor $\pi_{\tw(\A{n})\op}\colon\A{n}\times\tw(\A{n})\op\to\A{n}$ is the projection away from $\tw(\A{n})\op$. We are aiming for a different description of $D_n\otimes_{[\A{n}]} X$ for $X\in\D^{\A{n}}$ which shows the desired relation to the Serre functor. The final step uses \eqref{eq:sAn} and \autoref{lem:Serre-shift}. 

Since $\A{n}$ is a poset, we can, by \autoref{lem:poset}, consider $(\id\times(t\op,s\op))^\ast (D_n\otimes X)$ as the restriction of $D_n\otimes X\in\D^{\A{n}\times\A{n}\op\times\A{n}}$ to the full subcategory of $\A{n}\times\A{n}\op\times\A{n}$ given by
\[
\{(k,l,m)\in\A{n}\times\A{n}\op\times\A{n}\mid l\geq m\}.
\]
In fact, $j\colon C=\{(l,m)\mid l\geq m\}\subseteq\A{n}\op\times\A{n}$ is the sieve generated by the diagonal of $\A{n}$. The functor $\id\times(t\op,s\op)$ can thus be written as the horizontal composition in the following commutative diagram:
\[
\xymatrix{
\A{n}\times\tw(\A{n})\op\ar[r]^-\cong\ar[rd]_-{\pi_{\tw(\A{n})\op}}&\A{n}\times C\ar[d]^-{\pi_C}\ar[r]^-{\id\times j}&\A{n}\times\A{n}\op\times\A{n}\ar[dl]^-{\;\;\pi_{\A{n}\op\times\A{n}}}\\
&\A{n}&
}
\]
Since $(1,n)\in\A{n}\op\times\A{n}$ is a terminal object, the left Kan extension functor along $\A{n}\op\times\A{n}\to\bbone$ is naturally isomorphic to the evaluation functor $(1,n)^\ast$ (this is an instance of the cofinality of right adjoints given by \autoref{egs:htpy}). Thus we obtain canonical isomorphisms
\begin{align}
D_n\otimes_{[\A{n}]} X&= (\pi_{\tw(\A{n})\op})_!(\id\times(t\op,s\op))^\ast (D_n\otimes X)\\
&\cong (\pi_C)_!(\id\times j)^\ast (D_n\otimes X)\\
&\cong (\pi_{\A{n}\op\times\A{n}})_!(\id\times j)_!(\id\times j)^\ast (D_n\otimes X)\\
&\cong (\id\times(1,n))^\ast(\id\times j)_!(\id\times j)^\ast (D_n\otimes X).
\end{align}

The next step consists of describing $(\id\times j)^\ast(D_n\otimes X)$ as the restriction of a diagram in $\D^{M_n,\exx}$. By \autoref{lem:poset} we know that $D_n\in\Sp(\A{n}\times\A{n}\op)$ is obtained by right extension by zero from $\lS_{\tw(\A{n})\op}$ along the inclusion $\tw(\A{n})\op\to\A{n}\times\A{n}\op$. Let $\chi\colon \A{n}\times\A{n}\op\to[1]$ be the unique functor which sends $\tw(\A{n})\op\subseteq\A{n}\times\A{n}\op$ constantly to $0$ and the remaining objects to $1$. Then \autoref{lem:poset} implies that there is an isomorphism $D_n\cong\chi^\ast 0_\ast(\lS_{\bbone})$ where $0\colon\bbone\to[1]$ picks the object $0\in[1]$. Since the action $\otimes\colon\Sp\times\D\to\D$ is exact, we obtain natural isomorphisms 
\[
D_n\otimes X\cong (\chi^\ast 0_\ast\lS_\bbone)\otimes X\cong (\chi\times\id_{\A{n}})^\ast(0\times\id_{\A{n}})_\ast (X).
\]
Using implicitly the obvious isomorphism between $[1]\times\A{n}$ and the full subcategory $K\subseteq M_n$ spanned by the objects $(0,m),(m,0),1\leq m\leq n$, we obtain a natural isomorphism between $(0\times\id_{\A{n}})_\ast (X)$ and the restriction of $FX$ to $K$ ($F$ is again the equivalence of \autoref{thm:AR}). We define the functor $u\colon \A{n}\times C \to K\subseteq M_n$ by
\[
u(k,l,m) =
\begin{cases}
(0,m) & \textrm{if }l \geq k \\
(m,0) & \textrm{if } l < k. \\
\end{cases}
\]
The reader checks that this functor is well-defined. Summarizing this part, there is a natural isomorphism $(\id\times j)^\ast(D_n\otimes X)\cong u^\ast FX$, and, combined with the first part, this yields a natural isomorphism
\begin{align}
D_n\otimes_{[\A{n}]} X&\cong (\id\times(1,n))^\ast(\id\times j)_!(\id\times j)^\ast (D_n\otimes X)\\
&\cong (\id\times(1,n))^\ast(\id\times j)_!u^\ast FX.
\end{align}

Next, we show that $(\id\times j)_!u^\ast FX$ can be obtained from $FX$ by restriction. Since Kan extensions are pointwise (\cite[Corollary~3.14]{groth:ptstab}) a coherent diagram $Y\in\D^{\A{n}\times\A{n}\op\times\A{n}}$ lies in the essential image of $(\id\times j)_!$ if and only if $Y_k=Y_{(k,-,-)}$ lies in the essential image of $j_!\colon\D^C\to\D^{\A{n}\op\times\A{n}}$ for all $1\leq k\leq n$. Now, the inclusion $j\colon C\to\A{n}\op\times\A{n}$ factors over intermediate categories in a way that each step consists of adding one new square only. More formally, this is achieved by adding the objects $(l,m)$ with $l<m$ by a nested induction in the increasing $m$-direction and in the decreasing $l$-direction. A repeated application of \autoref{lem:detection} together with \autoref{egs:htpy} shows that $Y_k$ lies in the essential image of $j_!$ if and only if $Y_k$ makes all these new squares cocartesian. As an upshot, $Y\in\D^{\A{n}\times\A{n}\op\times\A{n}}$ lies in the essential image of $(\id\times j)_!$ if and only if $Y_k,1\leq k\leq n,$ makes all the new squares cocartesian, and, using the fully faithfulness of $\id\times j$, this is the case if and only if the counit $(\id\times j)_!(\id\times j)^\ast Y\to Y$ is an isomorphism (again by \autoref{egs:htpy}).

One observes that the above functor $u\colon\A{n}\times C\to M_n$ can be extended to the functor $v\colon\A{n}\times\A{n}\op\times\A{n}\to M_n$ which is defined by
\[
v(k,l,m) =
\begin{cases}
(0,m) & \textrm{if } l \ge k \\
(m,0) & \textrm{if } l, m < k. \\
(k-1,m+1-k) & \textrm{if } l < k \textrm{ and } m \ge k.
\end{cases}
\]
We leave it to the reader to check that this functor is well-defined and that it extends~$u$, i.e., that we have $u=v\circ (\id\times j)\colon \A{n}\times C\to M_n$. Moreover, the exactness properties of $FX$ (see \autoref{thm:AR}) imply that $Y=v^\ast FX$ lies in the essential image of $(\id\times j)_!$ and the counit $(\id\times j)_!(\id\times j)^\ast v^\ast FX\to v^\ast FX$ is hence an isomorphism. Combining this with the previous steps, we obtain natural isomorphisms
\begin{align}
D_n\otimes_{[\A{n}]} X&\cong (\id\times(1,n))^\ast(\id\times j)_!u^\ast FX\\
&= (\id\times(1,n))^\ast(\id\times j)_!(\id\times j)^\ast v^\ast FX\\
&\cong (\id\times(1,n))^\ast v^\ast FX\\
&= \big(v\circ (\id\times(1,n))\big)^\ast FX.
\end{align}

Finally, it suffices to observe that the functor $(v\circ (\id\times(1,n))\colon \A{n}\to M_n$ sends $k$ to $(k-1,n+1-k)$ and hence agrees with $s_{\A{n}}\colon \A{n}\to M_n$ defined in \eqref{eq:sAn}. Thus, by \autoref{lem:Serre-shift} we obtain natural isomorphisms
\[
D_n\otimes_{[\A{n}]} X \cong s_{\A{n}}^\ast FX\cong S X
\]
which concludes the proof.
\end{proof}

\begin{rmk}\label{rmk:Nakayama}
The proof of \autoref{thm:Serre-Nakayama} actually shows that the result holds more generally. Let \V be a closed monoidal, stable derivator and let \D be a stable derivator which is also a $\V$-module. Then the Serre functor $S\colon\D^{\A{n}}\to\D^{\A{n}}$ is naturally isomorphic to the Nakayama functor associated to the given $\V$-action. This remark applies, in particular, to stable derivators which are closed modules over the derivator of a field or a commutative ring.
\end{rmk}

\section{Admissible morphisms of stable derivators}
\label{sec:admissible}

In the previous section we saw that Serre functors for linearly oriented $A_n$-quivers are canonical Nakayama functors. Generalizing this, we now define three large classes of morphisms of stable derivators, namely left admissible, right admissible, and admissible morphisms, and show that they are also induced by explicitly constructed spectral bimodules (\autoref{thm:kernel}). The representing spectral bimodules give rise to interesting constructions in arbitrary stable homotopy theories as will be illustrated by specific examples in \S\S\ref{sec:universal-tilting}-\ref{sec:higher}.

\subsection{Spectral bimodules induced by admissible morphisms}

Let us recall that every stable derivator \D is canonically a closed $\cSp$-module. In particular, there is an action $\otimes\colon\cSp\times\D\to\D$ which is a left adjoint of two variables. As such, the action commutes with restrictions and left Kan extensions, but also, since we are in the stable context, with homotopy finite limits (\autoref{thm:exact}). The idea behind the definition of a \emph{left admissible morphism} is to encode such compatibilities (but see also \autoref{rmk:kernel-admissible}(iv)). This notion and related ones encompass essentially all previously constructed morphisms of derivators in this paper as well as those in \cite{gst:basic,gst:tree}.

\begin{defn}\label{defn:admissible}
Let \D be a stable derivator and let $A,B\in\cCat$. A morphism $\D^A\to\D^B$ is \textbf{left admissible} if it can be written as a composition of 
\begin{itemize}[leftmargin=4em]
\item[(LA1)] restriction morphisms $u^\ast\colon\D^{B'}\to\D^{A'}$,
\item[(LA2)] left Kan extensions $u_!\colon\D^{A'}\to\D^{B'}$, 
\item[(LA3)] right Kan extensions $u_\ast\colon\D^{A'}\to\D^{B'}$ along fully faithful functors which amount precisely to adding a cartesian square or right Kan extensions along countable compositions of such functors (see below for details), and
\item[(LA4)] right extensions by zero $u_\ast\colon\D^{A'}\to\D^{B'}$ for sieves $u\colon A'\to B'$.
\end{itemize}
\end{defn}

Let us be more precise about the class (LA3) which makes perfectly well sense for arbitrary derivators~\D. By definition the basic building blocks are fully faithful functors $u\colon A\to B$ between small categories such that $B-u(A)$ consists of precisely one object $b_0$ and such that there is a homotopy exact square
\begin{equation}\label{eq:add-bicartesian}
\vcenter{
\xymatrix{
\lrcorner\ar[d]_-{i_\lrcorner}\ar[r]^-j&A\ar[d]^-u\\
\square\ar[r]_-k&B
}
}
\end{equation} 
with $k$ satisfying $k(0,0)=b_0$. Using the compatibility of canonical mates with respect to pasting (\autoref{egs:htpy}), the following two factorizations
\[
\xymatrix{
\lrcorner\ar[r]^-j\ar[d]_-{i_\lrcorner}&A\ar[r]^-u\ar[d]_-u&B\ar[d]^-= \ar@{}[rrd]|{=}&& 
\lrcorner\ar[r]^-{i_\lrcorner}\ar[d]_-{i_\lrcorner}&\square\ar[r]^-k\ar[d]_-=&B\ar[d]^-=\\
\square\ar[r]_-k&B\ar[r]_-=&B && 
\square\ar[r]_-=&\square\ar[r]_-k&B
}
\]
together with \cite[Lemma~1.21]{groth:ptstab} imply that $X\in\D^B$ lies in the essential image of $u_\ast$ if and only if $k^\ast(X)\in\D^\square$ is cartesian. Thus, $u_\ast\colon\D^A\to\D^B$ induces an equivalence onto the full subderivator of $\D^B$ spanned by all $X$ such that $k^\ast(X)$ is cartesian, justifying the above terminology. In (LA3) we consider more generally countable compositions of such functors. Relevant examples of such morphisms are right Kan extensions which are taken care of by the `detection lemma' (\autoref{lem:detection}); see \autoref{lem:LA3-detect} for a precise statement.

Obviously, there are dual versions (RA1),(RA2),(RA3), and (RA4) of the defining classes in \autoref{defn:admissible} leading to \textbf{right admissible} morphisms $\D^A\to\D^B$ for stable derivators~\D. Finally, assuming still that \D is stable, a morphism $\D^A\to\D^B$ is \textbf{admissible} if it is left admissible and right admissible.

The goal of this section is to show that left admissible morphisms are canceling tensor products with spectral bimodules and similarly in the other two cases (\autoref{thm:kernel}). This result builds on the following rather obvious lemmas. Since careful proofs of these lemmas are a bit lengthy and since the details of the proofs are not relevant to the understanding of this paper, they are given in the separate subsection~\S\ref{subsec:lemmas}.

\begin{lem}\label{lem:LA3}
A morphism of derivators which preserves pullbacks also preserves right Kan extensions along functors of type (LA3).
\end{lem}

\begin{lem}\label{lem:sieves}
A morphism of derivators which preserves terminal objects also preserves right Kan extensions along sieves.
\end{lem}

We note that these two lemmas apply to exact morphisms of stable derivators.

\begin{thm}\label{thm:kernel}
Let \D be a stable derivator and let $F\colon\D^A\to\D^B$ be a morphism.
\begin{enumerate}
\item If $F$ is left admissible then there is a bimodule $M\in\cSp(B\times A\op)$ and a natural isomorphism $F\cong M\otimes_{[A]}-\colon\D^A\to\D^B$.
\item If $F$ is right admissible then there is a bimodule $N\in\cSp(A\times B\op)$ and a natural isomorphism $F\cong -\lhd_{[A]}N\colon\D^A\to\D^B$.
\end{enumerate}
\end{thm}
\begin{proof}
As we shall discuss in \autoref{lem:la-dual} right after the proof, it suffices to take care of the first statement only. Passing to canceling tensor products of the respective spectral bimodules, we can assume that~$F$ belongs to one of the classes (LA1),(LA2),(LA3), or (LA4) in \autoref{defn:admissible}. We begin with the case (LA3) and consider the diagram 
\[
 \xymatrix{
\cSp^{A\times A\op}\times\D^A\ar[d]_-{(u\times\id)_\ast\times\id}\ar[r]^-{\otimes_{[A]}}&\D^A\ar[d]^-{u_\ast}\\
\cSp^{B\times A\op}\times\D^A\ar[r]_-{\otimes_{[A]}}&\D^B,
}
\]
which is induced by the canonical action of $\cSp$ on \D. For $K\in\cSp^{A\times A\op}$ and $X\in \D^A$, forming a composition analogous to \eqref{eq:first}, there is a canonical mate transformation $\beta\colon((u\times\id)_\ast K)\otimes_{[A]}X\to u_\ast(K\otimes_{[A]}X)$. In fact, the partial canceling tensor product $-\otimes_{[A]}X$ defines a morphism of derivators and $\beta$ is the canonical mate describing the compatibility with right Kan extensions. Since $-\otimes_{[A]}X$ is a left adjoint morphism it preserves pushouts \cite[Proposition~2.9]{groth:ptstab}. In the stable context, this implies that $-\otimes_{[A]}X$ also preserves pullbacks and \autoref{lem:LA3} then guarantees that $\beta$ is an isomorphism. Specializing to $K=\lI_A\in\cSp(A\times A\op)$, we obtain natural isomorphisms
\[
((u\times\id)_\ast\lI_A)\otimes_{[A]} X\stackrel{\beta}{\to} u_\ast(\lI_A\otimes_{[A]} X)\cong u_\ast(X),
\]
which is to say that $u_\ast\colon\D^A\to\D^B$ is naturally isomorphic to a canceling tensor product. 

The class (LA4) works similarly using \autoref{lem:sieves} this time, and the class (LA1) follows immediately from the pseudo-functoriality of the action since there are isomorphisms
$((u\times\id)^\ast\lI_A)\otimes_{[A]} X\cong u^\ast(\lI_A\otimes_{[A]} X)\cong u^\ast(X).$ Finally, since the morphism~$-\otimes_{[A]}X$ is a left adjoint it commutes with left Kan extensions \cite[Proposition~2.9]{groth:ptstab} and we obtain natural isomorphisms
\[
u_!(X)\cong u_!(\lI_A\otimes_{[A]} X)\cong ((u\times\id)_!\lI_A)\otimes_{[A]} X.
\]
Thus, also morphisms of type (LA2) are canceling tensor products, concluding the proof. 
\end{proof}

Let us now discuss two technical aspects regarding the theorem. First of all, we have proved only part (i) since part (ii) follows by duality. In fact, if $F\colon \D \to \E$ is a morphism of derivators, we obtain the \emph{opposite morphism} $F\op\colon \D\op \to \E\op$, where $F\op_A\colon \D\op(A) \to \E\op(A)$ is defined as $(F_{A\op})\op$, the opposite of the $A\op$-component $F_{A\op}\colon \D(A\op) \to \E(A\op)$ of $F$. If $\D$ is a stable derivator, then so is $\D\op$ and as such it carries a canonical closed $\cSp$-module structure (see \autoref{prop:stable-frames-opp}). More generally, if $\D$ is a closed $\V$-module, so is canonically $\D\op$; details will appear in~\cite{gs:enriched}.
Suppose now that $F\colon \D^A \to \D^B$ is a left admissible morphism. Upon identifying $(\D^A)\op \cong (\D\op)^{A\op}$ and $(\D^B)\op \cong (\D\op)^{B\op}$, it follows directly from the definitions that $F\op\colon (\D\op)^{A\op} \to (\D\op)^{B\op}$ is right admissible, and dually.

\begin{lem} \label{lem:la-dual}
Let $\D$ be a stable derivator. If $F \cong M \otimes_{[A]} -\colon \D^A \to \D^B$ for $M \in \cSp(B \times A\op)$, then there is a natural isomorphism $F\op \cong - \lhd_{[A\op]} M\colon (\D\op)^{A\op} \to (\D\op)^{B\op}$, considering $M$ as an object of $\cSp(A\op \times (B\op)\op)$. Dually, if $F \cong - \lhd_{[A]} N$ for $N \in \cSp(A \times B\op)$, then $F\op \cong N \otimes_{[A\op]} -$.
\end{lem}

\begin{proof}
Suppose that $F \cong M \otimes_{[A]} -$. We claim that there is a natural isomorphism
\[ M \otimes - \cong (- \lhd M)\op \colon \D \to \D^{B \times A\op}. \]
In order to unravel this formula, note that the external version of the canonical action of $\cSp$ on $\D$ yields a morphism $M \otimes -\colon \D \to \D^{B \times A\op}$. On the other hand, the cotensor by $M$ on $\D\op$ provides us with a morphism $- \lhd M\colon \D\op \to (\D\op)^{(B\times A\op)\op}$. Passing to the opposite morphism and identifying $(\D\op)^{(B\times A\op)\op} \cong (\D^{B \times A\op})\op$ gives the morphism $(- \lhd M)\op \colon \D \to \D^{B \times A\op}$. Having said that, the claim is an immediate consequence of \autoref{prop:stable-frames-opp}.

The morphism $F \cong M\otimes_{[A]}-\colon \D^A \to \D^B$ is by definition obtained by composing $M \otimes -\colon \D^A \to \D^{B \times A\op \times A}$ with the parametrized coend. Similarly, $- \lhd_{[A\op]} M\colon (\D\op)^{A\op} \to (\D\op)^{B\op}$ is defined as a composition of $-\lhd M$ with an end. Passing to the opposite functor, $(-\lhd_{[A\op]} M)\op\colon \D^A \to \D^B$ is isomorphic to the composition of $(- \lhd M)\op$ with a parametrized coend again. Combining this observation with the above isomorphism $M \otimes - \cong (- \lhd M)\op$, it follows that
\[ F \cong M \otimes_{[A]} - \cong (-\lhd_{[A\op]} M)\op. \]
Thus, $F\op \cong -\lhd_{[A\op]} M$, as stated.
\end{proof}

As for the second aspect, the following seemingly technical lemma guarantees that \autoref{thm:kernel} applies to our earlier constructions; see \S\S\ref{sec:universal-tilting}-\ref{sec:higher}. The lemma roughly says that those  right Kan extensions which amount to adding countably many cartesian squares all of which are detected by the `detection lemma' (\autoref{lem:detection}) belong to the class (LA3).

\begin{lem}\label{lem:LA3-detect}
Let $u\colon A=(B-\{b_0\})\to B$ be the fully faithful inclusion of the complement of $b_0\in B$ and let us consider a diagram~\eqref{eq:add-bicartesian}. If the induced functor $\tilde{j}\colon\lrcorner\to(b_0/A)$ is a left adjoint, then the square \eqref{eq:add-bicartesian} is homotopy exact, i.e., $u_\ast$ belongs to (LA3).
\end{lem}

Also the proof of this lemma will be given in \S\ref{subsec:lemmas}.

\begin{rmk}\label{rmk:kernel-admissible}
\begin{enumerate}
\item The proof of \autoref{thm:kernel} gives explicit constructions of the resulting bimodules, which will be useful later. Imposing the following additional assumption it is immediate that these bimodules are unique up to isomorphism. Note that given a left admissible or right admissible morphism $F_\D\colon\D^A\to\D^B$, then there is also such a morphism $F_\E\colon\E^A\to\E^B$ for arbitrary stable derivators~$\E$. \emph{If we speak of such a morphism, then we usually have this family in mind.} The uniqueness of the bimodules follows if we insist that they realize all these $F_\E$. In fact, it suffices to assume this for $F_{\cSp}$ for left admissible morphisms as follows immediately from $F_{\cSp}(\lI_A)\cong M\otimes_{[A]}\lI_A\cong M$. Dually, it suffices to consider $F_{\cSp\op}$ for right admissible morphisms.
\item By the same proof there is a variant of \autoref{thm:kernel} for stable derivators which are closed modules over monoidal, stable derivators~\V. Obviously, left admissible, right admissible, and admissible morphisms make sense for such derivators and a variant of \autoref{thm:kernel} then gives rise to bimodules in $\V$, which are unique up to isomorphism if we impose a similar condition as above. In particular, if we consider such closed modules over the derivator $\D_k$, $k$ a commutative ring, then \autoref{thm:kernel} yields uniquely determined objects in
\[
\D_k(B\times A\op)\simeq D(kB\otimes_k kA\op),
\]
i.e., chain complexes of bimodules over category algebras.

\item Let us also mention the relation to the model situation from representation theory. One classically considers the situation where $\V = \D = \D_k$ is the derivator of a field $k$ and $T \in \D_k(A \times B\op) \simeq D(kA \otimes_k kB\op)$ is a complex of $kA$-$kB$-bimodules, which then corepresents a functor between derived categories
\[ \Rhom_{kA}(T,-)\colon D(kA) \longrightarrow D(kB). \]
But $\Rhom_{kA}(T,-)$ is precisely $(- \lhd_{[A]} T)_\bbone\colon \D_k^A(\bbone) \to \D_k^B(\bbone)$, i.e., the underlying functor of the morphism
\[
 - \lhd_{[A]} T\colon \D_k^A \longrightarrow \D_k^B. 
\]
Classically, the name \emph{tilting complex} or, in our terminology, \emph{tilting bimodule} is reserved for bimodules~$T$ such that $\Rhom_{kA}(T,-)$ is an equivalence; see~\cite{rickard:derived-fun,keller:deriving-dg}. We will encounter such bimodules in our general setting in \S\S\ref{sec:universal-tilting}-\ref{sec:higher}.

\item \autoref{thm:kernel} suffices for our applications in this paper but the following more general statement should be true. Let us say that a right Kan extension functor $u_\ast\colon\D^{A'}\to\D^{B'}$ is \textbf{cofinally homotopy finite} if for every $b'\in B'$ there is a cofinal morphism $C\to(b'/A')$ with a homotopy finite domain. This is for example the case if we can find homotopy finite subcategories in $(b'/A')$ such that the inclusion is a left adjoint. \autoref{thm:kernel} should still be true if in \autoref{defn:admissible} one uses cofinally homotopy finite right Kan extensions instead of the class (LA3). In fact, every left exact morphism should preserve cofinally homotopy finite right Kan extensions.
\end{enumerate}
\end{rmk}

\subsection{Some proofs concerning admissible morphisms}
\label{subsec:lemmas}

In this subsection we provide the missing proofs of \autoref{lem:LA3}, \autoref{lem:sieves}, and \autoref{lem:LA3-detect}. The reader less inclined into the formalism of homotopy exact squares is suggested to continue with \S\S\ref{sec:kernel}-\ref{sec:cone} and to only then get back to these proofs. 

To begin with we collect the following variant of \cite[Lemma~1.21]{groth:ptstab} which is also of independent interest.

\begin{lem}\label{lem:ff-kan}
Let $F\colon \D\to\E$ be a morphism of derivators, let $u\colon A\to B$ be fully faithful, and let $X\in\D(A)$. The canonical mate $\beta\colon (F\circ u_\ast)(X)\to (u_\ast\circ F)(X)$ is an isomorphism in $\E(B)$ if and only if $\beta_b\colon (F\circ u_\ast)(X)_b\to (u_\ast\circ F)(X)_b$ is an isomorphism for all $b\in B-u(A)$.
\end{lem}
\begin{proof}
By (Der2) the canonical mate $\beta$ is an isomorphism if and only if $\beta_b$ is an isomorphism for all $b\in B$. Thus, it is enough to show that $u^\ast(\beta)$ is always an isomorphism. Since $u$ is fully faithful, the adjunction counit $\epsilon\colon u^\ast u_\ast\to\id$ is an isomorphism (\autoref{egs:htpy}). Recall that a morphism of derivators comes with coherent structure isomorphisms $\gamma=\gamma_u\colon u^\ast\circ F_B\cong F_A\circ u^\ast$ and that $\beta$ is the canonical mate of $\gamma$. One easily checks that
\begin{equation}\label{eq:ff-kan}
\vcenter{
\xymatrix{
u^\ast F_B u_\ast\ar[d]_-{\gamma u_\ast}\ar[r]^-{u^\ast\beta}&u^\ast u_\ast F_A\ar[d]^-{\epsilon {F_A}}\\
F_Au^\ast u_\ast\ar[r]_-{F_A\epsilon}&F_A
}
}
\end{equation}
commutes. Since $\epsilon$ and $\gamma$ are isomorphisms, the same is true for $u^\ast\beta.$
\end{proof}

Now we can take care of \autoref{lem:LA3}.

\begin{proof}[Proof of \autoref{lem:LA3}]
Let $F\colon\D\to\E$ be a morphism of derivators and let $u$ be a functor as in (LA3), adding one cartesian square. By assumption on $u$, there is a homotopy exact square \eqref{eq:add-bicartesian}. Moreover, the unique object in $B-u(A)$ lies in the image of $k$, so that \autoref{lem:ff-kan} implies that the canonical mate $\beta\colon F_B u_\ast\to u_\ast F_A$ is an isomorphism if and only if $k^\ast\beta$ is an isomorphism. The morphism $k^\ast\beta$ sits in the diagram
\begin{equation}\label{eq:exact}
\vcenter{
\xymatrix{
k^\ast F_B u_\ast\ar[r]^-{k^\ast\beta}\ar[d]_-{\gamma u_\ast}^-\cong&k^\ast u_\ast F_A\ar[r]^-\cong&(i_\lrcorner)_\ast j^\ast F_A\ar[d]^-{(i_\lrcorner)_\ast \gamma}_-\cong\\
F_\square k^\ast u_\ast\ar[r]_-\cong&F_\square(i_\lrcorner)_\ast j^\ast\ar[r]&(i_\lrcorner)_\ast F_\lrcorner j^\ast.
}
}
\end{equation}
In this diagram the vertical morphisms are the structure isomorphisms of $F$ and the horizontal morphisms labeled by the isomorphism symbol only are the canonical isomorphisms associated to the homotopy exact square \eqref{eq:add-bicartesian}. The remaining morphism in the bottom row is the canonical mate $F_\square(i_\lrcorner)_\ast\to(i_\lrcorner)_\ast F_\lrcorner$ and is by assumption on $F$ an isomorphism. Using the compatibility of mates with pasting one checks that this square commutes, and we conclude that $k^\ast\beta$ and hence $\beta$ is an isomorphism, which is to say that $F$ preserves right Kan extensions along~$u$.

The case of a countable composition~$v$ of such functors is taken care of by induction. In fact, the canonical morphism $\beta'\colon F v_\ast\to v_\ast F$ is an isomorphism if and only if this is true on the countably many objects not lying in the image of~$v$. But each such object is taken care of by the argument above.
\end{proof}

We now turn to \autoref{lem:sieves}.

\begin{proof} [Proof of \autoref{lem:sieves}]
Let $F\colon\D\to\E$ be a morphism of derivators which preserves terminal objects and let $u\colon A\to B$ be the inclusion of a sieve. We have to show that the canonical mate $\beta\colon Fu_\ast\to u_\ast F$ is an isomorphism. Let $v\colon A'\to B$ be the inclusion of the complement of $A$. Since $u$ is fully faithful, by \autoref{lem:ff-kan} it suffices to show that $v^\ast\beta\colon v^\ast F u_\ast\to v^\ast u_\ast F$ is an isomorphism. But by the unpointed version of \autoref{lem:extbyzero} (see \cite[Prop.~1.23]{groth:ptstab}), $u_\ast$ is right extension by terminal objects and, since $F$ preserves terminal objects, all components of $v^\ast F u_\ast$ and $v^\ast u_\ast F$ are terminal objects. Hence all components of $v^\ast\beta$ are necessarily isomorphisms and we can conclude by (Der2).
\end{proof}

It remains to take care of \autoref{lem:LA3-detect} providing many morphisms of type (LA3). The proof of this lemma relies on the following one which is also of independent interest and in which we consider a commutative square of small categories
\begin{equation}\label{eq:base}
\vcenter{
\xymatrix{
A'\ar[r]^-j\ar[d]_-{u'}&A\ar[d]^-u\\
B'\ar[r]_-k&B.
}
}
\end{equation}

\begin{lem}\label{lem:base}
Let \eqref{eq:base} be a commutative square in \cCat such that $u$ and $u'$ are fully faithful. The square is homotopy exact if and only if for all derivators \D the canonical mate $\beta_{b'}\colon (k^\ast u_\ast)_{b'}\to (u'_\ast j^\ast)_{b'}$ is an isomorphism for all $b'\in B'-u'(A')$.
\end{lem}
\begin{proof}
By (Der2) the square \eqref{eq:base} is homotopy exact if and only if the canonical mate $\beta\colon k^\ast u_\ast \to u'_\ast j^\ast$ is an isomorphism at all components, and it is hence enough to show that under our assumptions this is always true for objects of the form $u'(a')$. For this purpose, let us consider the following pasting on the left
\[
\xymatrix{
\bbone\ar[r]\ar[d]_-=&(u'(a')/A')\ar[r]^-p\ar[d]_-\pi&A'\ar[r]^-j\ar[d]^-{u'}&A\ar[d]^-u&&
\bbone\ar[r]\ar[d]_-=&(uj(a')/A)\ar[r]^-q\ar[d]_-\pi&A\ar[d]^-u\\
\bbone\ar[r]_-=&\bbone\ar[r]_-{u'(a')}&B'\ar[r]_-k\ultwocell\omit{}&B&&
\bbone\ar[r]_-=&\bbone\ar[r]_-{uj(a')}&B\ultwocell\omit{}
}
\]
in which the square in the middle is a comma square. The unlabeled morphism $\bbone\to(u'(a')/A')$ classifies the object $(a',\id\colon u'(a')\to u'(a'))$ which is an initial object since $u'$ is fully faithful. By (Der4) the comma square is homotopy exact as is the square on the left by \autoref{egs:htpy}. The compatibility of mates with respect to pasting implies that $\beta_{u'(a')}$ is an isomorphism if and only if the canonical mate associated to the pasting on the left is an isomorphism. Note that this pasting can also be rewritten as the pasting on the right. In that diagram the square on the right is again a comma square while the morphism $\bbone\to (uj(a')/A)$ classifies the object $(j(a'),\id\colon uj(a')\to uj(a'))$. The fully faithfulness of $u$ implies that this object is initial in $(uj(a')/A)$. Thus, using again that comma squares are homotopy exact together with \autoref{egs:htpy} it follows that the pasting on the right is homotopy exact, concluding the proof.
\end{proof}

Finally, we can take care of \autoref{lem:LA3-detect}.

\begin{proof}[Proof of \autoref{lem:LA3-detect}]
By \autoref{defn:admissible} it suffices to prove that the square \eqref{eq:add-bicartesian} is homotopy exact. Thus, we have to show that in any derivator the canonical mate $\beta\colon k^\ast u_\ast\to (i_\lrcorner)_\ast j^\ast$ is an isomorphism. By \autoref{lem:base} it is enough to check that this is the case for the component $\beta_{(0,0)}$. Let us consider the 
following pasting
\begin{equation}\label{eq:IIIa}
\vcenter{
\xymatrix{
\lrcorner\ar[r]^-f_-\cong\ar[d]_-\pi&((0,0)/\lrcorner)\ar[r]^-p\ar[d]_-\pi&\lrcorner\ar[r]^-j\ar[d]^-{i_\lrcorner}&A\ar[d]^-u\\
\bbone\ar[r]_-=&\bbone\ar[r]_-{(0,0)}&\square\ar[r]_-k\ultwocell\omit{}&B
}
}
\end{equation}
in which the square in the middle is a comma square. Using the compatibility of mates with respect to pasting, the canonical mate of this diagram factors as
\[
(k^\ast u_\ast)_{(0,0)}\stackrel{\beta_{(0,0)}}{\to} ((i_\lrcorner)_\ast j^\ast)_{(0,0)}\stackrel{\cong}{\to}
\pi_\ast p^\ast j^\ast\stackrel{\cong}{\to}\pi_\ast f^\ast p^\ast j^\ast.
\]
Thus, $\beta\colon k^\ast u_\ast\to (i_\lrcorner)_\ast j^\ast$ is an isomorphism if and only if \eqref{eq:IIIa} is homotopy exact. But the reader easily checks that \eqref{eq:IIIa} can also be obtained as the pasting
\begin{equation}\label{eq:IIIb}
\vcenter{
\xymatrix{
\lrcorner\ar[r]^-{\tilde{j}}\ar[d]_-\pi&(b_0/A)\ar[r]^-q\ar[d]_-\pi&A\ar[d]^-u\\
\bbone\ar[r]_-=&\bbone\ar[r]_-{b_0}&B\ultwocell\omit{}
}
}
\end{equation}
in which the square on the right is a comma square. By (Der4) the comma square is homotopy exact and, since $\tilde{j}$ is a left adjoint, also the square on the left is homotopy exact (\autoref{egs:htpy}). Thus, the pasting \eqref{eq:IIIb}=\eqref{eq:IIIa} is homotopy exact, showing that $\beta$ is an isomorphism.
\end{proof}

\section{Tensor and hom functors in stable derivators}
\label{sec:kernel}

In \S\ref{sec:admissible} we observed that left admissible, right admissible, and admissible morphisms are induced by spectral bimodules. In this section we collect a few properties of tensor and hom functors associated to bimodules in monoidal derivators which will be of use in later sections. These results are parallel to standard algebraic facts using derived tensor and derived Hom functors and yield a useful calculus of such bimodules. 

To start with, let $\V$ be a closed monoidal derivator and let \D be a closed module over \V. Using canceling-external versions of $\otimes$ and $\lhd$, each $M\in\V(B\times A\op)$ induces morphisms
\begin{equation}\label{eq:weighted-limits}
M\otimes_{[A]}-\colon\D^A\to\D^B\qquad\text{and}\qquad -\lhd_{[B]}M\colon\D^B\to\D^A.
\end{equation}
These morphisms yield an adjunction $(M\otimes_{[A]}-,-\lhd_{[B]}M)\colon\D^A\rightleftarrows\D^B$ (\cite[Examples~8.15-8.16]{gps:additivity}).
 
\begin{rmk}
In the language of \emph{\V-enriched derivators} \cite{gs:enriched}, a bimodule $M\in\V(B\times A\op)$ is a \emph{weight} and the morphisms \eqref{eq:weighted-limits} are the associated \emph{weighted (co)limit constructions}. As we saw in \S\ref{sec:admissible} and as will be studied more systematically in \cite{gs:enriched}, in the stable context some seemingly \emph{limit} constructions can be obtained as \emph{weighted colimit} constructions, and dually.
\end{rmk}

The main case of interest to us is given by stable derivators endowed with their canonical closed module structure over the derivator \cSp of spectra. However, among the stable derivators some are of particular interest in algebraic settings, where one typically studies homotopy categories of certain \emph{additive} stable model categories. To be more precise, given a complete and cocomplete abelian category, one can consider model structures compatible with the abelian structure in the sense of Hovey~\cite{hovey:cotorsion,hovey:cotorsion2} (see also \cite{stovicek:exact-model} and references there for more details). Under a mild additional condition ensuring stability, the category of cofibrant-fibrant objects is a Frobenius exact category by~\cite{gillespie:exact} or \cite[Proposition 1.1.14]{becker:models-singularity}. As stable categories of Frobenius exact categories are the classical source of triangulated categories in algebra (see~\cite[Chapter I.]{happel:triangulated}), they are nowadays called \emph{algebraic triangulated categories} (see~\cite[\S7]{krause:chicago} and \cite{schwede:alg-versus-top}). Derivators of Hovey's abelian model structures often carry a canonical closed $\D_\lZ$-module structure. Let us also recall that $\D_\lZ$ and, more generally, $\D_k$, $k$ a commutative ring, come with monoidal, colimit preserving morphisms from $\cSp$.

\begin{eg}\label{eg:EM-spectra}
Let $k$ be a commutative ring. Extension of scalars along the characteristic $\lZ\to k$ induces a colimit preserving, monoidal morphism of derivators $k\otimes-\colon\D_\lZ\to\D_k$. We denote by $H\lZ$ the symmetric integral Eilenberg--Maclane ring spectrum and by $\D_{H\lZ}$ the stable derivator of $H\lZ$-module spectra. Associated to the map of symmetric ring spectra $\lS\to H\lZ$ there is the colimit preserving, monoidal morphism of derivators $H\lZ\wedge-\colon\cSp\to\D_{H\lZ}$. By \cite{shipley:spectra-dga} the model categories $\Ch(\lZ)$ and $\Mod(H\lZ)$ are related by a zigzag of weakly monoidal Quillen equivalences passing through combinatorial model categories only. But such weakly monoidal Quillen equivalences induce monoidal equivalences of homotopy derivators, so that we end up with a monoidal equivalence of derivators $\D_{H\lZ}\simeq\D_\lZ$. As an upshot, there are colimit preserving, monoidal morphisms of derivators
\[
k\otimes-\colon\cSp\stackrel{H\lZ\wedge-}{\to}\D_{H\lZ}\simeq\D_\lZ\stackrel{k\otimes-}{\to}\D_k.
\]
\end{eg}

This motivates the following definitions.

\begin{defn} \label{defn:alg-der}
An \textbf{algebraic stable derivator} is a stable derivator which has the property that the canonical action of $\cSp$ is the restriction of a closed $\D_\lZ$-module structure along $\lZ\otimes-\colon\cSp\to\D_\lZ$.  More generally, if $k$ is a commutative ring, then a \textbf{$k$-linear, stable derivator} is a stable derivator such that the canonical $\cSp$-module structure is the restriction of a closed $\D_k$-module structure along $k\otimes-\colon\cSp\to\D_k$.
\end{defn} 

\emph{In the remainder of this section the derivators \V and \D are assumed to be stable}. The following lemma applies to admissible morphisms.

\begin{lem} \label{lem:adm-adj-triple}
Given bimodules $M\in\V(B\times A\op), N\in\V(A\times B\op)$ together with natural isomorphisms $F\cong M\otimes_{[A]}-\cong -\lhd_{[A]}N$ there is an adjoint triple of morphisms
\[
(N\otimes_{[B]}-,F, -\lhd_{[B]} M).
\]
\end{lem}
\begin{proof}
This is immediate since the morphisms in \eqref{eq:weighted-limits} are part of an adjunction.
\end{proof}

\begin{cor}
Admissible morphisms are continuous and cocontinuous.
\end{cor}

In the following lemma we consider \textbf{admissible adjunctions}, i.e., adjunctions $(F,G)\colon\D^A\rightleftarrows\D^B$ such that $F,G$ both are admissible for every stable derivator \D. In such a situation there are isomorphisms
\begin{equation}\label{eq:admissible-adj}
F\cong M_F\otimes_{[A]}-\cong -\lhd_{[A]}N_F\qquad\text{and}\qquad G\cong M_G\otimes_{[B]}-\cong -\lhd_{[B]}N_G
\end{equation}
together with resulting adjoint triples
\[
(N_F\otimes_{[B]}-,F,-\lhd_{[B]}M_F)\qquad\text{and}\qquad (N_G\otimes_{[A]}-,G,-\lhd_{[A]}M_G).
\]

\begin{lem}
Let $(F,G)\colon\D^A\rightleftarrows\D^B$ be an admissible adjunction as in \eqref{eq:admissible-adj}. There is an isomorphism $M_F\cong N_G$ in $\V(B\times A\op)$ and an adjoint quadruple of morphisms
\[
(N_F\otimes_{[B]}-,F,G,-\lhd_{[A]} M_G).
\]
\end{lem}
\begin{proof}
This follows immediately from $M_F\otimes_{[A]}-\cong N_G\otimes_{[A]}-$ and the uniqueness of the representing bimodules.
\end{proof}

Thus, up to isomorphism there are only three bimodules involved in such an adjunction. Finally, we consider \textbf{admissible equivalences} $(F,G)\colon\D^A\rightleftarrows\D^B$, i.e., equivalences such that $F,G$ both are admissible for every stable derivator $\D$. Still using the notation as in \eqref{eq:admissible-adj}, the previous lemma implies that there are isomorphisms $M_F\cong N_G$ and $M_G\cong N_F$. 

\begin{cor}\label{cor:inverse-modules}
Let $(F,G)\colon\D^A\rightleftarrows\D^B$ be an admissible equivalence. Then there are bimodules $M\in\V(B\times A\op)$ and $N\in\V(A\times B\op)$ together with natural isomorphisms
$F\cong M\otimes_{[A]}-\cong -\lhd_{[A]}N$ and $G\cong N\otimes_{[B]}-\cong -\lhd_{[B]} M$. Moreover, $M$ and $N$ are respectively inverses of each other, i.e., there are isomorphisms
\[
N\otimes_{[B]} M\cong \lI_A\qquad\text{and}\qquad M\otimes_{[A]} N\cong \lI_B.
\]
Using the notation $N=M^{-1}\in\V(A\times B\op)$ the admissible equivalence $(F,G)$ is naturally isomorphic to
\begin{enumerate}
\item the tensor-hom equivalence $(M\otimes_{[A]}-,-\lhd_{[B]}M)\colon\D^A\rightleftarrows\D^B$,
\item the tensor-tensor equivalence $(M\otimes_{[A]}-,M^{-1}\otimes_{[B]}-)\colon\D^A\rightleftarrows\D^B$,
\item the hom-hom equivalence $(-\lhd_{[A]}M^{-1},-\lhd_{[B]}M)\colon\D^A\rightleftarrows\D^B$,
\item the hom-tensor equivalence $(-\lhd_{[A]}M^{-1},M^{-1}\otimes_{[B]}-)\colon\D^A\rightleftarrows\D^B$.
\end{enumerate}
\end{cor}
\begin{proof}
It only remains to show that $M$ and $N$ are inverses of each other which is immediate from the uniqueness of the representing bimodules.
\end{proof}

In the special case of $A=B$ we obtain elements in the \emph{Picard group of $A$ relative to~\V} which is defined as follows.
Recall that associated to a monoidal category $(\bV,\otimes,\lS)$ there is the \emph{Picard group} $\Pic(\bV)$ of $\bV$. Its elements are isomorphism classes $[v]$ of $\otimes$-\emph{invertible} objects, i.e., of objects $v$ such that there is a $w$ together with isomorphisms $v\otimes w\cong \lS\cong w\otimes v$. The group structure is induced by the monoidal structure and the neutral element is $[\lS]$. Note that in a bicategory every object has a monoidal category of endomorphisms. Applied to the bicategories $\cProf(\V)$ of \autoref{thm:bicategory} this yields the following definition.

\begin{defn}\label{defn:picard}
Let  \V be a monoidal derivator and let $A\in\cCat$. The \textbf{Picard group $\Pic_\V(A)$ of $A$ relative to \V} is the Picard group of the monoidal category $\cProf(\V)(A,A)$.
\end{defn}

Thus, elements of $\Pic_\V(A)$ are isomorphism classes of $\otimes_{[A]}$-invertible profunctors $M\in\cProf(\V)(A,A)=\V(A\times A\op)$, the group structure is induced by the canceling tensor product $\otimes_{[A]}$, and the neutral element is represented by the identity profunctor $\lI_A$. In particular, for $A\in\cCat$ we obtain the \textbf{spectral Picard group} $\Pic_{\cSp}(A)$, the \textbf{integral Picard group} $\Pic_{\D_\lZ}(A)$, and also the \textbf{$k$-linear Picard group} $\Pic_{\D_k}(A)$ for $k$ an arbitrary commutative ring. If $k$ is a field, then for particular choices of $A$ the group $\Pic_{\D_k}(A)$ coincides with the \emph{derived Picard group} of Miyachi and Yekutieli~\cite{miyachi-yekutieli}.

The Picard group construction enjoys the following functoriality property.

\begin{lem} \label{lem:Picard}
Let $F\colon\V\to\W$ be a colimit preserving, monoidal morphism of derivators and let $A\in\cCat$. The assignment $[M]\mapsto [FM]$ defines a group homomorphism $\Pic_\V(A)\to\Pic_\W(A)$, and this construction is functorial in $F$.
\end{lem}
\begin{proof}
It is a classical fact that the usual Picard group construction is functorial with respect to monoidal functors. Thus, it suffices to show that $F\colon\V\to\W$ induces a monoidal functor $\cProf(\V)(A,A)\to\cProf(\W)(A,A)$. Since $F$ is colimit preserving, $F$ also preserves coends and for $M,N\in\cProf(\V)(A,A)$ we obtain canonical isomorphisms
\[
F(M\otimes_{[A]}N)= F(\int^A M\otimes N) \cong \int^A F(M\otimes N) \cong \int^A FM\otimes FN = FM\otimes_{[A]}FN.
\]
In a similar way we obtain canonical isomorphisms 
\[
F(\lI_A)=F((t,s)_!\lS_{\tw(A)})\cong(t,s)_!F(\lS_{\tw(A)})\cong (t,s)_!\lS_{\tw(A)}=\lI_A.
\] 
One can check that these isomorphisms satisfy the desired coherence properties and hence define a monoidal functor $\cProf(\V)(A,A)\to\cProf(\W)(A,A)$. For the conclusion of the lemma however we merely need the existence of such isomorphisms. Clearly, this construction is functorial in~$F$.
\end{proof}

\begin{eg}\label{eq:picard}
Let $k$ be a commutative ring. By \autoref{eg:EM-spectra} 
we obtain for every $A\in\cCat$ group homomorphisms
\[
\Pic_{\cSp}(A)\to\Pic_{\D_\lZ}(A)\to\Pic_{\D_k}(A).
\]
\end{eg}

We next indicate that for \emph{symmetric} monoidal derivators~\V there is a natural  monoidal structure on the bicategory $\cProf(\V)$. Since monoidal bicategories are rather complicated objects (they are tricategories with one object), we content ourselves by defining the monoidal pairing and establishing the only coherence property used in our applications. The monoidal structure on $\cProf(\V)$ is essentially given by external versions of $\otimes\colon\V\times\V\to\V$, twisted by suitable canonical isomorphisms in $\cCat$, namely
\[
\V(B_1\times A_1\op)\times\V(B_2\times A_2\op)\stackrel{\otimes}{\to}\V(B_1\times A_1\op\times B_2\times A_2\op)\cong\V((B_1\times B_2)\times(A_1\times A_2)\op).
\]
We commit a minor abuse of notation and denote this composition by
\[
\otimes\colon\cProf(\V)(B_1,A_1)\times\cProf(\V)(B_2,A_2)\to\cProf(\V)(B_1\times B_2,A_1\times A_2).
\]

\begin{lem}\label{lem:mon-bicat}
Let \V be a symmetric monoidal derivator. Then for profunctors $M_i\in\V(B_i\times A_i\op), N_i\in\V(C_i\times B_i\op)$, $i=1,2$, there is a canonical isomorphism
\[
(N_1\otimes_{[B_1]} M_1)\otimes (N_2\otimes_{[B_2]} M_2) \cong (N_1\otimes N_2) \otimes_{[B_1\times B_2]} (M_1\otimes M_2)
\]
in $\V((C_1\times C_2)\times (A_1\times A_2)\op)$.
\end{lem}
\begin{proof}
Since \V is a symmetric monoidal derivator, the symmetry constraint of \V yields a canonical isomorphism
\[
N_1\otimes M_1\otimes N_2\otimes M_2 \cong N_1\otimes N_2 \otimes M_1\otimes M_2,
\]
where again the restriction along the symmetry constraint in $(\cCat,\times)$ is implicit. By Fubini's lemma \cite[Lemma 5.3]{gps:additivity} there is also a canonical isomorphism between the respective coends showing up in the statement of the lemma.
\end{proof}

If we are now given two left admissible morphisms of stable derivators together with associated spectral bimodules
\begin{equation}\label{eq:la-commute}
F_1\cong M_1\otimes_{[A_1]}-\colon\D^{A_1}\to\D^{B_1}\qquad\text{and}\qquad 
F_2\cong M_2\otimes_{[A_2]}-\colon \D^{A_2}\to\D^{B_2},
\end{equation}
then, using again that left admissible morphisms can be evaluated on arbitrary stable derivators, we obtain the diagram of morphisms
\[
\xymatrix{
\D^{A_1\times A_2} \ar[r]^-{F_1} \ar[d]_-{F_2} & \D^{B_1\times A_2} \ar[d]^-{F_2} \\
\D^{A_1\times B_2} \ar[r]_-{F_1} & \D^{B_1\times B_2}.
}
\]
At the level of the corresponding spectral bimodules associated to the composition through the upper right corner there are canonical isomorphisms
\begin{align}
(\lI_{B_1}\otimes M_2)\otimes_{[B_1\times A_2]}(M_1\otimes \lI_{A_2})&\cong (\lI_{B_1}\otimes_{[B_1]}M_1)\otimes (M_2\otimes_{[A_2]}\lI_{A_2})\\
&\cong M_1\otimes M_2
\end{align}
where we used \autoref{lem:mon-bicat} in the first step. Similarly, there are canonical isomorphisms
\begin{align}
(M_1\otimes \lI_{B_2})\otimes_{[A_1\times B_2]}(\lI_{A_1}\otimes M_2)&\cong (M_1\otimes_{[A_1]}\lI_{A_1})\otimes (\lI_{B_2}\otimes_{[B_2]}M_2)\\
&\cong M_1\otimes M_2.
\end{align}

\begin{defn}\label{defn:boxtimes}
Let \D be a stable derivator and let us consider left admissible morphisms with representing spectral bimodules as in \eqref{eq:la-commute}. Then we denote by 
$F_1\boxtimes F_2$ the morphism
\[
F_1\boxtimes F_2=(M_1\otimes M_2)\otimes_{[A_1\times A_2]}-\colon\D^{A_1\times A_2}\to\D^{B_1\times B_2}.
\]
\end{defn}

\begin{lem}\label{lem:boxtimes}
Let \D be a stable derivator and let us consider left admissible morphisms as in \eqref{eq:la-commute}. Then the following diagram commutes up to canonical isomorphisms
\[
\xymatrix{
\D^{A_1\times A_2} \ar[r]^-{F_1} \ar[d]_-{F_2} \ar[dr]|{F_1 \boxtimes F_2}& \D^{B_1\times A_2} \ar[d]^-{F_2} \\
\D^{A_1\times B_2} \ar[r]_-{F_1} & \D^{B_1\times B_2}.
}
\]
\end{lem}
\begin{proof}
This is immediate from the above discussion.
\end{proof}

Thus, left admissible morphisms in unrelated variables commute. The product~$\boxtimes$ is also compatible with compositions in the following sense.

\begin{lem} \label{lem:Prof-monidal}
Let $F_1\colon \D^{A_1} \to \D^{B_1}$, $F_2\colon \D^{A_2} \to \D^{B_2}$, $G_1\colon \D^{B_1} \to \D^{C_1}$, and $G_2\colon \D^{B_2} \to \D^{C_2}$ be left admissible morphisms with \D a stable derivator. Then
\[ 
(G_1 \circ F_1) \boxtimes (G_2 \circ F_2) \cong (G_1 \boxtimes G_2) \circ (F_1 \boxtimes F_2) 
\]
as morphisms of derivators $\D^{A_1 \times A_2} \to \D^{C_1 \times C_2}$.
\end{lem}
\begin{proof}
Using \autoref{thm:kernel} this is merely a reformulation of \autoref{lem:mon-bicat}.
\end{proof}

We conclude the section by collecting two immediate consequences of \autoref{lem:la-dual} regarding opposites of admissible morphisms.
Specializing for simplicity to $\V = \cSp$ only (although this is not strictly necessary in view of~\cite{gs:enriched}), these corollaries are useful in later computations. The first one corresponds to \autoref{lem:adm-adj-triple}.

\begin{cor} \label{cor:adm-dual}
Suppose that $F \cong M \otimes_{[A]} - \cong - \lhd_{[A]} N$ is admissible with $M \in \cSp(B \times A\op)$ and $N \in \cSp(A \times B\op)$. Then $F\op \cong N \otimes_{[A\op]} - \cong - \lhd_{[A\op]} M$. In particular, there is an adjoint triple
\[(M\otimes_{[B\op]} -, F\op, -\lhd_{[B\op]} N).\]
\end{cor}

The second consequence complements \autoref{cor:inverse-modules}. Here we also use the notation $M\inv \in \cSp(A \times B\op) \cong \cSp(B\op \times (A\op)\op)$ for the spectral bimodule representing the inverse of an admissible equivalence $F \cong M \otimes_{[A]} -$.

\begin{cor} \label{cor:adm-eq-dual}
Suppose that $(F_\D, G_\D) \cong (M \otimes_{[A]} -,M\inv \otimes_{[B]} -)\colon \D^A \rightleftarrows \D^B$ is an admissible equivalence. Then
\[ (F_{\D\op}\op, G_{\D\op}\op) \cong (M\inv \otimes_{[A\op]}-, M\otimes_{[B\op]}-)\colon \D^{A\op} \rightleftarrows \D^{B\op} \]
is also an admissible equivalence.
\end{cor}

\begin{proof}
By definition we assume that $F_\E \cong M \otimes_{[A]} -\colon \E^A \to \E^B$ is an equivalence for all stable derivators $\E$ (see also \autoref{rmk:kernel-admissible}(i)). Applying this to $\E = \D\op$ and using \autoref{cor:inverse-modules}, we obtain an equivalence $(F_{\D\op}, G_{\D\op}) \cong (- \lhd_{[A]} M\inv,-\lhd_{[B]} M)\colon (\D\op)^A \rightleftarrows (\D\op)^B$. Using further that $((\D\op)^A)\op \cong \D^{A\op}$ and $((\D\op)^B)\op \cong \D^{B\op}$, and applying \autoref{lem:la-dual}, we obtain an equivalence $(F_{\D\op}\op, G_{\D\op}\op) \cong (M\inv \otimes_{[A\op]}-,M\otimes_{[B\op]}-)\colon \D^{A\op} \rightleftarrows \D^{B\op}$.
\end{proof}

\section{Universal tilting modules}
\label{sec:universal-tilting}

In this section and \S\ref{sec:yoneda} we construct certain universal bimodules over spectral path algebras. Here we begin with the bimodules which realize the functors studied in \S\S\ref{sec:An}-\ref{sec:coxeter-reflection} in \emph{arbitrary stable derivators}, hence, in particular, in the stable derivators associated to fields, rings, schemes, differential-graded algebras, ring spectra, and also in further stable derivators arising in algebra, geometry, and topology (see \cite[\S5]{gst:basic}). Among the functors obtained this way are the reflection functors, the strong stable equivalences, the Coxeter functors, and the Serre functors. This allows us to see that for all $A_n$-quivers Serre functors arise as universal Nakayama functors. As a further illustration, we construct a very explicit universal tilting module realizing a strong stable equivalence between the commutative square and $D_4$-quivers.

In this section and the next one we focus on stable derivators with their canonical closed $\cSp$-module structure but there are similar results for stable derivators endowed with closed $\V$-module structures for \V a stable, closed symmetric monoidal derivator.

\subsection{Universal tilting modules for reflection functors}
\label{subsec:univ-APR}

We begin with the simplest case in which we consider an $A_n$-quiver~$Q$ together with a sink $a \in Q$ and the reflected quiver $Q' = \sigma_a Q$. As we know from \S\ref{subsec:Coxeter}, there is a strong stable equivalence $(s^-_a,s^+_a)\colon\D^{Q'}\rightleftarrows \D^Q$ given by certain \emph{reflection functors} $s^-_a,s^+_a$. Note also that the sink $a\in Q$ is a sink in the opposite of the reflected quiver~$(Q')\op$ as well as a source in both the reflected quiver~$Q'$ and in the opposite quiver~$Q\op$.

\begin{lem} \label{lem:refl-admis}
Let \D be a stable derivator, let $Q$ be an $A_n$-quiver with a sink $a \in Q$, and let $Q' = \sigma_a Q$ be the reflected quiver.
\begin{enumerate}
\item The reflection functors $(s^-_a,s^+_a)\colon \D^{Q'} \rightleftarrows \D^Q$ and the opposite reflection functors $((s^-_a)\op,(s^+_a)\op)\colon (\D\op)^{(Q')\op} \rightleftarrows (\D\op)^{Q\op}$ are admissible equivalences.
\item The opposite reflections $((s^-_a)\op,(s^+_a)\op)\colon (\D\op)^{(Q')\op} \rightleftarrows (\D\op)^{Q\op}$ are naturally isomorphic to the reflections $(s^+_a,s^-_a)$ for the opposite quivers $(Q')\op$,$Q\op$ applied to $\D\op$, respectively.
\end{enumerate}
\end{lem}

\begin{proof}
(i) If we choose inclusions $i_Q\colon Q \to M_n$ and $i_{Q'}\colon Q' \to M_n$ so that they satisfy \autoref{hyp:rel-embed}, the composition $i_{Q}^\ast F_{Q'}\colon \D^{Q'} \to \D^Q$ is an equivalence by \autoref{thm:AR-independent}, and it is also isomorphic to $s^-_a$ by construction. Since restriction morphisms are admissible, it suffices by \autoref{prop:sa} to show that the morphism $F_Q\colon \D^Q \to \D^{M_n}$ and its opposite are admissible. Similar to $F=F_{\A{n}}$ as defined by \eqref{eq:AR}, these morphisms are compositions of left and right extensions by zero and Kan extensions adding bicartesian squares and are hence admissible.

(ii) This is immediate from the definitions of the respective reflections.
\end{proof}

As admissible equivalences, reflection functors are induced by invertible spectral bimodules (see \autoref{cor:inverse-modules}). Inspired by~\cite{apr:tilting}, we make the following definition.

\begin{defn} \label{defn:apr-tilt}
Let $Q$ be an $A_n$-quiver with a sink $a \in Q$ and let $Q' = \sigma_a Q$ be the reflected quiver. The bimodules 
$T^-_a=T^-_{Q',a}\in\cSp(Q\times (Q')\op)$ and $T^+_{a}=T^+_{Q,a}\in\cSp(Q'\times Q\op)$
such that
\[
s^-_a \cong T^-_a \otimes_{[Q']} -
\qquad \textrm{and} \qquad
s^+_a \cong T^+_a \otimes_{[Q]} -
\]
are \textbf{universal Auslander--Platzeck--Reiten tilting bimodules} (or \textbf{universal APR tilting bimodules}) for reflection functors.
\end{defn}

By definition, these spectral bimodules $T^+_a,T^-_a$ represent the corresponding reflection functors in \emph{any stable derivator}, justifying that we refer to them as being `universal'. As we mention in \S\ref{sec:field}, by inducing up to chain complexes over a field, these spectral bimodules yield the usual tilting bimodules from \cite{apr:tilting}.

The proof of \autoref{thm:kernel} shows that these spectral bimodules can be obtained by applying $s^-_a\colon\cSp^{Q'}\to\cSp^{Q}$ and $s^+_a\colon\cSp^{Q}\to\cSp^{Q'}$ to the respective identity bimodules, i.e., we have
\[
T^-_a=s^-_a(\lI_{Q'})\in\cSp(Q\times (Q')\op)\qquad\text{and}\qquad T^+_a=s^+_a(\lI_{Q})\in\cSp(Q'\times Q\op).
\]
This implies that the tilting bimodules satisfy some relations; see \autoref{cor:inverse-modules}.

\begin{lem} \label{lem:apr-commute}
Let $Q$ be an $A_n$-quiver.
\begin{enumerate}
\item If $a \in Q$ is a sink and $Q' = \sigma_a Q$, then there are isomorphisms 
\[
T^-_a \otimes_{[Q']} T^+_a \cong \lI_{Q}\qquad \text{and} \qquad T^+_a \otimes_{[Q]} T^-_a \cong \lI_{Q'}.
\]
\item For sinks $a_1\neq a_2 \in Q$ and $\sigma_{\{a_1,a_2\}}Q=\sigma_{a_1}\sigma_{a_2} Q=\sigma_{a_2}\sigma_{a_1} Q$ there are isomorphisms
\[
T^+_{\sigma_{a_2} Q,a_1} \otimes_{[\sigma_{a_2} Q]} T^+_{Q,a_2} \cong T^+_{\sigma_{a_1} Q,a_2} \otimes_{[\sigma_{a_1} Q]} T^+_{Q,a_1}.
\]
\end{enumerate}
\end{lem}

\begin{proof}
Statement (i) is an instance of \autoref{cor:inverse-modules} and (ii) is immediate from \autoref{cor:sasb-commute}.
\end{proof}

The following computational lemma can be viewed as an abstract analog of~\cite[Proposition 1.4]{miyashita:tilting}.

\begin{lem} \label{lem:opposite-tilting}
Let $Q$ be an $A_n$-quiver, $a \in Q$ be a sink and $Q' = \sigma_a Q$. If $T^+_a \in \cSp(Q' \times Q\op)$ is the universal tilting module for $s^+_a\colon \D^Q \to \D^{Q'}$, then it is also the universal tilting module for $s^+_a\colon \D^{(Q')\op} \to \D^{Q\op}$, and dually. In particular, there are canonical isomorphisms
\[
\lI_{Q'} \cong (s^+_a \boxtimes s^-_a)(\lI_Q)
\qquad \textrm{and} \qquad
\lI_Q \cong (s^-_a \boxtimes s^+_a)(\lI_{Q'}).
\]
\end{lem}

\begin{proof}
Since we have $(s^-_a,s^+_a) \cong (T^-_a \otimes_{[Q']} -,T^+_a \otimes_{[Q]} -) \colon \D^{Q'} \rightleftarrows \D^Q$, the first statement follows directly from \autoref{lem:refl-admis}(ii) and \autoref{cor:adm-eq-dual}.

By \autoref{defn:boxtimes}, the morphism $s^+_a \boxtimes s^-_a\colon \D^{Q \times Q\op} \to \D^{Q' \times (Q')\op}$ is represented by the external tensor product $T^+_a \otimes T^-_a$. We claim that there are isomorphisms
\[ (s^+_a \boxtimes s^-_a)(\lI_Q) \cong (T^+_a \otimes T^-_a) \otimes_{[Q \times Q\op]} \lI_Q \cong T_a^+ \otimes_{[Q]} \lI_Q \otimes_{[Q]} T_a^- \cong T_a^+ \otimes_{[Q]} T_a^- \cong \lI_{Q'}. \]
The third one follows from \autoref{thm:bicategory} and the last one from \autoref{lem:apr-commute}(i). Finally, the second isomorphism follows from \autoref{lem:tensor-opp} and Fubini's lemma \cite[Lemma 5.3]{gps:additivity}.
The isomorphism $\lI_Q \cong (s^-_a \boxtimes s^+_a)(\lI_{Q'})$ is obtained by a similar argument.
\end{proof}

We can in a straightforward way extend the above discussion to compositions of reflection functors. Let $Q$ and $Q'$ be two $A_n$-quivers with a possibly different orientation and let $i_Q\colon Q \to M_{n}$ and $i_{Q'}\colon Q' \to M_{n}$ be two embeddings as in~\eqref{eq:iQ}. Then according to \autoref{prop:tiltAn}, $(i_{Q}^\ast F_{Q'},i_{Q'}^\ast F_{Q})\colon\D^{Q'}\rightleftarrows \D^Q$ is an admissible equivalence obtained by iterated reflections.

\begin{defn} \label{defn:iter-tilt}
Let $Q,Q'$ be two $A_n$-quivers and let $(i_{Q}^\ast F_{Q'},i_{Q'}^\ast F_{Q})\colon\D^{Q'}\rightleftarrows \D^Q$ be an equivalence as above. The spectral bimodules $T_{Q,Q'} \in \cSp(Q\times (Q')\op)$ and $T_{Q',Q} \in \cSp(Q'\times Q\op)$ such that
\[ i_{Q}^\ast F_{Q'} \cong T_{Q,Q'} \otimes_{[Q']} - \qquad \textrm{and} \qquad i_{Q'}^\ast F_{Q} \cong T_{Q',Q} \otimes_{[Q]} - \]
are the \textbf{universal tilting bimodules}.
\end{defn}

Of course, we again have $T_{Q,Q'} \otimes_{[Q']} T_{Q',Q} \cong \lI_Q$ and $T_{Q',Q} \otimes_{[Q]} T_{Q,Q'} \cong \lI_{Q'}$, and \autoref{lem:opposite-tilting} takes the following form.

\begin{lem} \label{lem:opposite-iter-tilt}
Let $Q,Q'$ be two $A_n$-quivers and suppose that $T_{Q',Q}$ represents an equivalence $\D^Q \to \D^{Q'}$ given by iterated reflections. Then $T_{Q',Q}$  represents also the equivalence $\D^{(Q')\op} \to \D^{Q\op}$ given by the \emph{same} sequence of reflections.
\end{lem}

\begin{proof}
This follows from \autoref{lem:opposite-tilting} by induction.
\end{proof}

\subsection{Universal tilting modules for Coxeter and Serre functors}
\label{subsec:univ-Coxeter}

As in \S\ref{sec:coxeter-reflection}, we will also consider autoequivalences $\D^Q \to \D^Q$ induced by iterated reflections, and in particular the Coxeter functors $(\Phi^-,\Phi^+)\colon \D^Q \rightleftarrows \D^Q$ (\autoref{defn:Coxeter}) and the Serre functor $S\colon \D^Q \to \D^Q$. Recall from \autoref{defn:Serre} and \autoref{thm:Serre-Nakayama} that $S=\Sigma\circ\Phi^+\cong\Phi^+\circ \Sigma$ and that for the linear orientation $\A{n}$, the corresponding spectral bimodule is the canonical duality bimodule $D_n$. Here we extend the latter result to arbitrary orientations.

To start with, we restate basic results about the Coxeter functors in terms of bimodules. That is, as a special case of \autoref{defn:iter-tilt}, there are bimodules $C^-_Q, C^+_Q$ in $\cSp(Q \times Q\op)$ and natural isomorphisms
\[ 
\Phi^- \cong C_Q^- \otimes_{[Q]} -
\qquad \textrm{and} \qquad
\Phi^+ \cong C_Q^+ \otimes_{[Q]} -.
\]

\begin{lem} \label{lem:Cox-ker}
Let $Q$ be an $A_n$-quiver with an arbitrary orientation. For every admissible sequence of sinks $(a_1, \dots, a_n)$ and any admissible sequence of sources $(a'_1, \dots, a'_n)$, there are isomorphisms
\[
C_Q^- \cong T_{a'_n}^- \otimes_{[Q'_{n-1}]} \cdots \otimes_{[Q'_1]} T_{a'_1}^-
\qquad \textrm{and} \qquad
C_Q^+ \cong T_{a_n}^+ \otimes_{[Q_{n-1}]} \cdots \otimes_{[Q_1]} T_{a_1}^+.
\]
\end{lem}

\begin{proof}
This follows immediately from \autoref{prop:indep-admis} and the uniqueness of representing bimodules.
\end{proof}

Recall from \autoref{defn:dualizing} that the canonical duality module $D_Q \in \cSp(Q \times Q\op)$ is defined as $D_Q = \lI_Q \rhd \lS$ (we continue dropping restrictions along symmetry constraints from notation). Since $Q$ can also be viewed as a Hasse diagram of a poset, we can use the direct description of $D_Q$ provided by \autoref{lem:poset}(ii). Thus, for instance, if $Q$ is the $A_3$-quiver $(1 \leftarrow 2 \rightarrow 3)$, then $\lI_Q$ and $D_Q$ have the following diagrams, respectively (compare to \autoref{eg:D2}):
\begin{equation}\label{eq:D-of-A3}
\vcenter{
\xymatrix@=1.0em{
\lS\ar[d]&0\ar[l]\ar[r]\ar[d]&0\ar[d]&&& 
\lS\ar[d]&\lS\ar[l]\ar[r]\ar[d]&0\ar[d]\\
\lS&\lS\ar[l]\ar[r]&\lS && &
0&\lS\ar[l]\ar[r]&0\\
0\ar[u]&0\ar[l]\ar[r]\ar[u]&\lS\ar[u] && &
0\ar[u]&\lS\ar[l]\ar[r]\ar[u]&\lS\ar[u]
}
}
\end{equation}

In complete analogy to the linear orientation studied in \autoref{thm:Serre-Nakayama}, we will obtain a natural isomorphism between the Nakayama functor $D_Q \otimes_{[Q]} -\colon \D^Q \to \D^Q$ and the Serre functor $S = \Sigma\circ\Phi^+$. Before proving this theorem, we need a lemma about the compatibility of the reflection functors with the duality morphism $-\rhd\lS\colon (\cSp\op)^Q \to \cSp^Q$. Intuitively, if $X \in (\Sp\op)^Q \cong (\Sp^{Q\op})\op$ is a $Q\op$-shaped diagram of spectra and $a\in Q$ is a sink, then dualizing $X$ and reflecting it at $a$ yields the same result as first reflecting $X$ at the source $a \in Q\op$ and then dualizing $X$. That is, the diagram
\[
\xymatrix{
(\Sp^{Q\op})\op \ar[r]^-\cong \ar[d]_{(s^-_a)\op} & (\cSp\op)^{Q} \ar[r]^-{-\rhd\lS} & \cSp^Q \ar[d]^{s^+_a}   \\
(\Sp^{(Q')\op})\op \ar[r]_-\cong & (\cSp\op)^{Q'} \ar[r]_-{-\rhd\lS} & \cSp^{Q'}
}
\]
should commute up to a natural isomorphism. Using universal tilting bimodules, we obtain the following more general statement.

\begin{lem} \label{lem:SPdual}
Let $Q,Q'$ be $A_n$-quivers and let $T \in \cSp(Q \times (Q')\op)$ be a universal tilting bimodule as in \autoref{defn:iter-tilt}.
The following diagram
\[
\xymatrix{
(\Sp^{Q\op})\op \ar[r]^-\cong \ar[d]_{(T \otimes_{[Q\op]} -)\op} & (\cSp\op)^{Q} \ar[r]^-{-\rhd\lS} & \cSp^Q \ar[d]^{T \rhd_{[Q]} -}   \\
(\Sp^{(Q')\op})\op \ar[r]_-\cong & (\cSp\op)^{Q'} \ar[r]_-{-\rhd\lS} & \cSp^{Q'}
}
\]
commutes up to a natural isomorphism and the vertical morphisms are equivalences.
\end{lem}

\begin{proof}
By definition, $T \otimes_{[Q']} -\colon \cSp^{Q'} \to \cSp^Q$ is an equivalence. Hence so is $- \lhd_{[Q]} T \cong T \rhd_{[Q]} -\colon \cSp^Q \to \cSp^{Q'}$ and also $T \otimes_{[Q\op]} -\colon \cSp^{Q\op} \to \cSp^{(Q')\op}$ by \autoref{lem:opposite-iter-tilt}. The commutativity of the square is immediate since there are natural isomorphisms $(T \otimes_{[Q\op]} X) \rhd \lS \cong T \rhd_{[Q]} (X \rhd \lS)$ for arbitrary $X\in(\Sp^{Q\op})\op$.
\end{proof}

Now we can prove the theorem.

\begin{thm} \label{thm:Serre-Nakayama-indep}
Let $\D$ be a stable derivator and $Q$ be an arbitrarily oriented $A_n$-quiver. There is an isomorphism $\Sigma(C_Q^+) \cong D_Q$ in $\cSp(Q \times Q\op)$. In particular, we have a natural isomorphism between Serre and Nakayama functors
\[ S \cong D_Q \otimes_{[Q]} -\colon \D^Q \to \D^Q. \]
\end{thm}

\begin{proof}
Clearly the suspension $\Sigma\colon \cSp^Q \to \cSp^Q$ is an admissible morphism and has kernel $\Sigma(\lI_Q) \in \cSp(Q \times Q\op)$. Thus $\Sigma(C_Q^+) \cong \Sigma(\lI_Q) \otimes_{[Q]} C_Q^+$ and we can rephrase the conclusion of \autoref{thm:Serre-Nakayama} as the existence of an isomorphism
\[ \Sigma(C_{\A{n}}^+) \cong D_n = \lI_{\A{n}} \rhd \lS \]
in $\cSp(\A{n}\times \A{n}\op)$.

Consider an admissible equivalence
\[ (G,G\inv) \cong (T_{\A{n},Q} \otimes_{[Q]} -, T_{Q,\A{n}} \otimes_{[\A{n}]} -)\colon \D^Q \rightleftarrows \D^{\A{n}} \]
obtained as a composition of reflections, where $T = T_{\A{n},Q}$ and $T' = T_{Q,\A{n}}$ are the corresponding universal tilting bimodules.
By two applications of \autoref{lem:Serre-match}, the following square commutes up to a natural isomorphism
\[
\xymatrix{
\D^Q \ar[d]_G \ar[r]^S & \D^Q \ar[d]^G \\
\D^{\A{n}} \ar[r]_S & \D^{\A{n}}.
}
\]
Hence our task is reduced to proving that there is an isomorphism
\begin{equation} \label{eq:nakQ-ker}
D_Q \cong T' \otimes_{[\A{n}]} D_n \otimes_{[\A{n}]} T.
\end{equation}

Note first that \autoref{cor:inverse-modules} yields canonical isomorphisms
\begin{align} \label{eq:nak-comp}
T' \otimes_{[\A{n}]} D_n &= T' \otimes_{[\A{n}]} (\lI_{\A{n}} \rhd \lS) \\
&\cong T \rhd_{[\A{n}]} (\lI_{\A{n}} \rhd \lS) \\
&\cong (T \otimes_{[\A{n}]} \lI_{\A{n}}) \rhd \lS \\
&\cong T \rhd \lS
\end{align}
in $\cSp$. Hence, proving~\eqref{eq:nakQ-ker} reduces to proving that $\lI_Q \rhd \lS = D_Q \cong (T \rhd \lS) \otimes_{[\A{n}]} T$.
This isomorphism is obtained by a further computation in $\cSp$,
\begin{align} \label{eq:nak-comp2}
(T \rhd \lS) \otimes_{[\A{n}]} T &\cong T \otimes_{[\A{n}\op]} (T \rhd \lS) \\
&\cong T' \rhd_{[\A{n}\op]} (T \rhd \lS) \\
&\cong (T' \otimes_{[\A{n}\op]} T) \rhd \lS \\
&\cong \lI_Q \rhd \lS,
\end{align}
using \autoref{lem:tensor-opp} and \autoref{cor:inverse-modules} again.
\end{proof}

We will see in \S\ref{sec:field} that the spectral bimodules constructed so far yield non-trivial elements in spectral Picard groups $\Pic_{\cSp}(Q)$, $Q$ a Dynkin quiver of type~$A$. In fact, for arbitrary fields $k$ there is a split epimorphism $\Pic_{\cSp}(Q)\to\Pic_{\D_k}(Q)$ and the group $\Pic_{\D_k}(Q)$ is known to be non-trivial and can be described explicitly.

\subsection{A very explicit universal tilting bimodule}

As a further illustration of the techniques of \S\S\ref{sec:admissible}-\ref{sec:kernel}, in this short subsection we revisit the strong stable equivalence between the commutative square and $D_4$-quivers (see \cite[\S9.1]{gst:basic}). This example illustrates that the spectral bimodules of \autoref{thm:kernel} and \autoref{cor:inverse-modules} can be constructed rather explicitly.  

Let $Q=\square=[1]\times[1]$ be the commutative square, let \D be a stable derivator, and let $X\in\D^Q$ have underlying diagram as depicted in the diagram to the left in \autoref{fig:square-D4}.
\begin{figure}[h]
\centering
\[
\xymatrix{
x\ar[r]\ar[d]&y\ar[ddr]& & x\ar[r]\ar[d]&y\ar[d]& & &y\ar[d]&\\
z\ar[rrd]&& & z\ar[r]&p\ar[dr]& & z\ar[r]&p\ar[dr]&\\
&&w, & &&w, & &&w
}
\]
\caption{Strong stable equivalence between $\square$ and a $D_4$-quiver}
\label{fig:square-D4}
\end{figure}
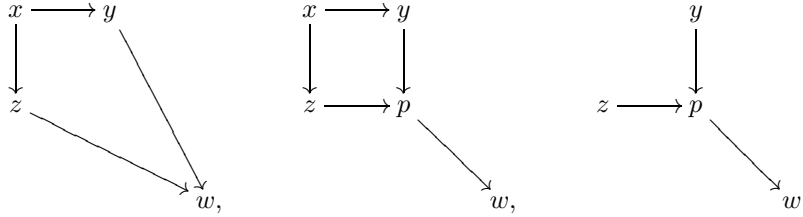
As described in \cite[Theorem~9.2]{gst:basic}, starting from $X$, passing to a left Kan extension which amounts to adding a cocartesian square yielding a coherent diagram of shape $B$ (see the middle of \autoref{fig:square-D4}) and then restricting to the $D_4$-quiver $D$ as depicted on the right in \autoref{fig:square-D4}, we obtain a strong stable equivalence
\[
(F,G)\colon\D^Q\rightleftarrows\D^D.
\]
Clearly, this equivalence is admissible and \autoref{cor:inverse-modules} hence yields an invertible spectral bimodule $T_{D,Q}\in\cSp(D\times Q\op)$ and a natural isomorphism
\[
F\cong T_{D,Q}\otimes_{[Q]}-\colon\D^Q\to\D^D.
\]
Moreover, the proof of \autoref{thm:kernel} gives us an explicit construction of $T_{D,Q}$. In fact, starting with the identity profunctor $\lI_Q\in\cSp(Q\times Q\op)$ it suffices to apply the above left Kan extension and restriction morphisms to $\lI_Q\in\cSp^Q$. Since the square is a poset, \autoref{lem:poset} implies that $\lI_Q\in\cSp(Q\times Q\op)$ looks like the restriction along the embedding $Q\times Q\op\to B\times Q\op$ of the object in $\cSp(B\times Q\op)$ which in turn looks like \eqref{eq:TDQ}, 
\begin{equation}\label{eq:TDQ}
\vcenter{
\xymatrix{
\lS\ar[r]\ar[d]&\lS\ar[d] & && 
0\ar[r]\ar[d]&\lS\ar[d]\\
\lS\ar[r]&\lS\ar[rd]\ar@{}[ddd]|(0.4){\uparrow}\ar@{}[rrr]|(0.6){\ot}& && 
0\ar[r]&\lS\ar[rd]\ar@{}[ddd]|(0.4){\uparrow}& \\
&&\lS && 
&&\lS\\
0\ar[r]\ar[d]&0\ar[d] & && 
0\ar[r]\ar[d]&0\ar[d]\\
\lS\ar[r]&\lS\ar[rd]\ar@{}[rrr]|(0.6){\ot}& && 
0\ar[r]&0\ar[rd] \\
&&\lS &&
&&\lS.
}
}
\end{equation}
Moreover, \autoref{lem:poset} also implies that all maps in $\lI_Q$ which can be isomorphisms actually are isomorphisms. Since pushouts of isomorphisms are isomorphisms it is immediate that the first step of the strong stable equivalence $F\colon\D^Q\to\D^D$ sends~$\lI_Q$ to a coherent diagram \eqref{eq:TDQ} in which still all maps are isomorphisms if possible. Finally, the universal tilting module $T_{D,Q}\in\cSp(D\times Q\op)$ is obtained by restricting \eqref{eq:TDQ} along the embedding $D\times Q\op\to B\times Q\op$.

In \cite{gst:tree} we constructed further examples of strong stable equivalences, which are obtained from reflection morphisms associated to trees. It turns out that also these reflections are induced by universal tilting modules. We plan to get back to this in \cite{gst:acyclic}.

\section{Yoneda bimodules and a spectral Serre duality result}
\label{sec:yoneda}

In this section we introduce universal constructors for coherent Auslander--Reiten quivers associated to $A_n$-quivers (which specialize to universal constructors for \emph{higher triangles} as we see in~\S\ref{sec:higher}). Moreover, it turns out that for each fixed $n\in\lN$ there is one spectral bimodule $U_n\in\cSp(M_n\times M_n\op)$ which restricts to all the remaining universal tilting bimodules related to arbitrary $A_n$-quivers. We also show that these Yoneda bimodules $U_n$ are self-dual up to a twist, a result which can be interpreted as a spectral version of classical Serre duality.

\subsection{Universal constructors for coherent Auslander--Reiten quivers}
\label{subsec:AR_Q}

Given a quiver $Q$ of type $A_n$, both $F_Q\colon \D^Q \to \D^{M_n}$ and its opposite are admissible by construction. Hence we can make the following definition.

\begin{defn}\label{defn:AR-con}
The bimodules $\ARQ \in \cSp(M_{n} \times Q\op)$ and $\ARQinv \in \cSp(Q \times M_{n}\op)$ such that
\[ F_Q \cong \ARQ \otimes_{[Q]} - \qquad \textrm{and} \qquad F_Q\op \cong \ARQinv \otimes_{[Q\op]} - \]
are \textbf{universal constructors for coherent Auslander--Reiten quiver}.
\end{defn}

By definition for every $X\in\D^Q$ the tensor product $\ARQ \otimes_{[Q]} X \in \D^{M_n,\exx}$ is the associated coherent AR-quiver. Let us justify the notation for $\ARQinv$ (related to this see also \autoref{warn:AR-con}).

\begin{lem} \label{lem:ar-destructor}
For every stable derivator \D and $A_n$-quiver $Q$ there are natural isomorphisms
%
\[
(F_Q,i_Q^\ast) \cong (\ARQ \otimes_{[Q]}-,\ARQinv \otimes_{[M_{n}]} -)\colon \D^Q\rightleftarrows \D^{M_n,\exx}.
\]
\end{lem}

\begin{proof}
An application of \autoref{cor:adm-dual} to $F_Q\op \cong \ARQinv \otimes_{[Q\op]} -$ yields a natural isomorphism $F_Q \cong - \lhd_{[Q]} \ARQinv\colon \D^Q \to \D^{M_n}$. Hence there is an adjunction $(\ARQinv \otimes_{[M_n]}-,F_Q)\colon \D^{M_n} \rightleftarrows \D^Q$. Since the essential image of $F_Q$ lies in $\D^{M_n,\exx}$, this restricts to a further adjunction $(\ARQinv \otimes_{[M_n]}-,F_Q)\colon \D^{M_n,\exx} \rightleftarrows \D^Q$ which actually is an equivalence.
\end{proof}

\begin{warn}\label{warn:AR-con}
Note that \autoref{lem:ar-destructor} does not immediately follow from \autoref{cor:inverse-modules} because $\D^{M_n,\exx}$ is not simply a shifted derivator, but a full subderivator of a shifted derivator determined by vanishing and exactness properties. Related to this we remark that \autoref{lem:ar-destructor} yields an isomorphism $\ARQinv\otimes_{[M_n]}\ARQ\cong\lI_Q$. However, the tensor product $\ARQ\otimes_{[Q]}\ARQinv\in\cSp(M_n\times M_n\op)$ is \emph{different} from~$\lI_{M_n}$ as we observe in \autoref{warn:Yoneda}.
\end{warn}

It also follows that $\ARQ$ allows us to recover the universal tilting bimodules from \S\ref{sec:universal-tilting} by restriction.

\begin{lem} \label{lem:ARQ-to-tilt}
Let $\D$ be a stable derivator, $Q,Q'$ two $A_n$-quivers and $T_{Q',Q}$ be a universal tilting bimodule as in \autoref{defn:iter-tilt}. Then there is an isomorphism
\[ T_{Q',Q} \cong (i_{Q'}\times\id_{Q\op})^\ast(\ARQ). \]
\end{lem}

\begin{proof}
We have $T_{Q',Q} \cong i_{Q'}^\ast(F_Q(\lI_Q)) \cong i_{Q'}^\ast(\ARQ)$.
\end{proof}

\subsection{The Yoneda bimodules of the mesh categories}
\label{subsec:univ}

We begin by observing that the morphism $F_Q \boxtimes F_Q\op\colon \D^{Q \times Q\op} \to \D^{M_{n} \times M_{n}\op}$ restricts to a natural equivalence of stable derivators. 
In fact, since the mesh category $M_n$ is self-dual, there is an induced equivalence $\D^{M_n}\simeq\D^{M_n\op}$. If we denote the essential image of $\D^{M_n,\exx}$ by $\D^{M_n\op,\ex}$, then we obtain an induced equivalence $\D^{M_n,\ex}\simeq\D^{M_n\op,\ex}$.
For each stable derivator $\D$, the opposite of $F_Q$ applied to $\D\op$ then induces an equivalence $F_Q\op\colon \D^{Q\op} \to \D^{M_n\op,\exx}$. It follows from \autoref{lem:boxtimes} that there is a well-defined equivalence
\begin{equation}\label{eq:FQFQ}
F_Q \boxtimes F_Q\op\colon \D^{Q \times Q\op} \to \D^{M_{n} \times M_{n}\op,\exx},
\end{equation}
where the vanishing and exactness conditions on $\D^{M_n\times M_n\op,\exx}$ are imposed on each variable separately. We temporarily denote by $U_Q$ the image of the spectral identity profunctor under this equivalence,
\[
U_Q=(F_Q\boxtimes F_Q\op)(\lI_Q)\in\cSp(M_n\times M_n\op).
\]

\begin{rmk} \label{rmk:general-universal}
We can express $U_Q$ also as follows
\[ U_Q \cong \ARQ \otimes_{[Q]} \ARQinv \cong \ARQ \lhd_{[Q\op]} \ARQ \cong \ARQ \rhd_{[Q\op]} \ARQ. \]
The first isomorphism holds since $U_Q \cong (\ARQ \otimes \ARQinv) \otimes_{[Q \times Q\op]} \lI_Q \cong \ARQ \otimes_{[Q]} \lI_Q \otimes_{[Q]} \ARQinv$ by Fubini's lemma and \autoref{lem:tensor-opp}. The second isomorphism follows from \autoref{cor:adm-dual} by the same arguments as in the proof of \autoref{lem:ar-destructor}, and the last one needs essentially just the symmetry of $\otimes$ in $\cSp$.
\end{rmk}

\begin{thm} \label{thm:U_n-indep}
For $A_n$-quivers $Q,Q'$ there is an isomorphism $U_Q \cong U_{Q'}$.
\end{thm}

\begin{proof}
We know that $\D^{Q'} \simeq \D^Q$ by a series of reflections and by induction we can reduce to the situation that $Q$ and $Q'$ differ by a single reflection only. That is, $Q' = \sigma_a Q$ for a sink $a \in Q$ and we can choose embeddings of categories $i_Q\colon Q \to M_{n}$ and $i_{Q'}\colon Q' \to M_{n}$ as in \autoref{hyp:rel-embed} so that
\[ (s_a^-,s_a^+) \cong (i_{Q}^\ast F_{Q'},i_{Q'}^\ast F_{Q})\colon\D^{Q'}\rightleftarrows \D^Q. \]
In particular $F_Q s_a^- \cong F_Q i_{Q}^\ast F_{Q'} \cong F_{Q'}$ as morphisms $\D^{Q'} \to \D^{M_n,\exx}$.

Considering also the opposite of the latter isomorphism, we obtain a triangle of derivators commuting up to a natural isomorphism
\[
\xymatrix{
(\D\op)^{(Q')\op} \ar[rr]^-{(s_a^-)\op} \ar[dr]_-{F_{Q'}\op} && (\D\op)^{Q\op} \ar[dl]^-{F_Q\op} \\
& (\D\op)^{M_{n}\op,\exx}.
}
\]
Since $(s_a^-)\op \cong s_a^+$ by \autoref{lem:refl-admis}(ii), we have $F_Q\op s_a^+ \cong F_{Q'}\op$. \autoref{lem:Prof-monidal} yields a further triangle of equivalences which commutes up to a natural isomorphism
\[
\xymatrix{
\D^{Q' \times (Q')\op} \ar[rr]^-{s_a^- \boxtimes s_a^+} \ar[dr]_-{F_{Q'} \boxtimes F_{Q'}\op\,} && \D^{Q \times Q\op} \ar[dl]^-{F_Q \boxtimes F_Q\op} \\
& \D^{M_{n} \times M_{n}\op,\exx}.
}
\]
Now $\lI_Q \cong (s_a^- \boxtimes s_a^+)(\lI_{Q'})$ by \autoref{lem:opposite-tilting} and hence we obtain the desired isomorphism $U_{Q'} = (F_{Q'} \boxtimes F_{Q'}\op)(\lI_{Q'}) \cong (F_Q \boxtimes F_Q\op)(\lI_Q) = U_Q$ as desired.
\end{proof}

By the theorem the following is independent of the choice of the $A_n$-quiver.

\begin{defn} \label{defn:general-universal}
Let $Q$ be an $A_n$-quiver. The \textbf{Yoneda bimodule} $U_n$ is the spectral bimodule
\[
U_n=(F_Q \boxtimes F_Q\op)(\lI_Q)\in\cSp(M_n\times M_n\op).
\]
\end{defn}

Before we justify the terminology, let us quickly note that all previously constructed bimodules are suitable restrictions of these Yoneda bimodules.

\begin{cor} \label{cor:restr_Un}
\begin{enumerate}
\item If $Q$ is an $A_n$-quiver with embedding $i_Q\colon Q \to M_{n}$, then $\ARQ \cong (\id_{M_{n}}\times i_Q\op)^\ast(U_n)$ and $\ARQinv \cong (i_Q \times \id_{M_{n}\op})^\ast(U_n)$.
\item If $Q,Q'$ are two $A_n$-quivers and $T_{Q',Q}$ is a universal tilting bimodule as in \autoref{defn:iter-tilt}, then $T_{Q',Q} \cong (i_{Q'}\times i_Q\op)^\ast(U_n) \in \cSp(Q' \times Q\op)$.
\end{enumerate}
\end{cor}

\begin{proof}
The first isomorphism of (i) is obtained from the following chain
\begin{align} \label{eq:nak-comp2}
(\id_{M_{n}}\times i_Q\op)^\ast(U_n) &\cong (\id_{M_{n}}\times i_Q\op)^\ast(\ARQ\otimes_{[Q]}\ARQinv)\\
&\cong \ARQ\otimes_{[Q]}\big((i_Q\op)^\ast\ARQinv\big) \\
&\cong \ARQ\otimes_{[Q]}\lI_Q\\
&\cong \ARQ.
\end{align}
Here, the first isomorphism is given by \autoref{rmk:general-universal}, the second one is pseudo-functoriality of coends with parameters (\autoref{rmk:variants}), and the third one follows immediately from \autoref{defn:AR-con}.  The second isomorphism in (i) is obtained similarly. Finally, (ii) is a consequence of (i) together with \autoref{lem:ARQ-to-tilt}.
\end{proof}

In particular, there is an isomorphism $(i_Q \times i_Q\op)^\ast(U_n) \cong \lI_Q$, and this already determines $U_n$ as an object of $\cSp^{M_n\times M_n\op,\exx}(\bbone)$. As for every small category, it follows from the explicit description of $\lI_Q\in\cSp(Q\times Q\op)$ in \autoref{thm:bicategory} that for $a,b\in Q$ there is a canonical isomorphism
\begin{equation}\label{eq:id-prof}
\lI_Q(a,b)\cong\coprod_{\hom_Q(b,a)}\lS
\end{equation}
in $\cSp(\bbone)\cong\mathcal{SHC}$. In fact, since $\lI_Q$ is obtained by left Kan extension along a discrete opfibration, the pointwise formula (Der4) simplifies accordingly (see \cite[Prop.~1.24]{groth:ptstab}). Thus, $\lI_Q\in\cSp(Q\times Q\op)$ can be thought of as the enriched hom-functor associated to the spectral category freely generated by $Q$, and the existence of an isomorphism 
\begin{equation}\label{eq:Yoneda-justify}
(i_Q \times i_Q\op)^\ast(U_n) \cong \lI_Q
\end{equation}
hence justifies that we refer to $U_n$ as the Yoneda bimodule. 

\begin{warn}\label{warn:Yoneda}
Note that although the Yoneda bimodule $U_n\in\cSp(M_n\times M_n\op)$ restricts to the identity profunctors $\lI_Q\in\cSp(Q\times Q\op)$, the module $U_n$ itself is not isomorphic to $\lI_{M_n}$. In fact, $U_n\in\cSp^{M_n\times M_n\op,\exx}(\bbone)$ satisfies the defining vanishing and exactness properties of $\cSp^{M_n\times M_n\op,\exx}$. In particular, for $(k,0)\in M_n$ we have $U_n((k,0),(k,0))\cong 0$ while $\lI_{M_n}((k,0),(k,0))\cong\lS$ by formula \eqref{eq:id-prof} which is valid for every small category.
\end{warn}

\subsection{A spectral Serre duality result}

We finish the section with a remarkable feature of the Yoneda bimodule $U_n$, namely the fact that it is almost self-dual. As is explained at the end of \S\ref{sec:field}, this fact is closely related to the classical Serre duality of $D^b(kQ)$. Let us again consider the functor $s\colon M_n\to M_n$ defined by \eqref{eq:s} and let us recall from \autoref{lem:Serre-shift} that it represents the Serre functor on $\D^{M_n,\exx}$.

\begin{thm} \label{thm:U_n-Serre}
For every $n \in \lN$ there is an isomorphism $ U_n \rhd \lS \cong (s \times \id)^\ast(U_n). $
\end{thm}

\begin{proof}
Let $Q$ be a quiver of type $A_n$. We must equivalently prove that there is an isomorphism
\begin{equation} \label{eq:U_n-Serre}
(i_Q \times i_Q\op)^\ast(U_n \rhd \lS) \cong (i_Q \times i_Q\op)^\ast (s \times \id)^\ast(U_n),
\end{equation}
where $(i_Q \times i_Q\op)^\ast\colon \cSp^{M_{n}\times M_{n}\op,\exx} \longrightarrow \cSp^{Q \times Q\op}$ is the inverse equivalence to $F_Q \boxtimes F_Q\op$ as above. In particular we have an isomorphism \eqref{eq:Yoneda-justify}. Together with the pseudo-functoriality of $-\rhd\lS$ this allows us to identify the left-hand side of \eqref{eq:U_n-Serre} as
\begin{equation}\label{eq:U_n-Serre-lhs}
\begin{split}
(i_Q \times i_Q\op)^\ast(U_n \rhd \lS) &\cong \big((i_Q \times i_Q\op)^\ast U_n\big) \rhd \lS)\\
&\cong \lI_Q\rhd\lS\\
&= D_Q.
\end{split}
\end{equation}

In order to understand the right-hand side of~\eqref{eq:U_n-Serre}, we consider the following diagram
\begin{equation} \label{eq:U_n-Serre-diag2}
\vcenter{
\xymatrix{
\big(\cSp^{M_{n}\op,\exx}\big)^{M_{n},\exx} \ar[d]_{s^\ast} \ar[r]^-{(i_Q\op)^\ast} &
\big(\cSp^{Q\op}\big)^{M_{n},\exx} \ar[d]_{s^\ast \cong S} \ar[r]^-{i_Q^\ast} &
\big(\cSp^{Q\op}\big)^Q \ar[d] \ar[d]^{S \cong D_Q \otimes_{[Q]} -}
\\
\big(\cSp^{M_{n}\op,\exx}\big)^{M_{n},\exx} \ar[r]_-{(i_Q\op)^\ast} &
\big(\cSp^{Q\op}\big)^{M_{n},\exx} \ar[r]_-{i_Q^\ast} &
\big(\cSp^{Q\op}\big)^Q,
}
}
\end{equation}
where $S$ denotes the variants of the Serre functor from \autoref{defn:Serre}. Tracing $U_n$ through this diagram, we obtain isomorphisms
\begin{align}
&\, (i_Q\times i_Q\op)^\ast(s\times \id)^\ast(U_n) \\
=&\, (i_Q\times \id)^\ast(s\times \id)^\ast (\id\times i_Q\op)^\ast(U_n)  & \\
\cong&\, (i_Q\times \id)^\ast S (\id\times i_Q\op)^\ast(U_n) & &\text{(by \autoref{lem:Serre-shift})} \\
\cong&\, S(i_Q\times i_Q\op)^\ast (U_n) & &\text{(by \autoref{lem:Serre-match}) }\\
\cong&\, S(\lI_Q) & &\text{(by \eqref{eq:Yoneda-justify}) }\\
\cong&\, D_Q\otimes_{[Q]}\lI_Q & &\text{(by \autoref{thm:Serre-Nakayama-indep})}\\
\cong&\, D_Q.&	
\end{align}
A combination of these isomorphisms with \eqref{eq:U_n-Serre-lhs} show that there is a natural isomorphism~\eqref{eq:U_n-Serre}, and the theorem hence follows.
\end{proof}

\section{Derivators of fields and Picard groupoids}
\label{sec:field}

Classically, representation theory deals with (derived) categories of representations of quivers over a field. In our case this amounts to studying the derivator $\D_k$ of a field $k$. Various concepts defined in the text above (Auslander--Reiten quivers and translations, Coxeter and Serre functors) have played a very important role in a rather detailed understanding of $\D_k$ and it is this connection which we aim to explain here. This allows us also to reinterpret many of our results by saying that inducing up from spectra to chain complexes over a field induces a split epimorphism from the spectral Picard groupoid spanned by Dynkin quivers of type~$A$ to the corresponding (derived) Picard groupoid over an arbitrary field.

Let $k$ be a field and let $\D_k$ be the associated stable, closed symmetric monoidal derivator as in \autoref{egs:monoidal}. Given a finite quiver $Q$ without oriented cycles, we have $\D^Q_k(\bbone) \cong \D_k(Q) \simeq D(kQ)$, where $kQ$ is the corresponding finite dimensional path algebra. Classically~\cite{happel:fd-algebra,happel:triangulated} one studies the bounded derived category $D^b(kQ)$ which is the full subcategory of $D(kQ)$ whose objects are bounded complexes of finitely generated (or equivalently finite dimensional) $kQ$-modules. Since $kQ$ is a hereditary algebra~\cite[Proposition III.1.4]{auslander-reiten-smalo:rep-th}, every bounded complex of finitely generated $kQ$-modules is quasi-isomorphic to a bounded complex of finitely generated projective $kQ$-modules. In particular, $D^b(kQ)$ can be identified (up to closure under isomorphic images) with the full subcategory of compact objects in $D(kQ)$; see for instance~\cite[\S6.5]{krause:chicago}.

Other well known facts to mention about $D^b(kQ)$ are that it is a $k$-linear and idempotent complete triangulated category and for each pair of objects $x,y \in D^b(kQ)$ the $k$-vector space $\hom(x,y)$ is finite dimensional. As a crucial consequence, one has the following result which is usually called the Krull--Schmidt decomposition theorem.

\begin{lem}[{\cite[\S2.2, p. 52]{ringel:tame-alg}}] \label{prop:krull-schmidt}
Let $x \in D^b(kQ)$. Then we have $x \cong \coprod_{i=1}^n x_i$ for some non-negative integer $n$ and indecomposable objects $x_i \in D^b(kQ)$. Moreover, if $x \cong \coprod_{j=1}^m y_j$ is another such decomposition of $x$ into a coproduct of indecomposable objects, then $m = n$ and there exist a permutation $\pi$ such that $x_i \cong y_{\pi(i)}$ for each $i = 1, \dots, n$.
\end{lem}

Thus in order to describe $D^b(kQ)$, it essentially suffices to understand $\ind D^b(kQ)$, a skeleton of the full subcategory of $D^b(kQ)$ formed by indecomposable objects. If $Q$ is a Dynkin quiver, this problem has been completely solved by Happel in~\cite{happel:fd-algebra}.
In order to describe the solution we first need to define a so-called $k$-linear mesh category $M_Q(k)$.

Let $k$ be a commutative ring and $\widehat{Q}$ be the repetitive quiver of $Q$ as defined at the beginning of \S\ref{sec:An}. Given a vertex $(l,q)$ of $\widehat{Q}$ (that is, $l \in \lZ$ and $q \in Q$) and an arrow $\gamma\colon (l,q) \to (l',q')$ induced by an arrow $\alpha$ of $Q$, then there is a unique arrow $\sigma(\gamma)\colon (l',q') \to (l+1,q)$ induced by the same $\alpha$. Indeed, either $l' = l$, $\alpha\colon q \to q'$ in $Q$, $\gamma = \alpha\colon (l,q) \to (l,q')$, and $\sigma(\gamma) = \alpha^\ast\colon (l,q') \to (l+1,q)$, or $l' = l+1$, $\alpha\colon q' \to q$ in $Q$, $\gamma = \alpha^\ast$ and $\sigma(\gamma) = \alpha$. The \textbf{mesh relation} at $(l,q) \in \widehat{Q}$ is defined to be the following finite linear combination of paths $(l,q) \to (l+1,q)$ of length two:
\begin{equation} \label{eq:mesh-rel}
\sum_{\gamma\colon (l,q) \to (l',q')} \sigma(\gamma) \circ \gamma.
\end{equation}
The $k$-linear \textbf{mesh category} is defined as the quotient of $k\langle\widehat{Q}\rangle$, the free $k$-linear category generated by $\widehat{Q}$, modulo the two-sided ideal generated by the mesh relations at all vertices $(l,q) \in \widehat{Q}$.

\begin{eg} \label{eg:mesh-An}
Let us consider the case when $Q$ is Dynkin of type $A_n$. Then $\widehat{Q}$ is the quiver discussed in \S\ref{sec:An} and for $n=3$ depicted in~\eqref{eq:ar-quiver}. The mesh relations say that all the squares anti-commute and the paths of length two starting at the vertices at the upper and lower boundaries of $\widehat{Q}$ vanish. It is easy to check that if we replace the anti-commutativity of squares by commutativity relations, the resulting category is isomorphic to $M_Q(k)$.
\end{eg}

\begin{prop}[{\cite[Proposition 4.6]{happel:fd-algebra}}] \label{prop:ar-classical}
Let $k$ be a field, $Q$ be a Dynkin quiver, and $\ind D^b(kQ)$ be a skeleton of the category of indecomposable objects in the bounded derived category. Then $\ind D^b(kQ)$ is isomorphic to the $k$-linear mesh category $M_Q(k)$.

Moreover, the additive closure of $M_k(Q)$ inherits a triangulated structure from $D^b(kQ)$ under this isomorphism and for each object $(l,q) \in M_Q(k)$, there is then a triangle in $M_Q(k)$ of the following shape
\[ (l,q) \longrightarrow \coprod_{\gamma\colon (l,q) \to (l',q')} (l',q') \longrightarrow (l+1,q) \longrightarrow \Sigma (l,q), \]
with the obvious arrows $\gamma$ and $\sigma(\gamma)$ as components of the first two maps.
Thus, all the mesh relations come from the existence of these triangles.
\end{prop}

This also motivates the construction of the coherent Auslander--Reiten quiver in \S\ref{sec:An}. However, as the domain of a stable derivator $\D$ is the category of ordinary categories rather than additive categories, we need to use a trick to model the a priori additive mesh relations.
This is taken care of by the defining exactness properties of $\D^{M_n,\exx}$.

In order to really reconstruct the situation of \autoref{prop:ar-classical} for a quiver $Q$ of type $A_n$, we need to choose $\D = \D_k^{Q\op}$. Then the coherent diagram $\ARQ \in \D_k^{M_n \times Q\op}(\bbone) \simeq \D^{M_n}(\bbone)$ from \S\ref{subsec:AR_Q} is a coherent version of $\ind D^b(kQ\op)$. More precisely, the underlying incoherent diagram $\mathrm{dia}_{M_n}(\ARQ) \in \D(\bbone)^{M_n} = D(kQ\op)^{M_n}$ is, after stripping off the zero objects, a skeleton of indecomposables in $D^b(kQ\op)$.

Next, given a $k$-linear exact functor $F\colon D^b(kQ) \to D^b(kQ')$ for $Q,Q'$ Dynkin, we may wonder what is the explicit description of $F$ on the (additive closures of) the mesh categories $M_Q(k)$ and $M_{Q'}(k)$. The easiest case is the suspension autoequivalence $\Sigma$, which is classically known to act in the Dynkin $A$-case exactly as described by \autoref{cor:susp}.

Another important autoequivalence is the Serre functor $S\colon D^b(kQ) \to D^b(kQ)$ (see \S\ref{subsec:fCY}). It is well known that $S = (kQ)^\ast \Lotimes_{kQ} -$ in our case, where $kQ^\ast$ is the vector space dual of $kQ$ with the induced $kQ$-$kQ$-bimodule structure; see for instance~\cite[Theorem 3.4]{krause-le:AR}. As the (non-derived) functor 
\[ \nu = (kQ)^\ast \otimes_{kQ} -\colon \Mod kQ \to \Mod kQ \]
is often called the \emph{Nakayama functor} (see for instance \cite[\S X.1]{beligiannis-reiten}), the Serre functor is the derived Nakayama functor. The action of $S$ on $M_Q(k)$ in the classical case is precisely the one described in \autoref{lem:Serre-shift}.

From a representation theoretic perspective, it may be more instructive to consider the Auslander--Reiten translation $\tau\colon D^b(kQ) \to D^b(kQ)$, which is the composition of the Serre functor with a desuspension (in either order, see~\cite[\S I]{reiten-bergh:serre} for details). This functor and the corresponding theory of almost split triangles~\cite{happel:fd-algebra,happel:triangulated} is the key ingredient in the proof of \autoref{prop:ar-classical}. The action of $\tau$ on $M_Q(k)$ is very simple --- it sends $(l,q) \in M_Q(k)$ to $(l-1,q)$; compare to \autoref{defn:ar-transl}.

As seen in \S\S\ref{subsec:fCY}-\ref{subsec:symm-mesh}, the suspension and Serre functors on $D^b(kQ)$ for an $A_n$-quiver $Q$ satisfy the relation $\Sigma^{n-1} \cong S^{n+1}$. Triangulated categories with Serre functors satisfying a non-trivial relation between the Serre and suspension functors are said to have the \emph{fractional Calabi--Yau property} and they attracted a lot of attention recently; see for instance~\cite{keller:orbit,keller:CY-triang} and references there.
The fractional Calabi--Yau relation and commutativity between $\Sigma$ and $S$ are in fact the only relations as long as $n \ge 2$: One can check using \autoref{prop:ar-classical} that an autoequivalence of the form $\Sigma^i S^j$ cannot fix all isoclasses of indecomposable objects in $D^b(kQ)$ unless $(i,j)$ is an integral multiple of $(n-1,n+1)$. Moreover, all standard triangle autoequivalences of $D^b(kQ)$ are generated by $\Sigma$ and $S$ (\emph{standard} means that the autoequivalence is given by the derived tensor product with a complex of $kQ$-$kQ$-bimodules). This is proved and motivated in~\cite[Theorem 4.1]{miyachi-yekutieli}.

Of course, one is not interested only in autoequivalences of $D^b(kQ)$, but also in equivalences between $D^b(kQ)$ and $D^b(kQ')$ for distinct $Q$ and $Q'$. Classically, the classification \cite{gabriel:unzerlegbare} of indecomposable $kQ$-modules for a Dynkin quiver $Q$ revealed that the number and to some extent also the structure of these modules was independent of the orientation of $Q$. The latter fact was elegantly explained by Bern{\v{s}}te{\u\i}n, Gel$'$fand, and Ponomarev in~\cite{bernstein-gelfand-ponomarev:Coxeter} using the so called reflection functors. Later in~\cite{happel:fd-algebra} Happel showed that these reflection functors induced triangle equivalences between $D^b(kQ)$ and $D^b(kQ')$ for different orientations $Q$ and $Q'$ of the same simply laced Dynkin graph. In our abstract context, this process is explained in~\cite[\S5]{gst:tree} and the easier situation in the Dynkin $A$ case was already treated in~\cite[\S6]{gst:basic}. The abstract counterparts of other standard results in this context involving so-called (partial) Coxeter functors, mostly taken from~\cite{bernstein-gelfand-ponomarev:Coxeter}, have been discussed here in \S\ref{subsec:Coxeter}.

The above mentioned equivalences between bounded derived categories of distinct quivers stood at the dawn of tilting theory; see for instance \cite{apr:tilting,brenner-butler:tilting,happel-ringel:tilted-algebras,rickard:derived-fun,keller:deriving-dg} and also the collection \cite{angeleri-happel-krause:handbook} of survey articles. The idea was to represent the triangle equivalences between derived categories by complexes of bimodules and to find conditions when such equivalences exist. Our \S\S\ref{sec:admissible}--\ref{sec:yoneda} have been devoted to generalizing some of these results to the context of abstract stable derivators.

Classically, suppose that $R$ is a ring and $T \in D(R)$ is a complex. Then $T$ is a \emph{tilting complex} if it is an image of $S$ under a triangle equivalence $D(S) \to D(R)$ for some other ring $S$. Of course the endomorphism ring of $T$ in $D(R)$ is then precisely $S\op$. It follows from this definition that $D(R)$ and $D(S)$ are triangle equivalent if and only if there exists a tilting complex of $R$-modules with endomorphism ring $S\op$. The following result is standard and it allows to check whether a given complex is tilting. 

\begin{lem} \label{lem:tilting-char}
Let $T \in D(R)$ be a complex. Then $T$ is tilting if and only if
\begin{enumerate}
\item $T$ is a perfect complex, i.e., isomorphic to a bounded complex of finitely generated projective modules,
\item $T$ is rigid, i.e., $\hom_{D(R)}(T,\Sigma^i T) = 0$ whenever $i \ne 0$, and
\item $T$ generates $D(R)$, i.e., for each $0\ne X \in D(R)$ there exists $i \in \lZ$ and a non-zero homomorphism $f\colon \Sigma^i T \to X$.
\end{enumerate}
\end{lem}

\begin{proof}
If there is a triangle equivalence $D(S) \to D(R)$ sending $S$ to $T$, then $T$ clearly satisfies conditions (i)--(iii) since $S$ is well known to have the same properties in $D(S)$. The converse was proved by Rickard in~\cite[Theorem 6.4]{rickard:morita} and a more conceptual explanation of the phenomenon was later given by Keller in~\cite[\S9.2]{keller:dg-categories}.
\end{proof}

\begin{rmk} \label{rmk:zigzag-QE}
In our context, it is interesting to note that the existence of a tilting complex $T \in D(R)$ with endomorphisms ring $S\op$ is also equivalent to the existence of a zigzag of Quillen equivalences between the categories of complexes over $R$ and $S$; see \cite[Theorem 4.2]{dugger-shipley:K-theory}. By \cite{renaudin} this is also equivalent to the fact that the derivators $\D_R$ and $\D_S$ are equivalent. 
\end{rmk}

As long as we are interested not only in the existence of a triangle equivalence, but how the equivalence actually acts (e.g. as~\cite{miyachi-yekutieli} did), the situation gets more subtle. First of all, if $R$ and $S$ are algebras over a field $k$, it would be desirable to replace quasi-isomorphically a tilting complex of $R$-modules with endomorphism ring $S\op$ by a complex of $R$-$S$-bimodules with the base field acting centrally. In other words, we would like to lift the right action of $S$ on $T$ from $D(R)$ to the category of complexes of $R$ modules. As discussed in~\cite{rickard:derived-fun,keller:remark}, this is possible for algebras over a field, or more generally for algebras flat over the base ring, and will not be an issue for us.

Second, if we have such a tilting complex $T$ of $R$-$S$-bimodules, then it makes sense to write the adjunction of the derived functors
\[ T \Lotimes_S -\colon D(S) \rightleftarrows D(R)\colon \Rhom_R(T,-) \]
and this is a pair of inverse $k$-linear triangle equivalences. As already mentioned, such equivalences are called standard. Up to natural isomorphism these equivalences are precisely classified by isotypes of tilting bimodules in $D(R \otimes_k S\op)$. Also this has been proved in~\cite{keller:remark} for algebras over a field.

Our tilting bimodules from \S\ref{sec:universal-tilting} are refinements of classical tilting complexes.
Given such a tilting bimodule $T \in \cSp(Q' \times Q\op)$, we can apply the
monoidal morphism $k\otimes -\colon\cSp\to\D_k$ (see \autoref{eg:EM-spectra}) to construct a chain complex
\[
k \otimes T \in \D_k(Q' \times Q\op) \simeq D(kQ' \otimes_k kQ\op).
\]
Since $k\otimes -$ is a left adjoint, monoidal morphism between stable derivators, $k \otimes T$ represents the corresponding composition of classical reflection functors.

\begin{thm}\label{thm:picard}
Let $Q$ be a Dynkin quiver of type~$A$ and let $k$ be a field. Inducing up $k\otimes-\colon\cSp\to\D_k$ defines a split epimorphism of groups
\[
k\otimes-\colon\Pic_{\cSp}(Q)\to\Pic_{\D_k}(Q).
\]
\end{thm}
\begin{proof}
By \autoref{eq:picard} inducing up from spectra to chain complexes over a field defines such a group homomorphism. The (derived) Picard groups $\Pic_{\D_k}(Q)$ over a field are calculated explicitly in \cite[Theorem~4.1]{miyachi-yekutieli}, shown to be independent of the field, and to admit the presentations
\[
\Pic_{\D_k}(Q)=\langle f,s \mid fs=sf, f^{n-1} = s^{n+1} \rangle.
\]
Here, $f,s\in D(kQ\otimes_k kQ\op)$ can be chosen as tilting complexes for the shift functor $\Sigma\colon D(kQ)\to D(kQ)$ and the derived Nakayama functor $\nu\colon D(kQ)\to D(kQ)$, respectively. By \autoref{thm:Serre-Nakayama-indep} the Serre functor $S\colon\cSp^Q\to\cSp^Q$ is represented by the canonical duality module $D_Q\in\cSp(Q\times Q\op)$, and it follows from the explicit construction of $D_Q$ that the group homomorphism $\Pic_{\cSp}(Q)\to\Pic_{\D_k}(Q)$ sends $[D_Q]$ to $s$. Similarly, this group homomorphism sends $[\Sigma\lI_Q]$ to $f$ (\autoref{cor:susp}), showing that the homomorphism is surjective. In order to obtain a section to this homomorphism,
it suffices to observe that the spectral bimodules $\Sigma\lI_Q,D_Q$ satisfy up to isomorphism the defining relations in the above presentation of the derived Picard group $\Pic_{\D_k}(Q)$. Since the Serre functor $S$ is an equivalence of stable derivators, it clearly is exact and by the uniqueness of spectral bimodules there is an isomorphism $(\Sigma\lI_Q)\otimes_{[Q]}D_Q\cong D_Q\otimes_{[Q]}(\Sigma\lI_Q)$. It remains to check that there is an isomorphism
\[
\underbrace{(\Sigma\lI_Q)\otimes_{[Q]}\ldots\otimes_{[Q]}(\Sigma\lI_Q)}_{n-1\;\mathrm{times}}\cong 
\underbrace{D_Q\otimes_{[Q]}\ldots\otimes_{[Q]}D_Q}_{n+1\;\mathrm{times}}.
 \] 
 But this is an immediate consequence of \autoref{thm:Serre-Nakayama-indep} and the abstract fractionally Calabi--Yau property given in \autoref{cor:frac-CY} (using once more implicitly the uniqueness of representing spectral bimodules).
\end{proof}

\begin{rmk}\label{rmk:picard}
\begin{enumerate}
\item In a similar way one obtains a split epimorphism from the \textbf{spectral Picard groupoid} spanned by Dynkin quivers of type $A$ to the corresponding \textbf{(derived) Picard groupoid} over a field. These respective Picard groupoids are simply many-objects versions of the Picard groups; more formally, we consider the respective groupoids obtained by considering isomorphism classes of invertible bimodules between all Dynkin quivers of type $A$ at once. In fact, given two different $A_n$-quivers $Q,Q'$, a combination of reflection functors yields the invertible spectral bimodule $T_{Q',Q}\in\cSp(Q'\times Q\op)$ from \autoref{defn:iter-tilt}. And these spectral bimodules can be used to show that there is a split epimorphism of Picard groupoids.

The question whether these morphisms of groups or groupoids are actually isomorphisms remains open. 
\item Let $\V$ be a stable, closed symmetric monoidal derivator, let $k$ be a field, and let $\V\to\D_k$ be a colimit preserving, monoidal morphism of derivators. For every Dynkin quiver~$Q$ of type~$A$ the induced morphism $\Pic_\V(Q)\to\Pic_{\D_k}(Q)$ is a split epimorphism. Similarly, also the corresponding result for Picard groupoids is valid in these more general situations.
\end{enumerate}
\end{rmk}

It remains to interpret the Yoneda bimodules $U_n$ from \S\ref{sec:yoneda} in the classical situation.
Given a quiver $Q$, the category $D(kQ)$ is enriched over $D(k)$ --- we have the derived Hom functor
\begin{equation} \label{eq:Rhom-Dynkin}
\Rhom_{kQ}\colon D(kQ)\op \times D(kQ) \to D(k).
\end{equation}
In our context this is simply a consequence of the fact that $\D_{kQ} \simeq \D_k^Q$ admits a canonical closed $\D_k$-module structure. If $Q$ is Dynkin of type $A_n$, we know from \autoref{prop:ar-classical} that $\ind D^b(kQ) \cong M_Q(k)$. Under this identification, we can restrict \eqref{eq:Rhom-Dynkin} to
\begin{equation} \label{eq:Rhom-Dynkin-restr-add}
\Rhom_{kQ}\colon M_Q(k)\op \times M_Q(k) \to D(k).
\end{equation}
Since the (non-additive) category $M_{n}$ contains the generators for $M_Q(k)$, we have an obvious functor $g\colon M_{n} \to M_k(Q) \cup \{0\} \subseteq \D^b(kQ)$, which acts as follows:
\begin{enumerate}
\item The boundary objects of $M_{n}$ which have to be populated by zeros in $\D^{M_n,\exx}$ are sent to zero. Similarly, the arrows of $M_{n}$ incident with the boundary objects are sent to zero morphisms.
\item The interior objects and arrows of $M_{n}$ are sent to the corresponding generating morphisms of $M_Q(k)$.
\end{enumerate}
Then we can consider the composition
\begin{equation} \label{eq:Rhom-Dynkin-restr-non-add}
\Rhom_{kQ} \circ (g\op \times g) \colon M_{n}\op \times M_{n} \to D(k),
\end{equation}
as an object of $D(k)^{M_{n} \times M_{n}\op} \simeq \D_k(\bbone)^{M_{n} \times M_{n}\op}$
and $k \otimes U_{n} \in \D_k(M_{n} \times M_{n}\op)$ is simply a coherent version of this diagram. Hence $U_{n}$ can be viewed as a spectral and coherent version of a certain restriction of $\Rhom_{kQ}$. Note in this context that \autoref{thm:U_n-Serre} refines the classical Serre duality as stated in \S\ref{subsec:fCY}.

\section{Higher triangulations via abstract representation theory}
\label{sec:higher}

As an attempt to fix certain defects of triangulated categories, the idea of considering higher triangles goes back at least to \cite[Remark~1.1.14]{beilinson:perverse}. More recently, Maltsiniotis \cite{maltsiniotis:higher} complemented this by a precise definition of strong triangulations and sketched a strategy how to construct such triangulations in stable derivators. Since to the best of the knowledge of the authors of this paper a proof has not been worked out since, we recycle some of the results in \S\S\ref{sec:An}-\ref{sec:coxeter-reflection} to obtain such a proof here. It turns out that these higher triangulations are conveniently organized by means of abstract representation theory of linearly oriented Dynkin quivers of type~$A$. 

We begin by recalling the relevant definitions from \cite{maltsiniotis:higher}, starting with the $n$-triangles themselves. The established numbering convention is that $2$-triangles correspond to the usual triangles in the sense of Puppe and Verdier, hence the convention emphasizes the number of objects. Thus, $n$-triangles are closely related to representations of linear $A_n$-quivers $\A{n}$. By \autoref{thm:AR} we know that for every stable derivator~\D there is a natural equivalence $\D^\A{n}\simeq\D^{M_n,\exx}$ and also that the precomposition $f^\ast\colon\D^{M_n,\exx}\to\D^{M_n,\exx}$ with the flip symmetry~$f$ defined in \eqref{eq:f} is naturally isomorphic to the suspension (\autoref{cor:susp}).

Recall that a pair $(\cA,S)$ consisting of an additive category $\cA$ and an equivalence $S\colon\cA\to\cA$ is also called a \textbf{suspended category} and is often simply denoted by~$\cA$.

\begin{defn}
Let $\cA$ be a suspended category and let $n\geq 1$. An \textbf{$n$-triangle} in $\cA$ is a pair $(F,\phi)$ consisting of
\begin{enumerate}
\item a functor $F\colon M_n\to\cA$ such that $F_{(k,0)}\cong F_{(k,n+1)}\cong 0, k\in\lZ$, and
\item a natural isomorphism $\phi\colon F\circ f\to S\circ F$.
\end{enumerate}
A \textbf{morphism of $n$-triangles} $(F,\phi)\to(F',\phi')$ is a  natural transformation $F\to F'$ making the following square commute
\[
\xymatrix{
F\circ f\ar[r]^-\phi\ar[d]&S\circ F\ar[d]\\
F'\circ f\ar[r]_-{\phi'}&S\circ F'.
}
\]
\end{defn}

The idea of course is that an $n$-triangle encodes the incoherent shadow of $n-1$ composable morphisms together with all iterated (co)fibers. So, a $2$-triangle essentially looks like a classical `two-sided' Barratt--Puppe sequence (see \autoref{rmk:AR}). Substituting $\Sigma$ by $S$, a $3$-triangle has essentially the form of \eqref{eq:A3}.

Higher triangulations come with classes of distinguished $n$-triangles for varying~$n$ which are suitably compatible with restriction functors and the symmetries of the mesh categories. To make this precise we note that the assignment $\A{n}\mapsto M_n$ extends to a functor. We leave it to the reader to figure out the formulas, but mention that they can be obtained by a combination of explicit formulas in \cite{maltsiniotis:higher} and obvious changes of coordinates between his conventions and ours in \S\ref{sec:An}. It follows from these explicit formulas that for every $\alpha\colon\A{m}\to\A{n}$ the induced functor $\alpha_\ast\colon M_m \to M_n$ sends boundary points to boundary points. Moreover, the inclusions $i=i_n\colon\A{n}\to M_n$ defined in \eqref{eq:i} assemble to a natural transformation, i.e., for every $\alpha\colon \A{m}\to\A{n}$ the square
\begin{equation}\label{eq:higher}
\vcenter{\xymatrix{
\A{m}\ar[d]_-\alpha\ar[r]^-{i_m}&M_m\ar[d]^-{\alpha_\ast}\\
\A{n}\ar[r]_-{i_n}&M_n
}
}
\end{equation}
commutes. 

Given an $n$-triangle $(F,\phi)$ in $\cA$ we refer to $i_n^\ast(F,\phi)=F\circ i_n\colon\A{n}\to M_n\to\cA$ as the \textbf{base} of the $n$-triangle. Similarly, given $\alpha\colon\A{m}\to\A{n}$, the \textbf{inverse image} $\alpha^\ast(F,\phi)$ of $(F,\phi)$ is the pair 
\[
\alpha^\ast(F,\phi)=(F\circ \alpha_\ast,\phi\circ\alpha_\ast)
\]
which is easily seen to be an $m$-triangle. Recall the definition of the Auslander--Reiten translation $\tau\colon\D^{M_n,\exx}\to\D^{M_n,\exx}$ in \autoref{defn:ar-transl} as a restriction of the precomposition functor along $t\colon M_n\to M_n$ as defined in \eqref{eq:t}. The \textbf{translate} $t^\ast(F,\phi)$ of $(F,\phi)$ is the $n$-triangle given by
\[
t^\ast(F,\phi)=(F\circ t,\phi\circ t).
\]
Finally, the \textbf{flipped $n$-triangle} associated to $(F,\phi)$ is the $n$-triangle $f^\ast(F,\phi)$ obtained by restriction along the flip symmetry $f\colon M_n\to M_n$ as in \eqref{eq:f} and introducing a sign in the identification,  
\[
f^\ast(F,\phi)=(F\circ f,-\phi\circ f).
\]

\begin{defn}
Let $\cA$ be a suspended category. A \textbf{strong triangulation} or \textbf{$\infty$-triangulation} on $\cA$ consists of classes of $n$-triangles for $n\geq 2$, the so-called \textbf{distinguished $n$-triangles}, such that the following properties are satisfied.
  \begin{itemize}[leftmargin=4em]
  \item[(STC0)] The classes of distinguished $n$-triangles, $n\geq 2$, are closed under isomorphisms.
  \item[(STC1)] Every $X\colon \A{n}\to\cA,n\geq 2,$ is the base of a distinguished $n$-triangle.
  \item[(STC2)] Let $(F,\phi),(F',\phi')$ be distinguished $n$-triangles, $n\geq 2$. Any morphism $i_n^\ast(F,\phi)\to i_n^\ast(F',\phi')$ of the bases extends to a morphism of the distinguished $n$-triangles.
  \item[(STC3)]  For each $n\geq 2$ the following three closure properties are satisfied.
  \begin{itemize}
  \item[$\bullet$] Inverse images of distinguished $n$-triangles are again distinguished.
  \item[$\bullet$] The translate of a distinguished $n$-triangles is again distinguished.
  \item[$\bullet$] Flipped $n$-triangles associated to distinguished $n$-triangles are again distinguished.
  \end{itemize} 
  \end{itemize}

A \textbf{strongly triangulated} or \textbf{$\infty$-triangulated category} is a suspended category together with a strong triangulation.
\end{defn}

\begin{rmk}
Maltsiniotis \cite{maltsiniotis:higher} also introduces truncations of these notions, namely \emph{$N$-triangulations} and \emph{$N$-pretriangulations} for fixed $N\geq 2$. He also sketches a proof that $2$-triangulations induce ordinary triangulations and remarks that these notions are actually equivalent.
\end{rmk}

In order to construct strong triangulations in a stable derivator~\D, we have to be able to lift incoherent representations $\A{n}\to\D(A)$ to coherent diagrams $X\in\D(A\times \A{n})$, and similarly for morphisms of incoherent representations. This is taken care of by the following remark.

\begin{rmk}\label{rmk:strong}
Recall that a derivator is \textbf{strong} if it satisfies:
  \begin{itemize}[leftmargin=4em]
  \item[(Der5)] For any $A$, the partial underlying diagram functor $\D(A\times [1]) \to \D(A)^{[1]}$ is full and essentially surjective.
  \end{itemize}
In the context of strong triangulations we need a stronger version of this property. There are at least two ways of adressing this.
\begin{enumerate}
\item The above axiom can be replaced by the stronger version as used by Heller \cite{heller:htpythies}. His version asks that the partial underlying diagram functors are full and essentially surjective also if we replace the category $[1]$ by a finite, free category~$F$. Represented derivators and homotopy derivators of model categories and $\infty$-categories also satisfy this stronger version of the axiom, so that it is of no harm to assume this version.
\item Alternatively, as shown by Keller--Nicolas \cite[Prop.~A.5]{keller-nicolas} for \emph{stable} derivators it is a consequence of (Der5) in the above weaker form that the partial underlying diagram functor $\D(A\times\lN)\to\D(A)^\lN$ is also full and essentially surjective. Here, $\lN$ is the poset of natural numbers with the usual ordering considered as a category. Their proof also shows that a similar result is true if we replace the poset $\lN$ by the quivers $\A{n}$.
\end{enumerate}
\end{rmk}

Let \D be a strong, stable derivator, let $A\in\cCat$, and let $n\geq 2$. We know from \autoref{thm:triang} that every $\D(A)$ is canonically a triangulated category and hence, in particular, a suspended category $(\D(A),\Sigma)$. Every $X\in\D^{M_n,\exx}(A)$ gives rise to an $n$-triangle in $(F_X,\phi_X)$ in $\D(A)$ as follows. The functor $F_X$ is the partial underlying diagram $\mathrm{dia}_{M_n}(X)$ of $X\in\D^{M_n,\exx}(A)\subseteq\D(M_n\times A)$. The defining exactness properties of $\D^{M_n,\exx}$ imply that $F_X\colon M_n\to\D(A)$ satisfies the desired vanishing conditions. Moreover, \autoref{prop:susp} together with the exactness of the inclusion $\D^{M_n,\exx}\subseteq\D^{M_n}$ yields a canonical isomorphism $\phi_X\colon F_X\circ f\to\Sigma\circ F_X$. We refer to the $n$-triangle $(F_X,\phi_X)$ as the \textbf{standard $n$-triangle} associated to $X\in\D^{M_n,\exx}(A)$. An $n$-triangle in $\D(A)$ is \textbf{distinguished} if it is isomorphic to a standard $n$-triangle.

We now show that these classes define strong triangulations.

\begin{thm}\label{thm:higher}
Let \D be a strong, stable derivator and let $A\in\cCat$. The above classes of distinguished $n$-triangles, $n\geq 2$, define a strong triangulation on $\D(A)$.
\end{thm}
\begin{proof}
Endowed with the usual suspension functor the additive category $\D(A)$ becomes a suspended category. By definition the classes of distinguished $n$-triangles are closed under isomorphism which settles axiom (STC0). 

To verify (STC1) we make the following observation. If $y\cong y'\colon \A{n}\to\D(A)$ for $n\geq 2$ are given, then $y$ is the base of a distinguished $n$-triangle if and only if $y'$ is. This follows easily from (STC0). Given such a diagram $y\colon \A{n}\to\D(A)$, the strongness implies that we can find an object $Y\in\D(\A{n}\times A)$ and a natural isomorphism $\mathrm{dia}_{\A{n}}(Y)\cong y$ (see \autoref{rmk:strong}). The equivalence $i_n^\ast\colon\D^{M_n,\exx}\to\D^{\A{n}}$ yields the existence of $X\in\D^{M_n,\exx}(A)$ together with an isomorphism $Y\cong i_n^\ast(X)$ of coherent diagrams. Since the restriction $i_n^\ast\mathrm{dia}_{M_n}(X)\colon\A{n}\to\D(A)$ is the base of the standard $n$-triangle of~$X$, the natural isomorphism $y\cong \mathrm{dia}_{\A{n}}(Y)\cong\mathrm{dia}_{\A{n}}(i_n^\ast X)=i_n^\ast\mathrm{dia}_{M_n}(X)$ shows that also $y$ is the base of a distinguished $n$-triangle. Axiom (STC2) is very similar. In fact, using axiom (STC0) one easily reduces the problem to the case of standard $n$-triangles and then it suffices to apply the strongness of \D (see again \autoref{rmk:strong}). 

It remains to verify axiom (STC3). We leave it to the reader to show in each of the three parts that it suffices to show that the respective operation sends standard $n$-triangles to distinguished $n$-triangles. The first statement follows from the description of the functors $\alpha_\ast\colon M_m\to M_n$ in \eqref{eq:higher}. Indeed, the functor $\alpha_\ast$ sends boundary points to boundary points so that the induced precomposition functor $\alpha^\ast\colon\D^{M_n}\to\D^{M_m}$ already preserves the defining vanishing conditions of the subderivators under consideration. In the interior of the mesh categories, $\alpha^\ast$ amounts to pasting squares or inserting constant squares. But since pastings of bicartesian squares are again bicartesian  \cite[Proposition~4.6]{groth:ptstab} and since constant squares are bicartesian \cite[Propositions~4.5]{groth:ptstab}, $\alpha^\ast$ also respects the remaining exactness properties and hence induces a morphism $\alpha^\ast\colon\D^{M_n,\ex}\to\D^{M_m,\ex}$.
The equivalences $(F_n,i_n^\ast)$ of \autoref{thm:AR} are compatible with these restriction functors in the sense that the left square in \autoref{fig:inv-img} commutes on the nose (this is immediate from the commutativity of \eqref{eq:higher}) while the one on the right commutes up to coherent natural isomorphism.
It follows easily that the inverse image of a standard $n$-triangle under $\alpha^\ast$ is a distinguished $m$-triangle. The fact that translates of standard $n$-triangles are distinguished $n$-triangles is immediate from the fact that the exactness properties of $\D^{M_n,\ex}$ are invariant under the translation.

\begin{figure}
\centering
\[
\xymatrix{
\D^{\A{m}}&\D^{M_m,\exx}\ar[l]^-\simeq_-{i_m^\ast}&&
\D^{\A{m}}\ar[r]_-\simeq^-{F_m}&\D^{M_m,\exx}\\
\D^{\A{n}}\ar[u]^-{\alpha^\ast}&\D^{M_n,\exx}\ar[l]_-\simeq^-{i_n^\ast}\ar[u]_-{\alpha^\ast},&&
\D^{\A{n}}\ar[u]^-{\alpha^\ast}\ar[r]^-\simeq_-{F_n}&\D^{M_n,\exx}\ar[u]_-{\alpha^\ast}
}
\]
\caption{Towards inverse images of distinguished $n$-triangles.}
\label{fig:inv-img}
\end{figure}

Finally, and this is the trickiest part, we take care of the passage to flipped $n$-triangles. Recall from \autoref{cor:susp} that for every $n\in\lN$ there is a natural isomorphism $\Phi\colon \Sigma\cong f^\ast\colon\D^{M_n,\exx}\to\D^{M_n,\exx}$ where $f$ is the flip symmetry defined by \eqref{eq:f}. Let us recall more precisely how this identification is obtained. For this purpose we consider $X\in\D^{M_n,\exx}$ and $(k,l)\in M_n$, i.e., $k\in\lZ, l\in\A{n}$. As follows from the details of the proofs of \autoref{prop:susp} and \autoref{cor:susp}, the natural isomorphism $\Phi(X)_{(k,l)}$ is obtained as follows. We restrict $X$ along the embedding $q\colon\square\to M_n$ as given by
\begin{equation}\label{eq:sign}
\vcenter{
\xymatrix{
(k,l)\ar[r]\ar[d]&(k,n+1)\ar[d]\\
(k+l,0)\ar[r]&f(k,l)=(k+l,n+1-l)
}
}
\end{equation}
and observe that $q^\ast (X)$ is cocartesian. Moreover, since $q^\ast (X)$ vanishes at $(1,0)$ and~$(0,1)$, the definition of $\Sigma$ yields the desired canonical isomorphism 
\[
\Phi(X)_{(k,l)}\colon\Sigma (X_{(k,l)})\cong X_{f(k,l)}.
\]

The sign which shows up in the definition of flipped triangles is needed because the two isomorphisms $\Phi(f^\ast(X))_{(k,l)}$ and $\Phi(X)_{f(k,l)}$ differ by such a sign. In fact, the former of these isomorphisms is obtained by restricting $X$ along $f\circ q\colon \square\to M_n$ and then arguing as above. But this means that we have to restrict $X$ to the left square in \autoref{fig:signs}, 
\begin{figure}[h]
\centering
\[
\xymatrix{
(k+l,n+1-l)\ar[r]\ar[d]&(k+n+1,0)\ar[d]& (k+l,n+1-l)\ar[r]\ar[d]&(k+l,n+1)\ar[d]\\
(k+l,n+1)\ar[r]&(k+n+1,l),& (k+n+1,0)\ar[r]&(k+n+1,l).
}
\]
\caption{Towards signs in flipped $n$-triangles.}
\label{fig:signs}
\end{figure}
the coordinates of which are obtained by applying $f$ to each of the coordinates in \eqref{eq:sign}. If we want to construct the isomorphism $\Phi(X)_{f(k,l)}$ instead, we have to restrict $X$ along a square similar to \eqref{eq:sign} but starting with $f(k,l)$ as an upper left corner. Thus, in that case this amounts to restricting $X$ along the square on the right in \autoref{fig:signs}. Now, these two restrictions differ up to an application of the flip symmetry of $\square$ which interchanges $(1,0)$ and $(0,1)$. As discussed in \cite[\S4.1]{groth:ptstab} this symmetry precisely gives rise to a sign as this flip induces the coinversion map in natural cogroup structures on suspension objects in pointed derivators. From this observation it is rather immediate that distinguished $n$-triangles are stable under the passage to flipped $n$-triangles, concluding the proof.
\end{proof}

\begin{egs}
Examples of strong, stable derivators are homotopy derivators of stable model categories. Similarly, homotopy derivators of stable, complete, and cocomplete $\infty$-categories are strong and stable. Consequently, \autoref{thm:higher} yields canonical strong triangulations in the typical examples arising in algebra, geometry, and topology. A fairly long list of explicit examples of strong, stable derivators can be found in \cite[\S 5]{gst:basic}.
\end{egs}

We refer to the strong triangulations of the theorem as \textbf{canonical strong triangulations}. These strong triangulations turn out to be compatible with exact morphisms of strong, stable derivators as we discuss next.

Let $(\cA,S),(\cA',S')$ be suspended categories, let $G\colon\cA\to\cA'$ be an additive functor and let $\psi\colon GS\to S'G$ be a natural isomorphism. Given an $n$-triangle $(F,\phi)$ in $\cA$ for some $n\geq 2$, the \textbf{image} $G(F,\phi)$ is the $n$-triangle $(GF,\phi')$ where $\phi'$ is defined as
\[
\phi'\colon GFf \stackrel{\phi}{\to} GSF \stackrel{\psi}{\to} S'GF.
\]

\begin{defn}
Let $(\cA,S),(\cA',S')$ be strongly triangulated categories and let $G\colon\cA\to\cA'$ be an additive functor. An \textbf{exact structure} on $G$ is a natural isomorphism $\psi\colon GS\to S'G$ such that the corresponding images of distinguished $n$-triangles, $n\geq 2$, are again distinguished. The pair $(G,\psi)$ is an \textbf{exact morphism} of strongly triangulated categories.
\end{defn}

As emphasized by the definition, exactness of a morphism of strongly triangulated categories is not a property but amounts to the specification of additional \emph{structure}. This is in contrast to the context of stable derivators where exactness is a \emph{property} a morphism may possess or not.

\begin{prop}\label{prop:exact}
Let $F\colon\D\to\E$ be an exact morphism of strong, stable derivators. The components $F_A\colon\D(A)\to\E(A)$ can be canonically turned into exact morphisms with respect to the canonical strong triangulations.
\end{prop}
\begin{proof}
By \autoref{thm:triangcan} the components $F_A\colon\D(A)\to\E(A)$ can be turned into exact functors with respect to the canonical ordinary triangulations. Since the equivalences of \autoref{thm:AR} are natural with respect to exact morphisms of stable derivators, one easily checks that these exact structures also yield exact morphisms with respect to canonical strong triangulations.
\end{proof}

Since left and right adjoint morphisms of stable derivators are exact there are plenty of examples of exact morphisms, including the adjunctions induced from Quillen adjunctions. Other examples are adjunctions given by restriction and Kan extension morphisms.

\begin{cor}
Let \D be a strong, stable derivator and let $u\colon A\to B$ be a functor between small categories. The functors
\[
u^\ast\colon\D(B)\to\D(A),\quad u_!\colon\D(A)\to\D(B),\quad\text{and}\quad u_\ast\colon\D(A)\to\D(B)
\]
can be canonically turned into exact functors with respect to canonical strong triangulations.
\end{cor}
\begin{proof}
It is enough to consider the adjunctions $(u_!,u^\ast),(u^\ast,u_\ast)$ of derivators.
\end{proof}

\begin{rmk}
Given exact functors $G,G'\colon\cA\to\cA'$ between strongly triangulated categories, a natural transformation $G\to G'$ is \textbf{exact} if the following diagram
\[
\xymatrix{
GS\ar[r]\ar[d]&S'G\ar[d]\\
G'S\ar[r]&S'G'
}
\]
is commutative. Let \D be a strong stable derivator, let $u,v\colon A\to B$ be functors, and let $\alpha\colon u\to v$ be a natural transformation. Then one checks that $\alpha^\ast\colon u^\ast\to v^\ast$ is exact in this sense. Thus, \autoref{thm:higher}, \autoref{prop:exact}, and this observation together imply that every strong, stable derivator $\D\colon\cCat\op\to\cCAT$ admits a lift against the forgetful functor from the $2$-category of strongly triangulated categories, exact functors, and exact natural transformations.
\end{rmk}

The key tool in the construction of the distinguished $n$-triangles of the canonical strong triangulations of \autoref{thm:higher} are the equivalences $F_{\A{n}}\colon\D^{\A{n}}\to \D^{M_n,\exx}$ of \autoref{thm:AR}. This equivalence is admissible and is represented by the spectral bimodule $T_n=\mathrm{AR}_{\A{n}} \in \cSp(M_n \times \A{n}\op)$ (see \autoref{defn:AR-con}). This motivates us to also refer to $T_n$ as the \textbf{universal constructor for $n$-triangles}: the associated canceling tensor product
\[
T_n\otimes_{[\A{n}]}-\colon\D^{\A{n}}\to\D^{M_n,\exx}
\]
sends an abstract representation $X\in\D^{\A{n}}$ to the corresponding `coherent distinguished $n$-triangle'. Similarly, for arbitrary $A_n$-quivers $Q$ the morphism 
\[
\mathrm{AR}_Q\otimes_{[Q]}-\colon\D^Q\to\D^{M_n,\exx}
\]
sends abstract representations of $Q$ to coherent distinguished $n$-triangles.

\section{Epilogue: A double cone of axiomatizations}
\label{sec:cone}

We conclude this paper by a short `philosophical remark' which tries to clarify the role derivators play among the many other axiomatizations of Abstract Homotopy Theory. By now, there is a plethora of various axiomatic approaches to Homological Algebra or Abstract Homotopy Theory which can be roughly divided into two classes. And depending on the specific problem at hand, it might be advantageous to choose a class of approaches accordingly.
\begin{enumerate}
\item One class of approaches consists of those where one is given the `entire model', and hence all the information about the homotopy theory under consideration  is available. These approaches are often studied using methods from Higher Category Theory. 
\item The other class of approaches focuses on the homotopy categories or derived categories themselves. The rather poor categorical properties of these categories are to some extent compensated by imposing additional non-canonical structures, which are often useful in calculations.
\end{enumerate}
One way to describe how \emph{derivators} fit into this picture is by saying that they are --- using a somewhat informal language --- \emph{weakly terminal} among the approaches which belong to Higher Category Theory or Homotopical Algebra and at the same time \emph{weakly initial} among the more traditional approaches based on homotopy categories endowed with additional structure.

As a partial justification that derivators can be thought of as a weakly terminal approach to Higher Category Theory, let us recall that for many of the typical approaches to Higher Category Theory the passage to underlying homotopy categories factors through derivators --- which are possibly defined on certain sub-2-categories of $\cCat$ only.
\begin{enumerate}
\item Derived categories of nice \emph{abelian categories} are underlying categories of derivators in the background. More generally, a similar result is true for nice \emph{exact categories} in the sense of Quillen~\cite{quillen:k-theory} as was shown by Keller~\cite{keller:exact}. 
\item Homotopy categories of \emph{Quillen model categories} \cite{quillen:ha} are underlying categories of associated homotopy derivators (see \cite{cisinski:direct} for the general case and \cite{groth:ptstab} for a simple proof for combinatorial model categories).
\item As a common generalization of the first two examples, Cisinski \cite{cisinski:derivable} introduced the notion of a \emph{derivable category} and showed that there are associated derivators. This includes nice Waldhausen categories \cite{waldhausen:k-theory} and hence applies to contexts of Algebraic $K$-Theory.
\item Homotopy categories of complete and cocomplete \emph{$\infty$-categories} or \emph{quasi-categories} are underlying categories of homotopy derivators. For a sketch proof of this see \cite{gps:mayer}. And it is expected that also further approaches to the theory of $(\infty,1)$-categories give rise to homotopy derivators.
\end{enumerate}
Some information seems to be lost by passing to homotopy derivators, but at present it is hard to tell how large the resulting gap is. For example, it turns out that the Quillen equivalence type of a combinatorial model category is non-canonically encoded by the associated homotopy derivator \cite{renaudin}.

We next illustrate what we mean by saying that derivators are weakly initial among the more classical approaches based on homotopy categories endowed with additional non-canonical structures.
In all these cases we implicitly assume that the derivators are strong (which is true for derivators showing up in nature and) which allows us to approximate certain incoherent diagrams by coherent ones. 

\begin{enumerate}
\item The most prominent of the classical axiomatizations encode structures induced by stable homotopy theories.
\begin{enumerate}
\item Stable derivators give canonically rise to \emph{triangulated categories} in the sense of Grothendieck--Verdier (see the reprint \cite{verdier:derived} of the thesis of Verdier) and Dold--Puppe \cite{dold-puppe}; see \cite{franke:adams,maltsiniotis:k-theory,groth:ptstab}.
\item Similarly, as shown in \S\ref{sec:higher}, stable derivators induce canonical \emph{strong triangulations in the sense of Maltsiniotis} \cite{maltsiniotis:higher}.
\end{enumerate}
\item There are various notions of additive categories which are endowed with suitably \emph{compatible triangulations and monoidal structures}; see \cite{margolis:spectra,HPS:axiomatic,may:additivity,keller-neeman:D4}. As shown in \cite{gps:additivity} and reconsidered in \cite{gst:basic} these additional structures can be constructed in stable monoidal derivators and their compatibility axioms follow from lemmas about the calculus of tensor products of functors.
\item \emph{Pretriangulated categories in the sense of Beligiannis--Reiten} \cite{beligiannis-reiten} are a weakening of triangulated categories, which are obtained essentially by asking that the suspension $\Sigma$ is a left adjoint only as opposed to an equivalence. Consequently, one asks for two types of triangles, namely left triangles and right triangles. If these classes match appropriately, then one obtains a pretriangulation, and it was shown in \cite{groth:ptstab} that additive derivators give rise to canonical pretriangulations.
\end{enumerate}

Thus, the above more classical \emph{structures} (and conjecturally this also applies to further such notions) are formal consequences of \emph{properties} of (monoidal) derivators in the background. And once one knows this for derivators such results are also true for all the other approaches to Higher Category Theory constituting the upper part of the double cone.

\appendix

\bibliographystyle{alpha}
\bibliography{tilting3}

\end{document}

%% file: decls.tex
\usepackage{amssymb,amsmath,stmaryrd,mathrsfs}
\def\definetac{\newif\iftac}    
\ifx\tactrue\undefined
  \definetac
  \ifx\state\undefined\tacfalse\else\tactrue\fi
\fi
\iftac\else\usepackage{amsthm}\fi
\usepackage[all,2cell]{xy}
\UseAllTwocells
\usepackage{enumitem}
\usepackage{xcolor}
\definecolor{darkgreen}{rgb}{0,0.45,0} 
\usepackage[pagebackref,colorlinks,citecolor=darkgreen,linkcolor=darkgreen]{hyperref}
\usepackage{mathtools}          
\usepackage{tikz}
\usepackage{braket}             

\usepackage{url}                
\usepackage{xspace}             

\makeatletter
\let\ea\expandafter

\def\mdef#1#2{\ea\ea\ea\gdef\ea\ea\noexpand#1\ea{\ea\ensuremath\ea{#2}\xspace}}
\def\alwaysmath#1{\ea\ea\ea\global\ea\ea\ea\let\ea\ea\csname your@#1\endcsname\csname #1\endcsname
  \ea\def\csname #1\endcsname{\ensuremath{\csname your@#1\endcsname}\xspace}}

\DeclareRobustCommand\widecheck[1]{{\mathpalette\@widecheck{#1}}}
\def\@widecheck#1#2{%
    \setbox\z@\hbox{\m@th$#1#2$}%
    \setbox\tw@\hbox{\m@th$#1%
       \widehat{%
          \vrule\@width\z@\@height\ht\z@
          \vrule\@height\z@\@width\wd\z@}$}%
    \dp\tw@-\ht\z@
    \@tempdima\ht\z@ \advance\@tempdima2\ht\tw@ \divide\@tempdima\thr@@
    \setbox\tw@\hbox{%
       \raise\@tempdima\hbox{\scalebox{1}[-1]{\lower\@tempdima\box
\tw@}}}%
    {\ooalign{\box\tw@ \cr \box\z@}}}

\newcount\foreachcount

\def\foreachletter#1#2#3{\foreachcount=#1
  \ea\loop\ea\ea\ea#3\@alph\foreachcount
  \advance\foreachcount by 1
  \ifnum\foreachcount<#2\repeat}

\def\foreachLetter#1#2#3{\foreachcount=#1
  \ea\loop\ea\ea\ea#3\@Alph\foreachcount
  \advance\foreachcount by 1
  \ifnum\foreachcount<#2\repeat}

\def\definescr#1{\ea\gdef\csname s#1\endcsname{\ensuremath{\mathscr{#1}}\xspace}}
\foreachLetter{1}{27}{\definescr}
\def\definecal#1{\ea\gdef\csname c#1\endcsname{\ensuremath{\mathcal{#1}}\xspace}}
\foreachLetter{1}{27}{\definecal}
\def\definebold#1{\ea\gdef\csname b#1\endcsname{\ensuremath{\mathbf{#1}}\xspace}}
\foreachLetter{1}{27}{\definebold}
\def\definebb#1{\ea\gdef\csname l#1\endcsname{\ensuremath{\mathbb{#1}}\xspace}}
\foreachLetter{1}{27}{\definebb}
\def\definefrak#1{\ea\gdef\csname f#1\endcsname{\ensuremath{\mathfrak{#1}}\xspace}}
\foreachletter{1}{9}{\definefrak} 
\foreachletter{10}{27}{\definefrak}
\def\definebar#1{\ea\gdef\csname #1bar\endcsname{\ensuremath{\overline{#1}}\xspace}}
\foreachLetter{1}{27}{\definebar}
\foreachletter{1}{8}{\definebar} 
\foreachletter{9}{15}{\definebar} 
\foreachletter{16}{27}{\definebar}
\def\definetil#1{\ea\gdef\csname #1til\endcsname{\ensuremath{\widetilde{#1}}\xspace}}
\foreachLetter{1}{27}{\definetil}
\foreachletter{1}{27}{\definetil}
\def\definehat#1{\ea\gdef\csname #1hat\endcsname{\ensuremath{\widehat{#1}}\xspace}}
\foreachLetter{1}{27}{\definehat}
\foreachletter{1}{27}{\definehat}
\def\definechk#1{\ea\gdef\csname #1chk\endcsname{\ensuremath{\widecheck{#1}}\xspace}}
\foreachLetter{1}{27}{\definechk}
\foreachletter{1}{27}{\definechk}
\def\defineul#1{\ea\gdef\csname u#1\endcsname{\ensuremath{\underline{#1}}\xspace}}
\foreachLetter{1}{27}{\defineul}
\foreachletter{1}{27}{\defineul}

\def\autofmt@n#1\autofmt@end{\mathrm{#1}}
\def\autofmt@b#1\autofmt@end{\mathbf{#1}}
\def\autofmt@l#1#2\autofmt@end{\mathbb{#1}\mathsf{#2}}
\def\autofmt@c#1#2\autofmt@end{\mathcal{#1}\mathit{#2}}
\def\autofmt@s#1#2\autofmt@end{\mathscr{#1}\mathit{#2}}
\def\autofmt@f#1\autofmt@end{\mathsf{#1}}
\def\autofmt@u#1\autofmt@end{\underline{\smash{\mathsf{#1}}}}
\def\autofmt@U#1\autofmt@end{\underline{\underline{\smash{\mathsf{#1}}}}}
\def\autofmt@h#1\autofmt@end{\widehat{#1}}
\def\autofmt@r#1\autofmt@end{\overline{#1}}
\def\autofmt@t#1\autofmt@end{\widetilde{#1}}
\def\autofmt@k#1\autofmt@end{\check{#1}}

\def\auto@drop#1{}
\def\autodef#1{\ea\ea\ea\@autodef\ea\ea\ea#1\ea\auto@drop\string#1\autodef@end}
\def\@autodef#1#2#3\autodef@end{%
  \ea\def\ea#1\ea{\ea\ensuremath\ea{\csname autofmt@#2\endcsname#3\autofmt@end}\xspace}}
\def\autodefs@end{blarg!}
\def\autodefs#1{\@autodefs#1\autodefs@end}
\def\@autodefs#1{\ifx#1\autodefs@end%
  \def\autodefs@next{}%
  \else%
  \def\autodefs@next{\autodef#1\@autodefs}%
  \fi\autodefs@next}

\DeclareSymbolFont{bbold}{U}{bbold}{m}{n}
\DeclareSymbolFontAlphabet{\mathbbb}{bbold}

\newcommand{\bbone}{\ensuremath{\mathbbb{1}}\xspace}

\mdef\delbar{\overline{\partial}}

\newcommand{\inv}{^{-1}}

\mdef\hf{\textstyle\frac12 }
\mdef\thrd{\textstyle\frac13 }
\mdef\qtr{\textstyle\frac14 }

\newcommand{\op}{^{\mathrm{op}}}

\SelectTips{cm}{}
\newdir{ >}{{}*!/-10pt/@{>}}    

\mdef\Id{\mathrm{Id}}
\mdef\id{\mathrm{id}}
\alwaysmath{ell}
\alwaysmath{infty}
\alwaysmath{odot}
\def\frc#1/#2.{\frac{#1}{#2}}   
\mdef\ten{\mathrel{\otimes}}

\mdef\sqten{\mathrel{\boxtimes}}

\DeclareRobustCommand\widecheck[1]{{\mathpalette\@widecheck{#1}}}
\def\@widecheck#1#2{%
    \setbox\z@\hbox{\m@th$#1#2$}%
    \setbox\tw@\hbox{\m@th$#1%
       \widehat{%
          \vrule\@width\z@\@height\ht\z@
          \vrule\@height\z@\@width\wd\z@}$}%
    \dp\tw@-\ht\z@
    \@tempdima\ht\z@ \advance\@tempdima2\ht\tw@ \divide\@tempdima\thr@@
    \setbox\tw@\hbox{%
       \raise\@tempdima\hbox{\scalebox{1}[-1]{\lower\@tempdima\box
\tw@}}}%
    {\ooalign{\box\tw@ \cr \box\z@}}}

\DeclareMathOperator\colim{colim}

\DeclareMathOperator\Ho{Ho}

\DeclareMathOperator\Hom{Hom}

\DeclareMathOperator\homn{hom}
\newcommand{\Rhom}{\mathbf{R}\!\homn}
\newcommand{\Lotimes}{\otimes^\mathbf{L}}

\newcommand{\too}[1][]{\ensuremath{\overset{#1}{\longrightarrow}}}
\newcommand{\ot}{\ensuremath{\leftarrow}}

\mdef\we{\overset{\sim}{\longrightarrow}}
\mdef\leftwe{\overset{\sim}{\longleftarrow}}

\let\xto\xrightarrow

\def\rightarrowtailfill@{\arrowfill@{\Yright\joinrel\relbar}\relbar\rightarrow}
\newcommand\xrightarrowtail[2][]{\ext@arrow 0055{\rightarrowtailfill@}{#1}{#2}}

\def\twoheadrightarrowfill@{\arrowfill@{\relbar\joinrel\relbar}\relbar\twoheadrightarrow}
\newcommand\xtwoheadrightarrow[2][]{\ext@arrow 0055{\twoheadrightarrowfill@}{#1}{#2}}

\def\slashedarrowfill@#1#2#3#4#5{%
  $\m@th\thickmuskip0mu\medmuskip\thickmuskip\thinmuskip\thickmuskip
   \relax#5#1\mkern-7mu%
   \cleaders\hbox{$#5\mkern-2mu#2\mkern-2mu$}\hfill
   \mathclap{#3}\mathclap{#2}%
   \cleaders\hbox{$#5\mkern-2mu#2\mkern-2mu$}\hfill
   \mkern-7mu#4$%
}
\def\rightslashedarrowfill@{%
  \slashedarrowfill@\relbar\relbar\mapstochar\rightarrow}
\newcommand\xslashedrightarrow[2][]{%
  \ext@arrow 0055{\rightslashedarrowfill@}{#1}{#2}}
\mdef\hto{\xslashedrightarrow{}}
\mdef\htoo{\xslashedrightarrow{\quad}}

\long\def\my@drawfill#1#2;{%
\@skipfalse
\fill[#1,draw=none] #2;
\@skiptrue
\draw[#1,fill=none] #2;
}
\newif\if@skip
\newcommand{\skipit}[1]{\if@skip\else#1\fi}
\newcommand{\drawfill}[1][]{\my@drawfill{#1}}

\newif\ifhyperref
\@ifpackageloaded{hyperref}{\hyperreftrue}{\hyperreffalse}
\iftac
  \let\your@state\state
  \def\state#1{\gdef\currthmtype{#1}\your@state{#1}}
  \let\your@staterm\staterm
  \def\staterm#1{\gdef\currthmtype{#1}\your@staterm{#1}}
  \let\defthm\newtheorem
  \def\currthmtype{}
  \ifhyperref
    \def\autoref#1{\ref*{label@name@#1}~\ref{#1}}
  \else
    \def\autoref#1{\ref{label@name@#1}~\ref{#1}}
  \fi
  \AtBeginDocument{%
    \let\old@label\label%
    \def\label#1{%
      {\let\your@currentlabel\@currentlabel%
        \edef\@currentlabel{\currthmtype}%
        \old@label{label@name@#1}}%
      \old@label{#1}}
  }
\else
  \ifhyperref
    \def\defthm#1#2{%
      \newtheorem{#1}{#2}[section]%
      \expandafter\def\csname #1autorefname\endcsname{#2}%
      \expandafter\let\csname c@#1\endcsname\c@thm}
  \else
    \def\defthm#1#2{\newtheorem{#1}[thm]{#2}}
    \ifx\SK@label\undefined\let\SK@label\label\fi
    \let\old@label\label
    \let\your@thm\@thm
    \def\@thm#1#2#3{\gdef\currthmtype{#3}\your@thm{#1}{#2}{#3}}
    \def\currthmtype{}
    \def\label#1{{\let\your@currentlabel\@currentlabel\def\@currentlabel%
        {\currthmtype~\your@currentlabel}%
        \SK@label{#1@}}\old@label{#1}}
    \def\autoref#1{\ref{#1@}}
  \fi
\fi

\newtheorem{thm}{Theorem}[section]

\defthm{cor}{Corollary}
\defthm{prop}{Proposition}
\defthm{lem}{Lemma}
\defthm{sch}{Scholium}
\defthm{assume}{Assumption}
\defthm{claim}{Claim}
\defthm{conj}{Conjecture}
\defthm{hyp}{Hypothesis}
\defthm{fact}{Fact}
\defthm{quest}{Question}
\iftac\theoremstyle{plain}\else\theoremstyle{definition}\fi
\defthm{defn}{Definition}
\defthm{notn}{Notation}
\iftac\theoremstyle{plain}\else\theoremstyle{remark}\fi
\defthm{rmk}{Remark}
\defthm{eg}{Example}
\defthm{egs}{Examples}
\defthm{ex}{Exercise}
\defthm{ceg}{Counterexample}
\defthm{warn}{Warning}

\def\thmqedhere{\expandafter\csname\csname @currenvir\endcsname @qed\endcsname}

\setitemize[1]{leftmargin=2em}
\setenumerate[1]{leftmargin=*}

\iftac
  \let\c@equation\c@subsection
\else
  \let\c@equation\c@thm
\fi
\numberwithin{equation}{section}

\@ifpackageloaded{mathtools}{\mathtoolsset{showonlyrefs,showmanualtags}}{}

\alwaysmath{alpha}
\alwaysmath{beta}
\alwaysmath{gamma}
\alwaysmath{Gamma}
\alwaysmath{delta}
\alwaysmath{Delta}
\alwaysmath{epsilon}
\mdef\ep{\varepsilon}
\alwaysmath{zeta}
\alwaysmath{eta}
\alwaysmath{theta}
\alwaysmath{Theta}
\alwaysmath{iota}
\alwaysmath{kappa}
\alwaysmath{lambda}
\alwaysmath{Lambda}
\alwaysmath{mu}
\alwaysmath{nu}
\alwaysmath{xi}
\alwaysmath{pi}
\alwaysmath{rho}
\alwaysmath{sigma}
\alwaysmath{Sigma}
\alwaysmath{tau}
\alwaysmath{upsilon}
\alwaysmath{Upsilon}
\alwaysmath{phi}
\alwaysmath{Pi}
\alwaysmath{Phi}
\mdef\ph{\varphi}
\alwaysmath{chi}
\alwaysmath{psi}
\alwaysmath{Psi}
\alwaysmath{omega}
\alwaysmath{Omega}

\makeatother
